\documentclass[onefignum,onetabnum]{siamart190516}

\usepackage{lipsum}
\usepackage{amsfonts}
\usepackage{graphicx}
\usepackage{epstopdf}
\usepackage{algorithmic}
\usepackage{mathtools}
\usepackage{amssymb}
\usepackage{bm}
\usepackage{stmaryrd}
\usepackage{tikz}
\usepackage{subcaption}
\usepackage{dsfont}
\usepackage{mathrsfs}
\usepackage{enumitem}
\usepackage{tabularx}
\ifpdf
  \DeclareGraphicsExtensions{.eps,.pdf,.png,.jpg}
\else
  \DeclareGraphicsExtensions{.eps}
\fi

\newsiamremark{remark}{Remark}
\newsiamremark{hypothesis}{Hypothesis}
\crefname{hypothesis}{Hypothesis}{Hypotheses}
\newsiamthm{claim}{Claim}
\newsiamremark{example}{Example}
\crefname{subsection}{section}{sections}  
\Crefname{subsection}{Section}{Sections}

\headers{Flow in Porous Media with Fractures}{S. Burbulla, M. Hörl, and C. Rohde}

\title{Flow in Porous Media with Fractures of Varying Aperture\thanks{ \textbf{Funding:}
{This work was funded by the Deutsche Forschungsgemeinschaft (DFG, German Research Foundation) – Project Number 327154368 – SFB 1313. 
C.R. acknowledges funding by the DFG under Germany’s Excellence Strategy – EXC 2075–390740016.}}}

\author{Samuel Burbulla\thanks{Institute of Applied Analysis and Numerical Simulation, University of Stuttgart, Pfaffenwaldring~57, D-70569 Stuttgart, Germany (\email{samuel.burbulla@ians.uni-stuttgart.de}, \email{maximilian.hoerl@ians.uni-stuttgart.de}, \email{christian.rohde@ians.uni-stuttgart.de}).}
\and Maximilian Hörl\footnotemark[2]
\and Christian Rohde\footnotemark[2]}

\usepackage{amsopn}

\let\originalleft\left
\let\originalright\right
\renewcommand{\left}{\mathopen{}\mathclose\bgroup\originalleft}
\renewcommand{\right}{\aftergroup\egroup\originalright}

\let\oldphi\phi
\let\phi\varphi
\let\varphi\oldphi

\let\oldrho\rho
\let\rho\varrho
\let\varrho\oldrho

\let\oldepsilon\epsilon
\let\epsilon\varepsilon
\let\varepsilon\oldepsilon

\let\oldtheta\theta
\let\theta\vartheta
\let\vartheta\oldtheta

\newcommand*{\matr}[1]{\mathbf{#1}}
\newcommand*{\vct}[1]{{\bm{#1}}}
\newcommand{\abs}[1]{{\left\vert #1 \right\vert}}
\newcommand{\norm}[1]{{\left\Vert #1 \right\Vert}}
\newcommand\restr[2]{{ \left.\kern-\nulldelimiterspace #1 \vphantom{\big|} \right|_{#2} }}
\newcommand*{\ldblbrace}{\{\mskip-5mu\{}
\newcommand*{\rdblbrace}{\}\mskip-5mu\}}
\newcommand*{\avg}[1]{\ldblbrace #1 \rdblbrace}
\newcommand*{\jump}[1]{\llbracket #1 \rrbracket}
\newcommand{\partialDer}[2]{\frac{\partial #1}{\partial #2}}
\newcommand{\totalDer}[2]{\frac{\mathrm{d} #1}{\mathrm{d} #2}}

\allowdisplaybreaks

\ifpdf
\hypersetup{
  pdftitle={Flow in Porous Media with Fractures of Varying Aperture},
  pdfauthor={S. Burbulla, M. Hörl, and C. Rohde}
}
\fi

\begin{document}

\maketitle

\begin{abstract}
We study single-phase flow in a fractured porous medium at a macroscopic scale that allows to model fractures individually. The flow is governed by Darcy's law in both fractures and porous matrix. We derive a new mixed-dimensional model, where fractures are represented by $(n-1)$-dimensional interfaces between $n$-dimensional subdomains for $n\ge 2$. 
In particular, we suggest a generalization of the model in~\cite{martin05} by accounting for asymmetric fractures with spatially varying aperture. 
Thus, the new model is particularly convenient for the description of surface roughness or for modeling curvilinear or winding fractures.
The wellposedness of the new model is proven under appropriate conditions.
Further, we formulate a discontinuous Galerkin discretization of the new model and validate the model by performing two- and three-dimensional numerical experiments.
\end{abstract}



\section{Introduction}
In many situations, the fluid flow in a porous medium is significantly influenced, if not dominated, by the presence of fractures. 
Fractures are characterized by an extreme geometry with a thin aperture but wide extent in the remaining directions of space.
Hence, they form narrow heterogeneities that pose a challenge for classical continuum modeling. 
We assume that fractures are filled by another porous medium (e.g., debris) whose hydraulic properties may differ considerably from those of the surrounding porous matrix.
Depending on the permeability inside a fracture, the fracture may serve as primary transport path for the overall flow in a porous medium or, conversely, may act as an almost impermeable barrier.

Fluid flow in fractured porous media is of essential relevance for a wide range of applications, such as geothermal energy, enhanced oil recovery, groundwater flow, nuclear waste disposal, and carbon sequestration.
In the last decades, a diversity of mathematical models for fluid flow in fractured porous media has been developed. 
For a current review of modeling and discretization approaches, we refer to \cite{berre19} and the literature therein. 
Besides, for a comparison of numerical methods and benchmark test cases for single-phase flow in fractured porous media, we refer to~\cite{berre21,flemisch18,fumagalli19}.

A common macroscopic modeling approach for flow in fractured porous media is to consider a reduced representation where fractures are described explicitly as $(n-1)$-dimensional interfaces between $n$-dimensional bulk domains \cite{ahmed17,angot09,boon18,formaggia14,frih08,fumagalli21,jaffre11,kumar20,lesinigo11,martin05,schwenck15}. 
In contrast to a full-dimensional representation of fractures, this avoids thin equi-dimensional subdomains which require highly resolved grids in numerical methods.
Typically, these kinds of interface models are based on the idealized conception of a planar fracture geometry with constant aperture.
However, in this case, the resulting reduced model cannot account for surface roughness and does not properly describe curvilinear or winding fractures.
Thus, more generally, the geometry of
a fracture may be described by spatially varying aperture functions, which is the approach that we follow here.

Specifically, we consider single-phase fluid flow in a fractured porous medium governed by Darcy's law. 
We suggest a new mixed-dimensional model that accounts for asymmetric fractures with spatially varying aperture and, thereby, propose an extension and alternative derivation of the model in the seminal work in~\cite{martin05} that was derived for fractures with constant aperture. 
For the derivation of the new fracture-averaged model, we proceed from a domain-decomposed system for Darcy's flow with full-dimensional fracture that closely resembles the initial model in~\cite{martin05}.
However, as a central issue when dealing with a spatially varying fracture aperture, the normal vectors at the internal boundaries of the initial full-dimensional fracture depend on aperture gradients and are generally not aligned with the normal vector of the interfacial fracture in the desired reduced model. 
In contrast to the approach in~\cite{martin05}, we address this problem by employing a weak formulation when averaging across the fracture and by approximating suitable curve integrals across the fracture to obtain internal boundary conditions.

\Cref{sec:sec2} introduces an initial model problem with full-dimensional fracture in a weak formulation.
In \Cref{sec:sec3}, we derive a new reduced model with interfacial fracture that accounts for a spatially varying aperture. 
The resulting reduced model is summarized and discussed in \Cref{sec:sec4}, where we also introduce four related model variants.
\Cref{sec:sec5} introduces a discontinuous Galerkin (DG) discretization of the new model. 
Numerical results are presented in \Cref{sec:sec6}.

\section{Darcy Flow with Full-Dimensional Fracture} \label{sec:sec2}
Let $\Omega\subset \mathbb{R}^n$ be a bounded Lipschitz domain that is occupied by an $n$-dimensional porous medium, where $n\in\mathbb{N}$ with $n\ge 2$. We suppose that the flow of a single-phase fluid in~$\Omega$ is governed by Darcy's law and mass conservation, i.e.,
\begin{subequations}
\begin{alignat}{2}
-\nabla \cdot \left( \matr{K} \nabla p \right) &= q \qquad &&\text{in } \Omega , \\
p &= 0 \qquad &&\text{on } \partial \Omega .
\end{alignat}%
\label{eq:darcy}%
\end{subequations}
In \cref{eq:darcy}, $p\colon \Omega \rightarrow \mathbb{R}$ is the pressure and $q\in L^2 (\Omega )$ denotes the source term. 
The permeability matrix is denoted by~$\matr{K}\in L^\infty (\Omega ; \mathbb{R}^{n\times n} )$ and required to be symmetric and uniformly elliptic. 
The choice of homogeneous Dirichlet boundary conditions is only made for the sake of simplicity.
A weak formulation of the Darcy system~\eqref{eq:darcy} is given by the following problem.
Find $p\in H^1_0\left(\Omega\right)$ such that
\begin{align}
\int_\Omega \matr{K} \nabla p \cdot \nabla \phi \,\mathrm{d} V &= \int_\Omega q \phi \,\mathrm{d} V \qquad \text{for all } \phi\in H^1_0\left(\Omega\right). \label{eq:weakdarcyinit}
\end{align}
 
We consider the case of a single fracture as an $n$-dimensional open subdomain~$\Omega_\mathrm{f} \subset \Omega$ crossing the entire domain~$\Omega$ such that $\Omega$ is cut into two disjoint connected subdomains~$\Omega_1$ and $\Omega_2$, i.e., $\Omega \setminus \overline{\Omega}_\mathrm{f} = \Omega_1\: \dot{\cup}\: \Omega_2$. 
Moreover, we suppose that the fracture domain~$\Omega_\mathrm{f}$ can be parameterized by a hyperplane~$\Gamma$ and two functions~$d_1 , d_2 \in W^{1,\infty } (\Gamma )$ which describe the aperture of the fracture on the left and right side of~$\Gamma$ such that
\begin{align}
\Omega_\mathrm{f} = \big\{ \vct{\gamma} + \lambda \vct{n} \in \Omega \ \big\vert\ \vct{\gamma} \in \Gamma ,\, \lambda \in \left( - d_1\left(\vct{\gamma} \right), d_2\left( \vct{\gamma} \right) \right) \big\} . \label{eq:Omegaf}
\end{align}
In \cref{eq:Omegaf}, $\vct{n}$ denotes the unit normal of the hyperplane~$\Gamma$ that points in the direction of~$\Omega_2$. 
Further, we write $d := d_1 + d_2 > 0$ for the total aperture of the fracture. 
We only require the total aperture~$d$ to be positive, not the functions~$d_1$ and $d_2$, so that the hyperplane~$\Gamma$ is not necessarily required to be fully immersed inside the fracture domain~$\Omega_\mathrm{f}$. 
This allows the description of curvilinear or winding fractures in a natural way.  
Besides, w.l.o.g., the hyperplane~$\Gamma$ is represented by
$\Gamma = \left\{ \vct{x} \in \Omega \ \middle\vert\ \vct{n}  \cdot \vct{x} = 0\right\}$. 
In addition, we denote by $\rho_i := \partial\Omega_i  \cap  \partial\Omega$ the exterior boundary of the overall domain~$\Omega$ inside~$\Omega_i$ for $ i \in \left\{ 1,2,\mathrm{f} \right\}$ and by $\Gamma_i := \partial\Omega_i \cap \partial{\Omega_\mathrm{f}}$ the interface between the bulk domain~$\Omega_i$ and the fracture domain~$\Omega_\mathrm{f}$ for $i\in\left\{ 1,2\right\}$.
\begin{figure}[tb]
\centering
\includegraphics[width=0.6\textwidth]{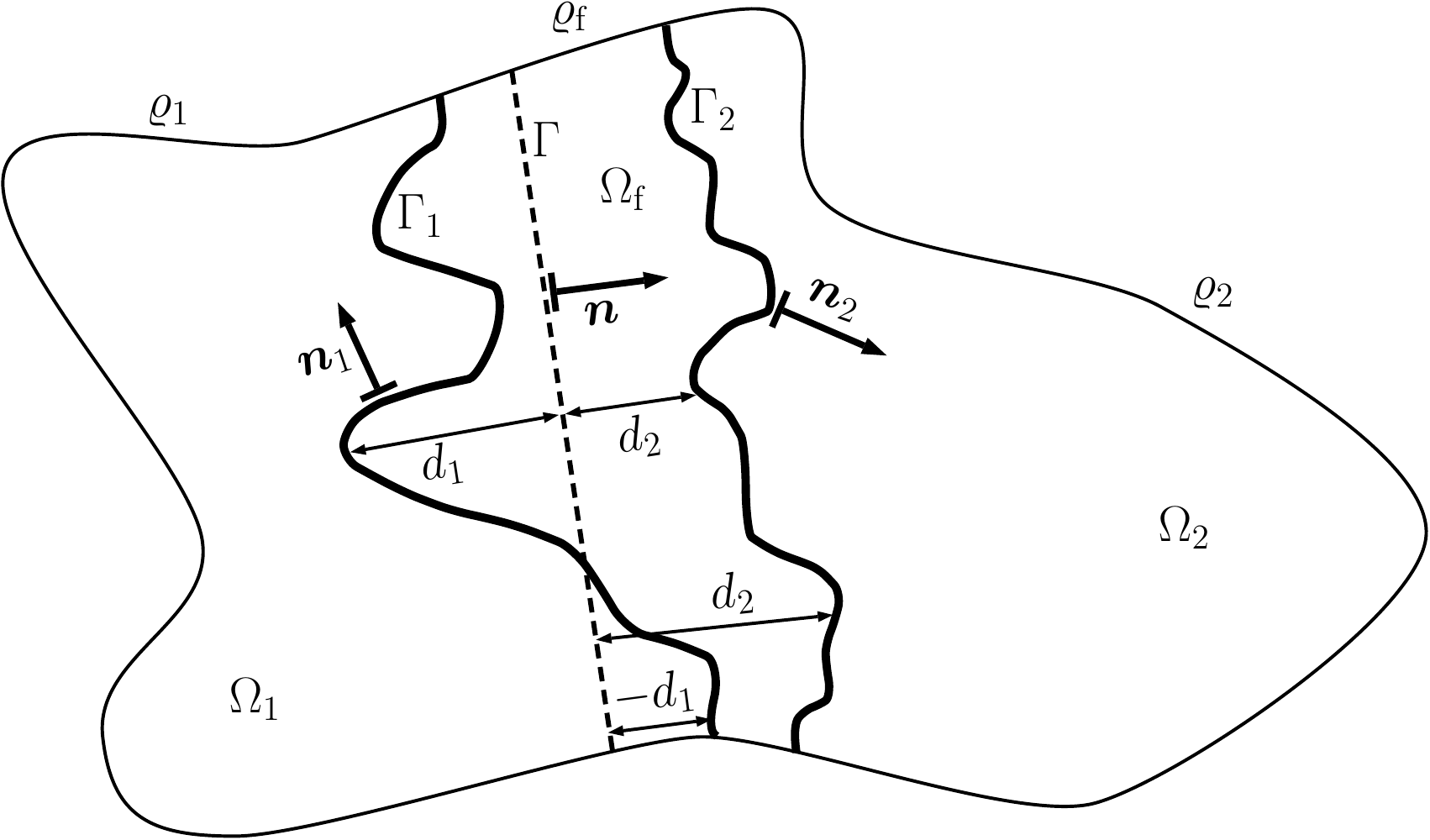}
\caption{Sketch of the geometry for the full-dimensional Darcy problem~\eqref{eq:darcydecomposed}.}
\label{fig:Omega_f}
\end{figure}
The specified geometric situation is sketched in \Cref{fig:Omega_f}. 
We remark that, for the representation of the fracture in \cref{eq:Omegaf}, it is required that the connecting lines between the interfaces~$\Gamma_1$ and $\Gamma_2$ along~$\vct{n}$ exist and are contained in the fracture domain~$\Omega_\mathrm{f}$. 

Next, for $i\in\left\{ 1,2,\mathrm{f}\right\}$ and functions~$f\colon \Omega \rightarrow B$ and $\bar{f}\colon \partial \Omega \rightarrow B$ with arbitrary codomain~$B$, we introduce the notation 
\begin{align}
f_i := \restr{f}{\Omega_i} \quad \text{and} \quad \bar{f}_i := \restr{\bar{f}}{\rho_i} .
\end{align}
Following a domain decomposition approach, this allows us to reformulate~\eqref{eq:darcy} as 
\begin{subequations}
\begin{alignat}{2}
-\nabla\cdot \left( \matr{K}_i\nabla p_i  \right) &= q_i \qquad&\text{in } \Omega_i , \quad &i\in\left\{1,2,\mathrm{f}\right\} ,\\
\label{eq:dirichletdd} p_i &= 0 \qquad&\text{on } \rho_i , \quad &i\in\left\{1,2,\mathrm{f}\right\} , \\
p_i &= p_\mathrm{f} \qquad&\text{on } \Gamma_i , \quad &i\in\left\{1,2\right\} ,\label{eq:pcontinuity1}\\
\matr{K}_i\nabla p_i \cdot \vct{n}_i &= \matr{K}_\mathrm{f}\nabla p_\mathrm{f}\cdot \vct{n}_i \qquad&\text{on } \Gamma_i , \quad &i\in\left\{1,2\right\}. \label{eq:ucontinuity}
\end{alignat}%
\label{eq:darcydecomposed}%
\end{subequations}
Here, for $i\in\left\{1,2\right\}$, we denote by $\vct{n}_i$ the unit normal to the interface~$\Gamma_i$ that points into the bulk domain~$\Omega_i$. 
The vectors~$\vct{n}_1$ and $\vct{n}_2$ are given by
\begin{align}
\vct{n}_1 = \frac{ -\vct{n} - \nabla d_1 }{\sqrt{1 + \abs{\nabla d_1}^2}} , \qquad   \vct{n}_2 = \frac{\vct{n} - \nabla d_2}{\sqrt{1 + \abs{\nabla d_2}^2}}. \label{eq:normals}
\end{align} 

Next, we set up a weak formulation for the domain-decomposed Darcy problem in~\cref{eq:darcydecomposed}.
For $i \in\{ 1, 2, \mathrm{f} \}$, we define the space~$V_i$ by
\begin{align}
V_i &:= \left\{ p_i \in H^1_{0,\rho_i} (\Omega_i )   \ \middle\vert \  \matr{K}_i \nabla p_i \in H_\mathrm{div} (\Omega_i )  \right\} .
\end{align}
Further, with $H_\mathrm{div} ( \Omega_i ) := \{ \vct{w} \in L^2 (\Omega_i ; \mathbb{R}^n ) \, \vert \, \nabla \cdot \vct{w} \in L^2 (\Omega_i ) \}$ and $H^1_{0,\rho_i } (\Omega_i ) :=  \{ f \in H^1 (\Omega_i ) \, \vert\, f = 0 \text{ on } \rho_i \}$ for~$i\in\{ 1,2,\mathrm{f}\}$, we define the domain-decomposed spaces
\begin{align}
\begin{split}
V_\mathrm{dd} &:= \big\{ \left( p_1 , p_2 , p_\mathrm{f} \right) \in \bigtimes\nolimits_{i=1,2,\mathrm{f}} V_i  \ \big\vert \ p_i = p_\mathrm{f} \text{ a.e. on } \Gamma_i ,\; i\in\left\{ 1,2\right\} \big\} , 
\end{split} \label{eq:Vdd} \\
\Phi_\mathrm{dd} &:= \bigtimes\nolimits_{i=1,2,\mathrm{f}} H^1_{0,\rho_i}\left( \Omega_i \right) .
\end{align}
Then, a weak formulation of the domain-decomposed Darcy problem~\eqref{eq:darcydecomposed} is given by the following problem. 
Find $p = \left( p_1 , p_2 , p_\mathrm{f}\right)\in V_\mathrm{dd}$ such that
\begin{align}
\mathcal{B}_\mathrm{dd} \left( p , \phi \right) &= \mathcal{L}_\mathrm{dd} \left( \phi\right) \qquad \text{for all } \phi = \left( \phi_1 , \phi_2 , \phi_\mathrm{f} \right) \in \Phi_\mathrm{dd} . \label{eq:weakdarcy}
\end{align}
In \cref{eq:weakdarcy}, given $p = \left( p_1 , p_2 , p_\mathrm{f}\right) \in V_\mathrm{dd}$ and $\phi = \left( \phi_1 , \phi_2 , \phi_\mathrm{f} \right) \in \Phi_\mathrm{dd}$, the bilinear form $\mathcal{B}_\mathrm{dd} \colon V_\mathrm{dd} \times \Phi_\mathrm{dd} \rightarrow \mathbb{R}$ and the linear form~$\mathcal{L}_\mathrm{dd} \colon \Phi_\mathrm{dd} \rightarrow \mathbb{R}$ are defined by
\begin{subequations}
\begin{align}
\begin{split}
\label{eq:B} \mathcal{B}_\mathrm{dd} \left( p , \phi \right) &:= \smash{\sum_{i=1,2,\mathrm{f}}} \int_{\Omega_i} \matr{K}_i \nabla p_i \cdot \nabla \phi_i \, \mathrm{d} V   - \sum_{i=1,2} \int_{\Gamma_i} \big[ \phi_\mathrm{f} - \phi_i \big] \matr{K}_i \nabla p_i \cdot \vct{n}_i \,\mathrm{d}\sigma , 
\end{split} \\
\mathcal{L}_\mathrm{dd} \left( \phi \right) &:= \sum_{i=1,2,\mathrm{f}} \int_{\Omega_i} q_i \phi_i \,\mathrm{d} V . \label{eq:L}
\end{align}
\end{subequations} 
Moreover, we state the following result, which is often used in the context of domain decomposition methods. For a proof, we refer to~\cite{hoerl22}.
\begin{theorem}
\begin{enumerate}[label=(\roman*)]
\item Let $p\in H_0^1(\Omega )$ be a weak solution of the single-domain Darcy problem~\eqref{eq:weakdarcyinit}. 
Then, $\left( p_1, p_2 , p_\mathrm{f}\right) \in V_\mathrm{dd}$ is a weak solution of the do\-main-de\-com\-posed Darcy problem~\eqref{eq:weakdarcy}.
\item Conversely, let $\big( p^{(1)}, p^{(2)} , p^{(\mathrm{f})}\big) \in V_\mathrm{dd}$ be a weak solution of the do\-main-de\-com\-posed Darcy problem~\eqref{eq:weakdarcy}. 
Then, the function $p \in H^1_0\left(\Omega \right)$ defined by 
\begin{align}
p \left( \vct{x} \right) &:= \begin{cases} 
p^{(i)}\left(\vct{x}\right) &\text{if } \vct{x} \in\Omega_i , \enspace i\in \{ 1,2,\mathrm{f} \} \\
 \end{cases}%
 \label{eq:pdecomposed}
\end{align}
is a weak solution of the single-domain Darcy problem~\eqref{eq:weakdarcyinit}.
\end{enumerate}
In particular, the problem~\eqref{eq:weakdarcy} has a unique weak solution~$\left( p_1, p_2 , p_\mathrm{f}\right) \in V_\mathrm{dd}$.
\end{theorem}

\section{Derivation of a Reduced Model} \label{sec:sec3}
Proceeding from the weak domain-decomposed Darcy problem~\eqref{eq:weakdarcy}, we will derive a new reduced model in which the fracture cutting the domain~$\Omega$ is solely described by the hyperplane~$\Gamma$ between the two bulk subdomains~$\Omega_1$ and $\Omega_2$. 
The new reduced model~\eqref{eq:weak} is obtained by introducing fracture-averaged effective quantities, splitting integrals over~$\Omega_\mathrm{f}$ into a surface integral over $\Gamma$ and a line integral in normal direction, and transforming integrals over the interfaces~$\Gamma_i$ into integrals over~$\Gamma$. 
In addition, two boundary conditions on~$\Gamma$ are found by approximating curve integrals across the fracture domain~$\Omega_\mathrm{f}$ using a quadrature rule and polynomial interpolation. 
Besides, the following result on the weak differentiation of parameter integrals will be useful.
\begin{lemma}[Weak Leibniz Rule] \label{lem:leibniz}
For $i\in\left\{ 1,2\right\}$, let $a_i, b_i \in \mathbb{R}$ with $a_i < b_i$ and $I_i := \left( a_i , b_i \right)$. Further, let $f \in W^{1,p} \left( I_1 \times I_2 \right)$ and $\chi , \psi \in W^{1,\infty } \left( I_2 \right) \cap W^{2,q} \left( I_2 \right)$ such that $\chi\left( I_2\right) , \psi \left(I_2\right) \subset I_1$ (except for null sets), where $p, q \in (1, \infty )$ with $\frac{1}{p} + \frac{1}{q} = 1$. Then, the function 
$\smash{I_2 \rightarrow \mathbb{R}, \ y \mapsto \int_{\chi \left( y \right)}^{\psi \left( y \right)} \! f\left( x , y\right) \,\mathrm{d} x}$ 
belongs to $W^{1,1} \left( I_2 \right)$ and the relation
\begin{subequations}
\begin{align}
\begin{split}
&\totalDer{}{y} \int_{\chi \left( y \right)}^{\psi \left( y \right)} \!  f\left( x , y\right) \,\mathrm{d} x  = \int_{\chi \left( y \right)}^{\psi \left( y \right)} \! \partialDer{}{y} f\left( x , y\right) \,\mathrm{d} x \, +\, f\left( \psi \left( y \right) , y \right) \psi^\prime \left( y\right) \, - \, f\left( \chi \left( y \right) , y \right) \chi^\prime \left( y\right)
\end{split} 
\end{align}
holds in a weak sense, i.e., for every test function~$\varphi \in \mathcal{C}^\infty_\mathrm{c} \left( I_2 \right)$, we have  
\begin{align}
\begin{split}
&-\int_{a_2}^{b_2} \left[ \int_{\chi \left( y \right)}^{\psi \left( y \right)} \! f\left( x , y\right) \,\mathrm{d} x \right] \varphi^\prime \left( y \right) \,\mathrm{d} y  = \int_{a_2}^{b_2} \left[ \int_{\chi \left( y \right)}^{\psi \left( y \right)} \! \partialDer{}{y} f\left( x , y\right) \,\mathrm{d} x  \right] \varphi\left( y\right) \,\mathrm{d}y \\
&\hspace{3.9cm} + \int_{a_2}^{b_2} \Big[ f\left( \psi \left( y \right) , y \right) \psi^\prime \left( y\right) \, - \, f\left( \chi \left( y \right) , y \right) \chi^\prime \left( y\right) \Big] \varphi\left( y\right) \,\mathrm{d}y .
\end{split} \label{eq:weakleibniz}
\end{align}%
\label{eq:leibniz}%
\end{subequations}
\end{lemma}
\begin{proof}
Approximating $f$, $\chi$, and~$\psi$ by smooth functions, the result can be traced back to the classical Leibniz rule.
\end{proof}

\subsection{Geometrical Setting and Notations}
We start by expanding the normal vector~$\vct{n}$ of the hyperplane~$\Gamma$ to an orthonormal basis $\mathcal{N} := \left( \vct{n} , \vct{\tau}_1 , \dots , \vct{\tau}_{n-1} \right)$ of the space~$\mathbb{R}^n$. 
Then, we can decompose the position vector~$\vct{x}$ as
\begin{align}
\vct{x} = \eta \vct{n} + \sum_{i=1}^{n-1} t_i \vct{\tau}_i =: \left( \eta , t_1 , \dots , t_{n-1} \right)_\mathcal{N}^T =  \big( \eta , \vct{t}^T \big)^{\! T}_{\! \mathcal{N}} . 
\end{align}
For a function~$f\colon  A \rightarrow B$ with~$A\subset\Omega$, we introduce the notation
\begin{align}
f \left( \vct{x} \right) =: f \left( \eta , t_1 , \dots , t_{n-1} \right) = f\left(\eta,  \vct{t}\right)  .\label{eq:basisnotation}
\end{align}
If $f$ in \cref{eq:basisnotation} is a function with domain~$A\subset \Gamma$, we usually omit the first argument and write $f \left(\vct{x} \right) = f (0, \vct{t} ) =: f (\vct{t} )$.
Besides, we write $D := \smash{\norm{d}_{L^\infty (\Gamma )}}$ for the maximum aperture of the fracture. 

Next, we introduce the parameterizations~$\vct{\pi}_1$, $\vct{\pi}_2$ of the interfaces~$\Gamma_1$, $\Gamma_2$ given by
\begin{align}
\label{eq:parameterization} \vct{\pi}_1 \left( \vct{t}\right) := \big( \! -\! d_1 (\vct{t} ) , \vct{t}^T \big)^{\! T}_{\! \mathcal{N}} , \qquad \vct{\pi}_2 \left( \vct{t} \right) := \big( d_2 (\vct{t} ) , \vct{t}^T \big)^{\! T}_{\! \mathcal{N}}.
\end{align}
In addition, for $i\in\left\{ 1,2\right\}$ and any function~$f^i\colon A_i \rightarrow B_i$ with arbitrary co\-do\-main~$B_i$,  domain~$A_i$ such that $\Gamma_i \subset \overline{A}_i$, and well-defined trace on~$\Gamma_i$, we introduce the notation%
\begin{subequations}
\begin{align}
\restr{f^1}{\Gamma_1} \left( \vct{t}\right) &= \restr{f^1}{\Gamma_1} \left( t_1 , \dots , t_{n-1}\right) := f^1\left( -d_1\left( \vct{t}\right), \vct{t} \right) , \\
\restr{f^2}{\Gamma_2} \left( \vct{t}\right) &= \restr{f^2}{\Gamma_2} \left( t_1 , \dots , t_{n-1}\right) := f^2\left( d_2\left( \vct{t}\right), \vct{t} \right)
\end{align}%
\label{eq:interfacerestr}%
\end{subequations}
for the trace on the interfaces~$\Gamma_1$ and $\Gamma_2$. 
Then, for $i\in\left\{ 1,2\right\}$, a surface integral on~$\Gamma_i$ can be transformed into an integral on~$\Gamma$ according to the relation
\begin{align}
\int_{\Gamma_i} f^i \,\mathrm{d}\sigma = \int_\Gamma \, \restr{f^i}{\Gamma_i} \sqrt{ 1 + \abs{\nabla d_i}^2} \,\mathrm{d}\sigma . \label{eq:Gammatransform}
\end{align} 
Moreover, we introduce jump and average operators across~$\Gamma$. 
\begin{definition}[Jump and average operators] \label{def:jumpavg}
Let $f \colon A \rightarrow \mathbb{R}$, $\vct{F}\colon A \rightarrow \mathbb{R}^n$ be functions with domain~$A\subset \Omega$ and a well-defined trace on the interfaces~$\Gamma_1$ and $\Gamma_2$. Then, for $\smash{( 0 , \vct{t}^T )_\mathcal{N}^T} \in \Gamma$, using the notation from \cref{eq:interfacerestr}, we denote by 
\begin{subequations}
\begin{align}
\jump{f}\left(\vct{t}\right) &:= \restr{f}{\Gamma_2} \!\left(\vct{t}\right) - \restr{f}{\Gamma_1} \!\left(\vct{t}\right) , \\
\begin{split}
\jump{\vct{F}}\left(\vct{t}\right) &:= \restr{\vct{F}}{\Gamma_1} \! \left(\vct{t}\right)  \cdot \left[ \vct{n} + \nabla d_1\left( \vct{t}\right) \right] - \restr{\vct{F}}{\Gamma_2} \! \left(\vct{t}\right)  \cdot \left[ \vct{n} - \nabla d_2 \left(\vct{t}\right) \right] 
\end{split}
\end{align}
\end{subequations}
the jump operators of $f$ and $\vct{F}$ across~$\Gamma$. In addition, we define by
\begin{subequations}
\begin{align}
\avg{f}\left(\vct{t}\right) &:= \frac{1}{2} \left( \restr{f}{\Gamma_1} \!\left(\vct{t}\right) + \restr{f}{\Gamma_2} \!\left(\vct{t}\right) \right), \\
\begin{split}
\avg{\vct{F}}\left(\vct{t}\right) &:= \frac{1}{2} \left( \restr{\vct{F}}{\Gamma_1} \! \left(\vct{t}\right)  \cdot \left[ \vct{n} + \nabla d_1\left( \vct{t}\right) \right] + \restr{\vct{F}}{\Gamma_2} \! \left(\vct{t}\right)  \cdot \left[ \vct{n} - \nabla d_2 \left(\vct{t}\right)  \right] \right)
\end{split}
\end{align}
\end{subequations}
the average operators of $f$ and $\vct{F}$ across~$\Gamma$.
\end{definition}
We note that, for vector-valued functions, the jump and average operators in \Cref{def:jumpavg} explicitly depend on the geometry of the fracture since they involve gradients of the aperture functions~$d_1$ and~$d_2$.  

It is evident that a reduced model cannot capture all information from the full-dimensional model~\eqref{eq:weakdarcy}. 
In particular, inside the fracture domain~$\Omega_\mathrm{f}$, we will restrict ourselves to a subspace of test functions
\begin{align}
\Phi_\mathrm{f} := \left\{ \phi_\mathrm{f} \in H^1_{0,\rho_\mathrm{f} } \left( \Omega_\mathrm{f}\right) \ \middle\vert \ \partial_\eta \phi_\mathrm{f} = 0   \right\} \subsetneq H^1_{0,\rho_\mathrm{f}} \left( \Omega_\mathrm{f}\right) ,\label{eq:Phif}
\end{align}
i.e., we will only consider test functions~$\phi_\mathrm{f} \in \Phi_\mathrm{f}$ that are invariant in perpendicular direction to~$\Gamma$. Moreover, we introduce new reduced quantities on the interface~$\Gamma$ that are obtained by averaging along straight lines perpendicular to~$\Gamma$ in~$\Omega_\mathrm{f}$. Specifically, for $\smash{( 0 , \vct{t}^T)_\mathcal{N}^T \in \Gamma}$, we define the average pressure~$p_\Gamma$ inside the fracture, the total source term~$q_\Gamma$, and, for any test function~$\phi_\mathrm{f} \in \Phi_\mathrm{f}$, the averaged test function~$\phi_\Gamma$ by%
\begin{subequations}
\begin{align}
p_\Gamma \left( \vct{t}\right) &:= \frac{1}{d\left(\vct{t}\right)} \int_{-d_1\left(\vct{t}\right)}^{d_2\left(\vct{t}\right)} p_\mathrm{f} \left(\eta , \vct{t}\right) \,\mathrm{d}\eta , \label{eq:pGamma}\\
q_\Gamma \left( \vct{t} \right) &:= \int_{-d_1\left(\vct{t}\right)}^{d_2\left(\vct{t}\right)} q_\mathrm{f} \left(\eta , \vct{t}\right) \, \mathrm{d}\eta , \label{eq:qGamma} \\
\phi_\Gamma \left(\vct{t} \right) &:= \frac{1}{d\left(\vct{t}\right)} \smash{\int_{-d_1\left(\vct{t}\right)}^{d_2\left(\vct{t}\right)}} \phi_\mathrm{f} \left(\eta , \vct{t}\right) \,\mathrm{d}\eta . \label{eq:phiGamma}
\end{align}%
\label{eq:effective}%
\end{subequations} 
Then, due to the definition of the space~$\Phi_\mathrm{f}$ in \cref{eq:Phif}, we have $\phi_\Gamma \left( \vct{t}\right) = \phi_\mathrm{f}\left( \eta , \vct{t}\right)$ for a.a.~$\smash{ ( \eta , \vct{t}^T )^T_\mathcal{N}} \in \Omega_\mathrm{f}$.
Further, we define the effective permeability~$\matr{K}_\Gamma$ of the fracture as the mean of $\matr{K}_\mathrm{f}$ in normal direction, i.e.,
\begin{subequations} 
\begin{align}
\matr{K}_\Gamma \left( \vct{t}\right) :=  \frac{1}{d\left(\vct{t}\right)} \int_{-d_1\left(\vct{t}\right)}^{d_2\left(\vct{t}\right)} \matr{K}_\mathrm{f} \left( \eta , \vct{t}\right) \,\mathrm{d}\eta .
\end{align}%
Then, for $\smash{( \eta , \vct{t}^T )_\mathcal{N}^T \in \Omega_\mathrm{f}}$, we have
\begin{align}
\matr{K}_\mathrm{f}\left(\eta , \vct{t} \right) = \matr{K}_\Gamma\left(\vct{t}\right) + \mathcal{O}\left( d \left(\vct{t}\right)\right) 
\end{align}
if $\matr{K}_\mathrm{f} $ is continuously differentiable with respect to~$\eta$. 
\label{eq:KGamma}%
\end{subequations}
We assume that $\matr{K}_\Gamma$ is block-diagonal with respect to the basis~$\mathcal{N}$, i.e., 
\begin{align}
\matr{K}_\Gamma^\mathcal{N} = 
\left[ 
\begin{array}{c|c} 
  K_{\Gamma , \vct{n} } & \vct{0} \\ 
  \hline 
  \vct{0} & \matr{K}_{\Gamma , \vct{t}} 
\end{array} 
\right] ,
\end{align}
where $K_{\Gamma , \vct{n} } \colon \Omega_\mathrm{f} \rightarrow \mathbb{R}$ and $\matr{K}_{\Gamma , \vct{t}} \colon \Omega_\mathrm{f} \rightarrow \mathbb{R}^{(n-1) \times (n-1)}$.
Moreover, for any $\smash{( 0 , \vct{t}^T)^T_{\mathcal{N}}} \in \Gamma$, we denote by ${\vct{r}_\vct{t}\colon \left( -d_1\left(\vct{t}\right) , d_2\left(\vct{t}\right)\right) \rightarrow \mathbb{R}^n}$ a continuously differentiable path such that%
\begin{subequations}
\begin{alignat}{2}
\vct{r}_\vct{t} \left( - d_1 \left( \vct{t} \right)\right) &= \big(\! -\! d_1 ( \vct{t} ) ,\, \vct{t}^T \big)^{\! T}_{\!\mathcal{N} }, \qquad
 &\vct{r}_\vct{t} \left( d_2 \left( \vct{t} \right)\right) &= \big( d_2 ( \vct{t} ) ,\, \vct{t}^T \big)^{\! T}_{\!\mathcal{N}} ,\\
\dot{\vct{r}}_\vct{t} \left( - d_1 \left( \vct{t} \right)\right) &= -\vct{n}_1 \left( \vct{t} \right)\sqrt{1 + \abs{\nabla d_1 \left(\vct{t}\right)}^2} , \qquad
&\dot{\vct{r}}_\vct{t} \left( d_2 \left( \vct{t} \right)\right) &= \vct{n}_2\left(\vct{t}\right)\sqrt{1 + \abs{\nabla d_2\left(\vct{t}\right)}^2} .
\end{alignat}%
\label{eq:rt}%
\end{subequations}
Additionally, we assume that $\dot{\vct{r}}_\vct{t}\left( s\right) \! = \! \totalDer{}{s} \vct{r}_\vct{t} \left( s \right)$ is an eigenvector of the permeability~$\matr{K}_\mathrm{f}\left( \vct{r}_\vct{t} \left(  s\right)\right)$ with eigenvalue $K_\mathrm{f}^\perp \left( \vct{r}_\vct{t} \left( s\right)\right)$ for a.a.~$s \in \left( -d_1\left(\vct{t}\right) , d_2\left(\vct{t}\right)\right)$. 
Then, we can define the effective permeability~$K_\Gamma^\perp$ in normal direction as the mean value
\begin{subequations}
\begin{align}
K_\Gamma^\perp \left( \vct{t}\right) := \frac{1}{L\left(\vct{r}_\vct{t}\right)} \int_{\vct{r}_\vct{t}} K_\mathrm{f}^\perp \,\mathrm{d} r ,
\end{align}
where $L\left(\vct{r}_\vct{t}\right)$ denotes the arc length of the path $\vct{r}_\vct{t}$. If $K_\mathrm{f}^\perp \circ \vct{r}_\vct{t}$ is continuously differentiable, we have
\begin{align}
K_\mathrm{f}^\perp \left( \vct{r}_\vct{t} \left( s \right)\right) = K_\Gamma^\perp \left( \vct{t}\right) + \mathcal{O} \left( d\left( \vct{t}\right)\right)   \label{eq:KperpApprox}
\end{align}  \label{eq:Kperp}
for $s \in \left( -d_1\left(\vct{t}\right) , d_2\left(\vct{t}\right)\right)$.
\end{subequations}

\subsection{Averaging Across the Fracture}
In the following, proceeding from the weak formulation~\eqref{eq:weakdarcy} of the domain-decomposed Darcy problem~\eqref{eq:darcydecomposed}, we derive a relation that governs the effective pressure~$p_\Gamma$ inside the reduced fracture~$\Gamma$.

Let $\phi_\mathrm{f} \in \Phi_\mathrm{f}$. By splitting the integral over~$\Omega_\mathrm{f}$ in \cref{eq:L} into an integral over~$\Gamma$ and a line integral in normal direction, we obtain  
\begin{align}
\int_{\Omega_\mathrm{f}} q_\mathrm{f} \phi_\mathrm{f} \,\mathrm{d} V &= \int_\Gamma \phi_\mathrm{f} \int_{-d_1\left(\vct{t}\right)}^{d_2\left(\vct{t}\right)} q_\mathrm{f} \, \mathrm{d}\eta \mathrm{d}\vct{t} = \int_\Gamma  q_\Gamma \phi_\Gamma \,\mathrm{d}\sigma  .
\end{align}
Likewise, splitting the integral over~$\Omega_\mathrm{f}$ in \cref{eq:B} and using \cref{eq:KGamma} results in
\begin{align}
\begin{split}
&\int_{\Omega_\mathrm{f}} \matr{K}_{\mathrm{f}} \nabla p_{\mathrm{f}} \cdot \nabla \phi_{\mathrm{f}} \,\mathrm{d} V = \int_\Gamma \nabla \phi_\mathrm{f} \cdot \! \int_{-d_1\left(\vct{t}\right)} ^{d_2\left(\vct{t}\right)} \matr{K}_\mathrm{f} \nabla p_\mathrm{f} \,\mathrm{d}\eta \mathrm{d}\vct{t} \\
&\hspace{2.25cm} =  \int_\Gamma \matr{K}_\Gamma \nabla   \phi_\mathrm{f} \cdot\! \int_{-d_1\left(\vct{t}\right)}^{d_2\left(\vct{t}\right)} \nabla p_\mathrm{f} \,\mathrm{d}\eta \mathrm{d}\vct{t} \, +\, \mathcal{O}\left( D\right)\\
&\hspace{2.25cm}  = \int_\Gamma  \matr{K}_\Gamma \nabla \phi_\Gamma \cdot \left[ \nabla \left( dp_\Gamma \right) - \restr{p_1}{\Gamma_1} \!\nabla d_1 - \restr{p_2}{\Gamma_2} \!\nabla d_2 \right] \,\mathrm{d}\sigma\, + \, \mathcal{O}\left( D\right) .
\end{split}%
\label{eq:splitting}%
\end{align}
Here, we have used \Cref{lem:leibniz} given the assumption that $d_1, d_2 \in W^{1,\infty } (\Gamma ) \cap H^2 (\Gamma )$. 
Besides, we have utilized the continuity condition for the pressure from the definition of the space~$V_\mathrm{dd}$ in \cref{eq:Vdd}.
We remark that the calculation in \cref{eq:splitting} is exact if the permeability~$\matr{K}_\mathrm{f}$ is constant along~$\vct{n}$, i.e., in perpendicular direction to~$\Gamma$.

Further, by transforming the integrals over the interfaces~$\Gamma_i$ in~\cref{eq:weakdarcy}, $i\in\left\{1,2\right\}$, into integrals over~$\Gamma$ according to \cref{eq:Gammatransform}, one finds
\begin{align}
\begin{split}
\sum_{i=1,2}\int_{\Gamma_i}\! \big[ \phi_\mathrm{f} - \phi_i \big] \matr{K}_i \nabla p_i \cdot \vct{n}_i \,\mathrm{d} \sigma  = \! \sum_{i=1,2} \int_\Gamma \left[ \phi_\Gamma  - \restr{\phi_i}{\Gamma_i}\right] \restr{\left( \matr{K}_i \nabla p_i\right)}{\Gamma_i}\! \cdot \vct{n}_i \sqrt{1 + \abs{\nabla d_i}^2} \,\mathrm{d}\sigma .  \label{eq:coupling}
\end{split}
\end{align}

Thus, in summary, the weak formulation of the reduced model up until now reads as follows.
Find $p = \left( p_1 , p_2 , p_\Gamma\right)$ such that
\begin{align}
\begin{split}
& \sum_{i = 1,2} \int_{\Omega_i} \matr{K}_i \nabla p_i \cdot \nabla \phi_i \,\mathrm{d}V + \int_\Gamma \matr{K}_\Gamma \nabla \phi_\Gamma \cdot \Big[ \nabla \left( dp_\Gamma \right) - \sum_{i=1,2} \restr{p_i}{\Gamma_i} \nabla d_i \Big] \,\mathrm{d}\sigma \\
&- \smash{\sum_{i=1,2}} \int_\Gamma \left[ \phi_\Gamma  - \restr{\phi_i}{\Gamma_i}\right] \restr{\left( \matr{K}_i \nabla p_i\right)}{\Gamma_i}\! \cdot \vct{n}_i \sqrt{1 + \abs{\nabla d_i}^2} \,\mathrm{d}\sigma \\
&\hspace{6.7cm} = \sum_{i=1,2} \int_{\Omega_i} q_i \phi_i \,\mathrm{d} V + \int_\Gamma q_\Gamma \phi_\Gamma \,\mathrm{d}\sigma
\end{split}%
\label{eq:weaksummary}%
\end{align}%
holds for all test functions $\phi = \left( \phi_1 , \phi_2 , \phi_\Gamma \right)$. 

Moreover, the weak problem in \cref{eq:weaksummary} corresponds to the following strong formulation.
Find $p = \left( p_1, p_2, p_\Gamma \right)$ such that
\begin{subequations}
\begin{alignat}{3}
-\nabla\cdot \left( \matr{K}_i\nabla p_i  \right) &= q_i \qquad &&\text{in } \Omega_i , \quad &&i\in\left\{1,2\right\} , \label{eq:bulk1}\\
-\nabla \cdot \Big[ \matr{K}_\Gamma \Big( \nabla \left( dp_\Gamma \right) -\! \smash{\sum_{i=1,2}} \restr{p_i}{\Gamma_i} \!\nabla d_i  \Big)  \Big] &= q_\Gamma -\jump{\matr{K}\nabla p} \enspace\quad && \text{in } \Gamma , \label{eq:Gamma1}\\
p_i &= 0 \qquad &&\text{on } \rho_i , \quad &&i\in\left\{1,2\right\} ,\label{eq:bulk2} \\
p_\Gamma &= 0 \qquad &&\text{on } \partial\Gamma  .\label{eq:Gamma2}
\end{alignat}%
\label{eq:decoupled}%
\end{subequations}%
We observe that the system in \cref{eq:decoupled} is decoupled.
Given a solution~$(p_1,p_2)$ of the bulk problem~\eqref{eq:bulk1},~\eqref{eq:bulk2}, which, so far, is independent of~$p_\Gamma$, the effective pressure~$p_\Gamma$ inside the fracture is obtained from the solution of the problem~\eqref{eq:Gamma1},~\eqref{eq:Gamma2}. 
However, in order to obtain a wellposed problem, we will have to supplement the bulk problem~\eqref{eq:bulk1},~\eqref{eq:bulk2} by two additional boundary conditions at the fracture~$\Gamma$. 
In general, these conditions will rely on the effective pressure~$p_\Gamma$ inside the fracture, which is why we refer to them as coupling conditions. 

\subsection{Coupling Conditions} \label{sec:sec32}
\subsubsection{First Coupling Condition}
For the derivation of a first coupling condition, we fixate $\smash{( 0 , \vct{t}^T )^T_\mathcal{N}} \in \Gamma$  and consider the line integral of~$\matr{K}_\mathrm{f} \nabla p_\mathrm{f}$ along the curve~$\vct{r}_\vct{t}$ specified in \cref{eq:rt}. 
Then, applying the trapezoidal rule yields
\begin{align}
\int_{\vct{r}_\vct{t}} \matr{K}_\mathrm{f} \nabla p_\mathrm{f} \cdot \mathrm{d} \vct{r}  =d\left(\vct{t}\right) \avg{\matr{K}\nabla p} \left(\vct{t}\right) + \mathcal{O} \left( D^3\right) , \label{eq:intrt1}
\end{align}
where we have used the continuity condition~\eqref{eq:ucontinuity}. 
The approximation error in \cref{eq:intrt1} holds true if $\matr{K}_\mathrm{f} \left( \vct{r}_\vct{t} ( \cdot )\right) \nabla p_\mathrm{f}\left(\vct{r}_\vct{t} (\cdot )\right) \cdot \dot{\vct{r}}_\vct{t} (\cdot )$ is two times continuously differentiable. 

Further, using that by assumption $\dot{\vct{r}}_\vct{t}$ is an eigenvector of~$\matr{K}_\mathrm{f}$, we obtain
\begin{align}
\begin{split}
\int_{\vct{r}_\vct{t}} \matr{K}_\mathrm{f} \nabla p_\mathrm{f} \cdot \mathrm{d} \vct{r} &= \int_{-d_1\left(\vct{t}\right)}^{d_2\left(\vct{t}\right)} \matr{K}_\mathrm{f} \left( \vct{r}_\vct{t}\left( s\right)\right) \nabla p_\mathrm{f} \left( \vct{r}_\vct{t}\left( s\right)\right) \cdot \dot{\vct{r}}_\vct{t}\left( s\right) \,\mathrm{d}s \\
&= K_\Gamma^\perp \left(\vct{t}\right) \int_{\vct{r}_\vct{t}} \nabla p_\mathrm{f} \cdot \mathrm{d} \vct{r} \, +\, \mathcal{O}\left( D \right) = K_\Gamma^\perp \left(\vct{t}\right) \jump{p}\left(\vct{t}\right) \,+\, \mathcal{O}\left( D \right), \label{eq:intrt2}
\end{split}
\end{align}
where we have used \cref{eq:Kperp} and the pressure continuity from~\cref{eq:Vdd}.
The calculation in \cref{eq:intrt2} is exact if the eigenvalue~$\smash{K_\mathrm{f}^\perp}$ is constant along the curve~$\vct{r}_\vct{t}$.

Now, combining \cref{eq:intrt1,eq:intrt2} suggests the coupling condition
\begin{align}
\avg{\matr{K}\nabla p}  &= \frac{ K_\Gamma^\perp}{d} \jump{p} . \label{eq:coupling1}
\end{align} 
We remark that, by \Cref{def:jumpavg}, the coupling condition~\eqref{eq:coupling1} depends on the gradients of the aperture functions~$d_1$ and~$d_2$. 

\subsubsection{Second Coupling Condition}
For the derivation of a second coupling condition along~$\Gamma$, we fixate $\smash{( 0 , \vct{t}^T)^T_\mathcal{N}} \in \Gamma$ and consider the definition of the mean pressure~$p_\Gamma$ in \cref{eq:pGamma}. 

Let $\psi_1, \psi_2 \in \mathcal{C}^\infty_\mathrm{c} \left( \mathbb{R}\right)$ with $0 \le \psi_i \le 1$ for $i\in\left\{1,2\right\}$, $\psi_1\left( - d_1\left(\vct{t}\right)\right) = \psi_2\left(d_2\left(\vct{t}\right)\right) = 1$, and $\mathrm{supp}\left(\psi_1\right) \subset B_1 \left( - d_1\left(\vct{t}\right)\right)$, $\mathrm{supp}\left(\psi_2\right) \subset B_1\left(d_2\left(\vct{t}\right)\right)$. 
Here, $B_r\left( z\right)$ denotes the interval~$B_r\left( z\right) := \left( z - r, z+r\right)$. 
Further, for $\epsilon > 0$, we define the functions $\psi_1^\epsilon , \psi_2^\epsilon \in \mathcal{C}_\mathrm{c}^\infty\left(\mathbb{R}\right)$ by
\begin{align}
\psi_1^\epsilon \left( s\right) := \psi_1 \left( \frac{s + d_1(\vct{t})}{\epsilon} - d_1( \vct{t} ) \right) , \quad\enspace
\psi_2^\epsilon \left( s \right) := \psi_2 \left( \frac{s - d_2(\vct{t})}{\epsilon} + d_2(\vct{t} ) \right)
\end{align}
such that $\mathrm{supp}\left(\psi_1^\epsilon\right)\subset B_\epsilon \left( -d_1\left(\vct{t}\right)\right)$ and $\mathrm{supp}\left(\psi_2^\epsilon\right) \subset B_\epsilon \left( d_2\left(\vct{t}\right)\right)$. 
In addition, for $\epsilon > 0$, let $\Psi_1^\epsilon$ and $\Psi_2^\epsilon$ be antiderivatives of $\psi_1^\epsilon$ and $\psi_2^\epsilon$ such that $\Psi_1^\epsilon \left( s\right) \rightarrow 0$ for $ s \rightarrow\infty$ and $\Psi_2^\epsilon\left( s\right) \rightarrow 0$ for $ s \rightarrow -\infty$. 
Then, for $ s \in \left( -d_1\left(\vct{t}\right) , d_2\left(\vct{t}\right)\right)$, we define the curve~$\vct{c}_\vct{t}^\epsilon \colon (-d_1\left(\vct{t}\right) , d_2\left(\vct{t}\right)) \rightarrow \mathbb{R}^n$ by
\begin{align}
\vct{c}_\vct{t}^\epsilon \left( s \right) := \big( s , \vct{t}^T \big)^{\! T}_{\! \mathcal{N}} + \Psi_1^\epsilon \left( s \right) \nabla d_1\left(\vct{t}\right) - \Psi_2^\epsilon \left( s \right) \nabla d_2\left(\vct{ t}\right) . \label{eq:cteps}
\end{align}
We remark that, since $\Psi_1^\epsilon$ and $\Psi_2^\epsilon$ vanish as $\epsilon \rightarrow 0$, the curve~$\vct{c}_\vct{t}^\epsilon$ lies inside the fracture domain~$\Omega_\mathrm{f}$ if $\epsilon > 0$ is sufficiently small. 
Besides, one can observe that
\begin{align}
\dot{\vct{c}}_\vct{t}^\epsilon \left( s\right) = \vct{n} + \psi_1^\epsilon \left( s\right) \nabla d_1\left( \vct{t}\right) - \psi_2^\epsilon\left( s\right) \nabla d_2\left(\vct{t}\right)
\end{align}
for $s \in \left( -d_1\left(\vct{t}\right) , d_2\left(\vct{t}\right)\right)$.
Consequently, if $\epsilon > 0 $ is sufficiently small, we have
\begin{align}
\label{eq:cdot} \dot{\vct{c}}_\vct{t}^\epsilon \left(-d_1\left(\vct{t}\right)\right) &= -\vct{n}_1 \smash{\sqrt{1 + \abs{\nabla d_1\left(\vct{t}\right)}^2}} , \qquad  \dot{\vct{c}}_\vct{t}^\epsilon \left(d_2\left(\vct{t}\right)\right) = \vct{n}_2 {\sqrt{1 + \abs{\nabla d_2\left(\vct{t}\right)}^2}} . 
\end{align}
Further, assuming that $p_\mathrm{f}$ is bounded in $\smash{\overline{\Omega}_\mathrm{f}}$, we have for a.a.~$\smash{( 0, \vct{t}^T )^T_\mathcal{N} \in \Gamma}$ that
\begin{align}
\frac{1}{d\left(\vct{t}\right)}\int_{\vct{c}_\vct{t}^\epsilon} p_\mathrm{f}\, \mathrm{d} r \rightarrow p_\Gamma \left(\vct{t}\right) \qquad \text{for } \epsilon \rightarrow 0 . \label{eq:epslimit}
\end{align}

Next, we approximate the pressure~$p_\mathrm{f}$ in~$\Omega_\mathrm{f}$ along the curve~$\vct{c}_\vct{t}^\epsilon$ by means of the third-order Hermite interpolation polynomial~$\pi_\vct{t}^\epsilon$ defined by the following conditions.%
\begin{subequations}
\begin{align}
\pi_\vct{t}^\epsilon \left( -d_1\left(\vct{t}\right)\right) &= \restr{p_1}{\Gamma_1}\!\left(\vct{t}\right) , \\
\pi_\vct{t}^\epsilon \left( d_2\left(\vct{t}\right)\right) &= \restr{p_2}{\Gamma_2}\!\left(\vct{t}\right) , \\
\label{eq:polycond_c} \dot{\pi}_\vct{t}^\epsilon \left(-d_1\left(\vct{t}\right)\right) &= \restr{\nabla p_\mathrm{f}\left( \vct{c}_\vct{t}^\epsilon\left( s\right) \right) \cdot \dot{\vct{c}}_\vct{t}^\epsilon \left( s\right)}{ s = - d_1\left(\vct{t}\right)} , \\
\label{eq:polycond_d} \dot{\pi}_\vct{t}^\epsilon \left(d_2\left(\vct{t}\right)\right) &= \restr{\nabla p_\mathrm{f}\left( \vct{c}_\vct{t}^\epsilon\left( s\right) \right) \cdot \dot{\vct{c}}_\vct{t}^\epsilon \left( s\right)}{ s =  d_2\left(\vct{t}\right)} .
\end{align}%
\label{eq:polycond}%
\end{subequations}
Assuming that $\smash{\restr{K_\mathrm{f}^\perp}{\Gamma_i}}$ and $\smash{\restr{\left[ \matr{K}_i\nabla p_i\right]}{\Gamma_i}}$ are continuous for $i\in\left\{ 1,2\right\}$, we have
\begin{subequations}
\begin{align}
\restr{\smash{K_\mathrm{f}^\perp}}{\Gamma_i}\! \left( \vct{t} + \Delta \vct{t} \right) &= \restr{\smash{K_\mathrm{f}^\perp}}{\Gamma_i}\! \left( \vct{t} \right) +  {\scriptstyle\mathcal{O}}\big( \abs{\Delta \vct{t}}^0\big), \\
\restr{\left[ \matr{K}_i\nabla p_i\right]}{\Gamma_i}\! \left( \vct{t} + \Delta \vct{t} \right) &= \restr{\left[ \matr{K}_i\nabla p_i\right]}{\Gamma_i}\! \left( \vct{t} \right) +  {\scriptstyle\mathcal{O}}\big( \abs{\Delta \vct{t}}^0\big)
\end{align}%
\label{eq:smallo}%
\end{subequations}
for $\smash{( 0, \vct{t}^T )^T_\mathcal{N} ,\, ( 0, \vct{t}^T\! + \Delta \vct{t}^T)^T_\mathcal{N} \in \Gamma}$ with~$\Delta \vct{ t} \in \mathbb{R}^{n-1}$. 
Thus, using \cref{eq:Kperp} and \cref{eq:smallo}, we can express the interpolation conditions~\eqref{eq:polycond_c} and~\eqref{eq:polycond_d} as 
\begin{subequations}
\begin{align}
\dot{\pi}_\vct{t}^\epsilon \left(-d_1\left(\vct{t}\right)\right) &=   -\frac{\restr{\left[\matr{K}_1 \nabla p_1\right]}{\Gamma_1} \!\left(\vct{t}\right)}{K_\Gamma^\perp \left(\vct{t}\right)}   \cdot \vct{n}_1 \left(\vct{t}\right) \sqrt{1 + \abs{\nabla d_1 \left(\vct{t}\right)}^2} + \mathcal{O} \left( D \right) +  {\scriptstyle\mathcal{O}}\left( \epsilon^0\right),  \\
\dot{\pi}_\vct{t}^\epsilon \left(d_2\left(\vct{t}\right)\right) &= \frac{\restr{\left[\matr{K}_2 \nabla p_2\right]}{\Gamma_2} \!\left(\vct{t}\right)}{K_\Gamma^\perp \left(\vct{t}\right)} \cdot \vct{n}_2 \left(\vct{t}\right)\sqrt{1 + \abs{\nabla d_2 \left(\vct{t}\right)}^2} + \mathcal{O} \left( D \right) +  {\scriptstyle\mathcal{O}}\left( \epsilon^0\right)
\end{align}%
\label{eq:couplinglandau}
\end{subequations}
if~$\epsilon > 0$ is sufficiently small.
Further, for $s\in\left( -d_1\left(\vct{ t} \right) , d_2\left(\vct{t}\right)\right)$, one has
\begin{align}
\pi_\vct{t}^\epsilon \left( s \right) := \sum_{i=0}^3 \alpha_{\vct{t}, i}^\epsilon  s^i  = p_\mathrm{f} \left( \vct{c}_\vct{t}^\epsilon\left( s\right)\right) + \mathcal{O} \big( D^4 \big) \label{eq:interpolation}
\end{align}
if $p_\mathrm{f}$ is four times continuously differentiable along~$\vct{c}_\vct{t}^\epsilon$.
Since the polynomial~$\pi_\vct{t}^\epsilon$ is uniquely defined by the conditions in \cref{eq:polycond}, we can determine explicit expressions for the coefficients~$\alpha_{\vct{t},i}^\epsilon$, $i \in \{ 0,\dots , 3\}$. 
Specifically, we obtain 
\begin{subequations}
\begin{align}
\alpha_{\vct{t}, 3}^\epsilon &= \mathcal{O} \big( d\left(\vct{t}\right)^{-1} \big) +  {\scriptstyle\mathcal{O}}\left( \epsilon^0\right) , \\
\alpha_{\vct{t}, 2}^\epsilon &= -\frac{1}{2 K_\Gamma^\perp\left(\vct{t}\right) d\left(\vct{t}\right)} \jump{\matr{K}\nabla p} \left(\vct{t}\right) + \mathcal{O} \big( d\left(\vct{t}\right)^0 \big) +  {\scriptstyle\mathcal{O}}\left( \epsilon^0\right),  \\
\alpha_{\vct{t}, 1}^\epsilon &= \frac{1}{d\left(\vct{t}\right)} \jump{p}\left(\vct{t}\right) -  \big[ d_2\left(\vct{t}\right) - d_1\left(\vct{t}\right)\big]\alpha_{\vct{t} , 2} + \mathcal{O} \big( d\left(\vct{t}\big)\right) +  {\scriptstyle\mathcal{O}}\left( \epsilon^0\right), \\
\begin{split}
\alpha_{\vct{t}, 0}^\epsilon &= \avg{p}\left(\vct{t}\right) - d_1\left(\vct{t}\right) d_2\left(\vct{t}\right) \alpha_{\vct{t} , 2}   - \frac{d_2\left(\vct{t}\right) - d_1\left(\vct{t}\right)}{2 K_\Gamma^\perp \left(\vct{t}\right)} \avg{\matr{K}\nabla p } \left(\vct{t}\right) \\
&\quad\; + \mathcal{O} \big( d\left(\vct{t}\right)^2 \big) +  {\scriptstyle\mathcal{O}}\left( \epsilon^0\right),
\end{split}
\end{align}%
\label{eq:coefficients}%
\end{subequations}
where we have utilized the first coupling condition~\eqref{eq:coupling1} as well as \cref{eq:couplinglandau}.

As a result, we can approximate the mean pressure~$p_\Gamma$ by
\begin{subequations}
\begin{align}
p_\Gamma \left(\vct{t} \right)  &= \lim_{\epsilon\rightarrow 0}\frac{1}{d\left(\vct{t}\right) } \int_{\vct{c}_\vct{t}^\epsilon} p_\mathrm{f} \,\mathrm{d} r = \lim_{\epsilon\rightarrow 0}\frac{1}{d\left(\vct{t}\right) } \int_{-d_1\left(\vct{t}\right)}^{d_2\left(\vct{t}\right)}\! \pi_\vct{t}^\epsilon\left( s \right) \,\mathrm{d} s + \mathcal{O}\left( D^4\right) .
\end{align}
Besides, with Simpson's rule, we have
\begin{align}
\begin{split}
&\frac{1}{d\left(\vct{t}\right) } \int_{-d_1\left(\vct{t}\right)}^{d_2\left(\vct{t}\right)} \pi_\vct{t}^\epsilon \left( s \right) \,\mathrm{d} s   = \frac{1}{6} \left[ \pi_\vct{t}^\epsilon \left( - d_1 \left(\vct{t} \right) \right) + 4 \pi_\vct{t}^\epsilon \left( \frac{d_2\left(\vct{t}\right) - d_1\left(\vct{t}\right)}{2}\right) + \pi_\vct{t}^\epsilon \left(d_2 \left(\vct{t}\right)\right) \right] . 
\end{split}
\end{align}%
\label{eq:simpson}%
\end{subequations}
Now, substituting the explicit form~\eqref{eq:interpolation} of the polynomial~$\pi_\vct{t}^\epsilon$ with coefficients~\eqref{eq:coefficients} into~\cref{eq:simpson} suggests the coupling condition
\begin{align}
\jump{\matr{K}\nabla p} = \frac{12 K_\Gamma^\perp}{ d} \big( p_\Gamma - \avg{p} \big) . \label{eq:coupling2}%
\end{align}%

We remark that, for a symmetric fracture with constant aperture, i.e., $d_1 = d_2 = d / 2 \equiv \mathrm{const}.$, the coupling conditions~\eqref{eq:coupling1} and \eqref{eq:coupling2} coincide with the coupling conditions formulated in~\cite{martin05} for $\xi = \frac{2}{3}$. In fact, the general form of the coupling conditions in~\cite{martin05} is given by
\begin{align}
\label{eq:martincoupling} \avg{\matr{K}\nabla p } &= \frac{K_\Gamma^\perp}{d} \jump{p} , \quad\enspace 
\jump{\matr{K}\nabla p} = \frac{4K_\Gamma^\perp}{\left(2\xi - 1\right) d} \big( p_\Gamma - \avg{p} \big) \qquad \Big( \xi > \frac{1}{2} \Big).
\end{align}
The second coupling condition in \cref{eq:martincoupling} arises in~\cite{martin05} after a motivation of the cases~$\xi = \smash{\frac{1}{2}}$, $\xi = \smash{\frac{3}{4}}$ and $\xi = 1$, inter alia, by approximating the pressure and velocity inside or at the fracture by means of mean values or differential quotients.
For $\xi = \frac{3}{4}$ and $\xi = 1$, this motivation does not immediately transfer to our situation with a fracture of varying aperture due to different normal vectors on the interfaces~$\Gamma_1$, $\Gamma_2$, and~$\Gamma$.
Besides, the case~$\xi = \frac{1}{2}$, where $p_\Gamma = \avg{p}$ is assumed, was found to be unstable~\cite{martin05}.
We recover the analogy to the model in~\cite{martin05} by also writing the coupling condition in \cref{eq:coupling2} with a general coupling parameter~$\xi > \smash{\frac{1}{2}}$ and, for convenience, we introduce the abbreviation 
\begin{align}
\beta_\Gamma := \frac{4K_\Gamma^\perp}{\left(2\xi - 1\right) d}. \label{eq:beta}
\end{align}
However, the derivation above renders the case $\xi=\smash{\frac{2}{3}}$ as an optimal choice in the sense that the unique lowest-order interpolation polynomial satisfying the conditions in \cref{eq:polycond} has been taken.

Now, by rearranging \cref{eq:coupling1,eq:coupling2} and with~$\beta_\Gamma$ as defined in \cref{eq:beta}, we find that the coupling conditions can also be written as 
\begin{subequations}
\begin{align}
\restr{\matr{K}_1\nabla p_1}{\Gamma_1} \! \cdot \vct{n}_1\sqrt{1 + \abs{\nabla d_1}^2}  &=  \frac{\beta_\Gamma}{2} \big( \avg{p} - p_\Gamma \big)  - \frac{K_\Gamma^\perp}{d} \jump{p}  , \\
\restr{\matr{K}_2\nabla p_2}{\Gamma_2} \! \cdot \vct{n}_2\sqrt{1 + \abs{\nabla d_2 }^2} &= \frac{\beta_\Gamma}{2} \big( \avg{p} - p_\Gamma \big)  + \frac{K_\Gamma^\perp}{d}\jump{p}.
\end{align}%
\label{eq:couplingNew}%
\end{subequations}
Then, the coupling conditions as given in \cref{eq:couplingNew} can be substituted directly into \cref{eq:coupling}. 
This results in the relation
\begin{align} 
\begin{split}
& -\sum_{i=1,2} \int_{\Gamma_i} \big[ \phi_\mathrm{f} - \phi_i \big] \matr{K}_i \nabla p_i \cdot \vct{n}_i \,\mathrm{d} \sigma \\
&\hspace{3cm} = \int_\Gamma \frac{K_\Gamma^\perp}{d} \jump{p} \jump{\phi} \,\mathrm{d}\sigma 
 + \int_\Gamma \beta_\Gamma  \big( p_\Gamma - \avg{p} \big)   \big( \phi_\Gamma - \avg{\phi} \big) \,\mathrm{d} \sigma .
\end{split}%
\label{eq:weakcoupling}%
\end{align}
Concluding the model derivation, we can now substitute the relation in \cref{eq:weakcoupling} into \cref{eq:weaksummary}.
The resulting reduced model is summarized in \Cref{sec:sec4} below.

\section{Darcy Flow with Interfacial Fracture} \label{sec:sec4}
In this section, we summarize the new reduced model derived in \Cref{sec:sec3} and discuss its wellposedness. 
Besides, the new model motivates the definition of different model variants with a simplified, less accurate description of the varying fracture aperture. 
These model variants are introduced in \Cref{sec:sec41}.
The geometry of the reduced problem is sketched in \Cref{fig:Omega_Gamma}. 

\begin{figure}[tbh]
\centering
\includegraphics[width=0.6\textwidth]{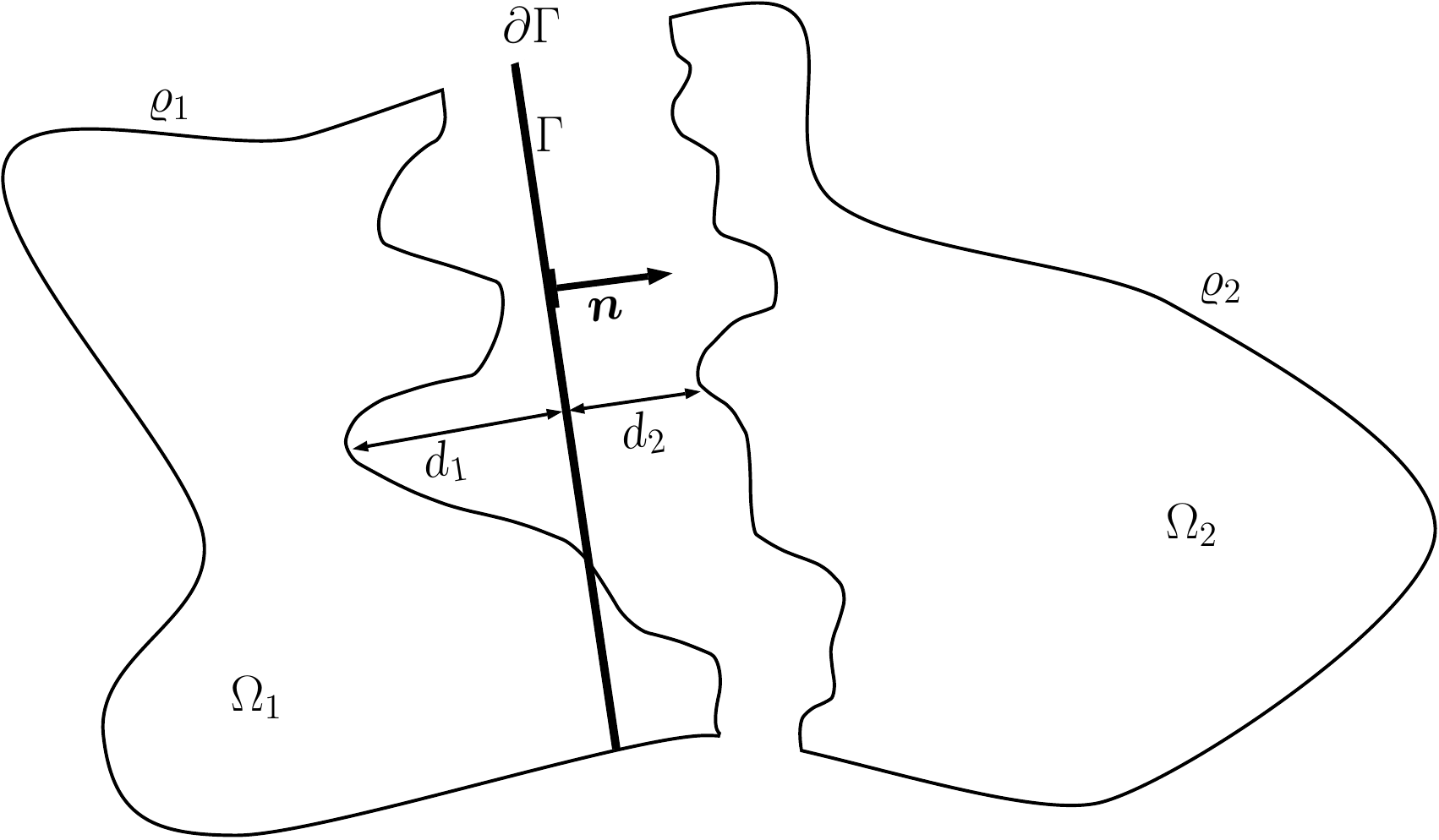}
\caption{Sketch of the geometry for the reduced Darcy problem~\eqref{eq:weak}.}
\label{fig:Omega_Gamma}
\end{figure}%
For $i\in\left\{ 1,2\right\}$, let $q_i \in L^2\left(\Omega_i\right)$, $q_\Gamma \in L^2\left( \Gamma \right)$, and $d_1 , d_2 \in W^{1,\infty} \left(\Gamma\right)$ with $d := d_1 + d_2 > d_\mathrm{min}$ for a constant~$d_\mathrm{min} > 0$. 
Besides, let $\matr{K}_i \in L^\infty\left(\Omega_i ; \mathbb{R}^{n\times n}\right)$ for $i\in\left\{ 1,2\right\}$ and $\matr{K}_\Gamma \in  L^\infty \left( \Gamma ; \mathbb{R}^{n \times n}\right) $ as well as $\smash{K_\Gamma^\perp \in L^\infty \left( \Gamma \right)}$. 
In addition, let $\matr{K}_1$, $\matr{K}_2$, and $\matr{K}_\Gamma$ be symmetric and uniformly elliptic, i.e., there exist constants~$\kappa_\mathrm{max}^\mathrm{b} \ge \kappa_\mathrm{min}^\mathrm{b} > 0 $ and~$\kappa_\mathrm{max}^\Gamma \ge \kappa_\mathrm{min}^\Gamma > 0 $ such that
\begin{subequations}
\begin{align}
\kappa_\mathrm{min}^\mathrm{b} \abs{\vct{z}}_2^2 &\le \matr{K}_i \left( \vct{x}_i \right) \vct{z} \cdot \vct{z} \le \kappa_\mathrm{max}^\mathrm{b} \abs{\vct{z}}_2^2 , \label{eq:lambda_ineq1}\\
\kappa_\mathrm{min}^\Gamma \abs{\vct{z}}_2^2 & \le \,  \matr{K}_\Gamma \left( \vct{t} \right) \vct{z} \cdot \vct{z} \,  \le \kappa_\mathrm{max}^\Gamma \abs{\vct{z}}_2^2 \label{eq:lambda_ineq2}
\end{align}
for all $\vct{z} \in \mathbb{R}^n$.
Besides, we require that 
\begin{align}
\kappa_\mathrm{min}^\Gamma &\le K_\Gamma^\perp \left( \vct{t} \right) \le \kappa_\mathrm{max}^\Gamma .
\end{align}%
\label{eq:lambda_ineq}%
\end{subequations}
The conditions in~\cref{eq:lambda_ineq} are supposed to hold for almost every~$\vct{x}_i \in \Omega_i$, $i\in\left\{ 1,2\right\}$, and $\smash{( 0 , \vct{t}^T)^T_{\mathcal{N}}} \in \Gamma$.
Moreover, let $\xi > \frac{1}{2}$ and $\beta_\Gamma \in L^\infty (\Gamma )$ be defined as in \cref{eq:beta}.

Further, we define the solution and test function spaces
\begin{align}
\Phi_\mathrm{b} &:= \bigtimes\nolimits_{i=1,2} H^1_{0, \rho_i }\left( \Omega_i \right), \quad\enspace
\Phi_\Gamma := H^1_0\left( \Gamma\right) , \quad\enspace 
\Phi := \Phi_\mathrm{b} \times \Phi_\Gamma .
\end{align} 
The space~$\Phi$ is equipped with the norm~$\norm{\cdot}_\Phi$ defined by 
\begin{align}
\norm{\phi}_\Phi^2 := \norm{\phi_\mathrm{b}}_{\Phi_\mathrm{b}}^2 + \norm{\phi_\Gamma}_{\Phi_\Gamma}^2 := \left( \norm{\phi_1}^2_{H^1\left(\Omega_1\right)}  + \norm{\phi_2}^2_{H^1\left(\Omega_2\right)} \right) + \norm{\phi_\Gamma}^2_{H^1\left(\Gamma\right)} \label{eq:phinorm}
\end{align}
for $\phi = \left( \phi_\mathrm{b} , \phi_\Gamma\right) \in \Phi$ with $\phi_\mathrm{b} = \left( \phi_1 , \phi_2\right) \in \Phi_\mathrm{b}$.
Then, a weak formulation of the reduced interface model derived in \Cref{sec:sec3} is given by the following problem. 
Find $p = \left( p_\mathrm{b} , p_\Gamma \right)\in  \Phi $ such that
\begin{align}
\mathcal{A}\left( p , \phi \right) &= \mathcal{R} \left( \phi\right) \qquad \text{for all } \phi = \left( \phi_\mathrm{b} , \phi_\Gamma \right) \in  \Phi .  \label{eq:weak}
\end{align}
Here, for $p = \left( p_\mathrm{b}, p_\Gamma \right) , \phi = \left( \phi_\mathrm{b}, \phi_\Gamma\right) \in \Phi$ with $p_\mathrm{b} = \left( p_1 , p_2 \right), \phi_\mathrm{b} = \left( \phi_1 , \phi_2\right) \in \Phi_\mathrm{b}$, the bilinear form~$\mathcal{A} \colon \Phi \times \Phi \rightarrow \mathbb{R}$ and the linear form~$\mathcal{R} \colon \Phi \rightarrow \mathbb{R}$ are defined by%
\begin{subequations}
\begin{align}
\mathcal{A} \left( p, \phi \right) &:= \mathcal{A}_\mathrm{b} \left( p_\mathrm{b}, \phi_\mathrm{b} \right) + \mathcal{A}_\Gamma \left( p , \phi_\Gamma \right) + \mathcal{I} \left( p , \phi \right) , \label{eq:Btot}\\ 
\mathcal{R} \left( \phi  \right) &:= \mathcal{R}_\mathrm{b} \left( \phi_\mathrm{b} \right) + \mathcal{R}_\Gamma \left( \phi_\Gamma \right) . \label{eq:Ltot}
\end{align}
\end{subequations}
Specifically, the bilinear forms~$\mathcal{A}_\mathrm{b} \colon \Phi_\mathrm{b} \times \Phi_\mathrm{b} \rightarrow \mathbb{R}$, $\mathcal{A}_\Gamma \colon \Phi \times  \Phi_\Gamma \rightarrow \mathbb{R}$, and $\mathcal{I}\colon \Phi\times \Phi \rightarrow \mathbb{R}$, which in this order represent the flow in the bulk domain, the effective flow inside the fracture, and the interfacial coupling between them, as well as the corresponding linear forms~$\mathcal{R}_\mathrm{b} \colon \Phi_\mathrm{b}  \rightarrow \mathbb{R}$ and $\mathcal{R}_\Gamma\colon \Phi_\Gamma \rightarrow \mathbb{R}$, are given by 
\begin{subequations}
\begin{align}
\mathcal{A}_\mathrm{b} \left( p_\mathrm{b} , \phi_\mathrm{b} \right) &:= \sum_{i=1,2} \int_{\Omega_i} \matr{K}_i \nabla p_i \cdot \nabla \phi_i \,\mathrm{d} V , \label{eq:Bbulk}\\
\mathcal{A}_\Gamma \left( p , \phi_\Gamma \right) &:= \!\int_\Gamma  \matr{K}_\Gamma \nabla \phi_\Gamma \cdot \!\left[ \nabla \left( dp_\Gamma \right) - \restr{p_1}{\Gamma_1} \!\nabla d_1 - \restr{p_2}{\Gamma_2} \!\nabla d_2 \right] \,\mathrm{d}\sigma , \label{eq:BGamma}\\
\begin{split}
\mathcal{I} \left( p , \phi\right) &:= \int_\Gamma \frac{K_\Gamma^\perp}{d} \jump{p_\mathrm{b}} \jump{\phi_\mathrm{b}} \,\mathrm{d}\sigma  + \int_\Gamma \beta_\Gamma  \big( p_\Gamma - \avg{p_\mathrm{b}} \big)   \big( \phi_\Gamma - \avg{\phi_\mathrm{b}} \big) \,\mathrm{d} \sigma , \label{eq:I}
\end{split} \\
\mathcal{R}_\mathrm{b} \left( \phi_\mathrm{b}\right) &:= \sum_{i=1,2} \int_{\Omega_i} q_i \phi_i \,\mathrm{d} V , \\
\mathcal{R}_\Gamma \left( \phi_\Gamma\right) &:= \int_\Gamma  q_\Gamma \phi_\Gamma \,\mathrm{d}\sigma .
\end{align}
\end{subequations}
The wellposedness of the weak problem~\eqref{eq:weak} is guaranteed by the following result.
\begin{theorem}  \label{thm:reducedwellposed}
Given the condition
\begin{align}
&\left[ \frac{\kappa_\mathrm{max}^\Gamma}{\kappa_\mathrm{min}^\Gamma} \right]^{ 2} \frac{D}{d_\mathrm{min}} \bigg[ ( 2\xi  - 1 ) \norm{\nabla d}_{L^\infty\left( \Gamma ; \, \mathbb{R}^n\right)}^2 + \norm{\nabla d_1  - \nabla d_2}_{L^\infty\left( \Gamma ;\, \mathbb{R}^n\right)}^2\bigg]  < 16 , 
\label{eq:wellposedcond}%
\end{align}%
the reduced Darcy problem~\eqref{eq:weak} has a unique solution $\left( p_\mathrm{b} , p_\Gamma \right) \in \Phi$. 
\end{theorem}
We remark that the condition in \cref{eq:wellposedcond} appears reasonable as it prohibits large permeability fluctuations within the fracture and rules out fractures that are geometrically extreme in terms of steep aperture gradients and large aperture fluctuations.

\begin{proof}
The proof is based on the Lax-Milgram theorem. 
It is easy to see that the bilinear form~$\mathcal{A}$ from \cref{eq:Btot} is continuous with respect to the norm in~\cref{eq:phinorm}.
In the following, we will show that $\mathcal{A}$ is coercive under the condition in \cref{eq:wellposedcond}.

Using~\eqref{eq:lambda_ineq1} and Poincaré's inequality, it is evident that the bilinear form~$\mathcal{A}_\mathrm{b}$ is coercive on~$\Phi_\mathrm{b}$.
Further, concerning the bilinear form~$\mathcal{A}_\Gamma$, a simple calculation yields
\begin{align}
\begin{split}
&\big[ \phi_\Gamma - \restr{\phi_1}{\Gamma_1} \big] \nabla d_1 + \big[ \phi_\Gamma - \restr{\phi_2}{\Gamma_2}\big] \nabla d_2\\
&\hspace{3.5cm} = \big( \phi_\Gamma - \avg{\phi_\mathrm{b}}\big) \nabla d + \frac{1}{2} \jump{\phi_\mathrm{b}} \big( \nabla d_1 - \nabla d_2\big) .
\end{split} \label{eq:gammacalc}
\end{align}
Thus, using the condition~\eqref{eq:lambda_ineq} and Hölder's inequality, we have
\begin{align}
\begin{split}
\mathcal{A}_\Gamma \left( \phi , \phi_\Gamma \right) &= \int_\Gamma d\matr{K}_\Gamma \nabla \phi_\Gamma \cdot \nabla \phi_\Gamma \,\mathrm{d}\sigma  + \smash{\sum_{i=1,2}} \smash{\int_\Gamma} \big[ \phi_\Gamma - \restr{\phi_i}{\Gamma_i}\big] \matr{K}_\Gamma \nabla \phi_\Gamma \cdot  \nabla d_i \,\mathrm{d}\sigma \\[6pt]
&\ge \kappa_\mathrm{min}^\Gamma d_\mathrm{min} \norm{\nabla \phi_\Gamma }_{L^2\left(\Gamma ;\,\mathbb{R}^n\right)}^2 \\
&\quad -\kappa_\mathrm{max}^\Gamma \norm{\nabla d }_{L^\infty\left( \Gamma ;\, \mathbb{R}^n\right)} \norm{\phi_\Gamma - \avg{\phi_\mathrm{b}}}_{L^2\left(\Gamma\right)} \norm{\nabla \phi_\Gamma }_{L^2\left(\Gamma ;\, \mathbb{R}^n\right)} \\
&\quad - \frac{\kappa_\mathrm{max}^\Gamma}{2} \norm{\nabla d_1 - \nabla d_2}_{L^\infty \left( \Gamma ; \, \mathbb{R}^n \right)} \norm{\jump{\phi_\mathrm{b}}}_{L^2\left( \Gamma \right)} \norm{\nabla \phi_\Gamma }_{L^2\left( \Gamma ;\, \mathbb{R}^n \right)} .
\end{split}
\end{align}
By Young's inequality, for any $\delta , \epsilon > 0$, it holds
\begin{align}
\norm{\phi_\Gamma - \avg{\phi_\mathrm{b}}}_{L^2\left(\Gamma\right)} \norm{\nabla \phi_\Gamma }_{L^2\left(\Gamma ;\, \mathbb{R}^n\right)}  &\le \epsilon  \norm{\phi_\Gamma - \avg{\phi_\mathrm{b}}}^2_{L^2\left(\Gamma\right)} + \frac{1}{4\epsilon} \norm{\nabla \phi_\Gamma }_{L^2\left(\Gamma ; \,\mathbb{R}^n\right)}^2 , \notag \\
\norm{\jump{\phi_\mathrm{b}}}_{L^2\left( \Gamma \right)} \norm{\nabla \phi_\Gamma }_{L^2\left( \Gamma ;\, \mathbb{R}^n \right)}  &\le \delta \norm{\jump{\phi_\mathrm{b}}}_{L^2\left( \Gamma \right)}^2 + \frac{1}{4\delta} \norm{\nabla \phi_\Gamma }_{L^2\left( \Gamma ;\, \mathbb{R}^n \right)}^2 .
\end{align}
Besides, with \cref{eq:lambda_ineq}, we have
\begin{align}
\mathcal{I} \left( \phi , \phi \right) \overset{\eqref{eq:lambda_ineq}}{\ge} \frac{\kappa_\mathrm{min}^\Gamma}{D} \norm{\jump{\phi_\mathrm{b}}}_{L^2\left(\Gamma\right)}^2 + \frac{4}{2\xi - 1} \frac{\kappa_\mathrm{min}^\Gamma}{D} \norm{\phi_\Gamma - \avg{\phi_\mathrm{b}}}_{L^2\left(\Gamma\right)}^2 
\end{align}
for the bilinear form~$\mathcal{I}$. 
As a consequence, we obtain
\begin{align}
\begin{split}
&\mathcal{B}_\Gamma \left( \phi , \phi_\Gamma \right) + \mathcal{I} \left( \phi , \phi \right) \\
&\hspace{2.25cm} \ge \norm{\nabla \phi_\Gamma }^2_{L^2\left( \Gamma ;\, \mathbb{R}^n\right)} T_1 + \norm{\jump{\phi_\mathrm{b}}}_{L^2\left(\Gamma\right)}^2 T_2 + \norm{\phi_\Gamma - \avg{\phi_\mathrm{b}}}_{L^2\left(\Gamma\right)}^2 T_3
\end{split}%
\label{eq:B+Ige}%
\end{align}
with $T_1$, $T_2$, and $T_3$ defined by
\begin{subequations}
\begin{align}
T_1 &:=  \kappa_\mathrm{min}^\Gamma d_\mathrm{min} - \frac{\kappa_\mathrm{max}^\Gamma}{4\epsilon}\norm{\nabla d }_{L^\infty\left( \Gamma ;\, \mathbb{R}^n\right)}  - \frac{\kappa_\mathrm{max}^\Gamma}{8\delta} \norm{\nabla d_1 - \nabla d_2}_{L^\infty\left( \Gamma ; \, \mathbb{R}^n \right)}  , \\
T_2 &:=  \frac{\kappa_\mathrm{min}^\Gamma}{D} - \frac{\delta \kappa_\mathrm{max}^\Gamma}{2} \norm{\nabla d_1 - \nabla d_2 }_{L^\infty\left( \Gamma ;\, \mathbb{R}^n\right)} , \\
T_3 &:=  \frac{4}{2\xi - 1} \frac{\kappa_\mathrm{min}^\Gamma}{D} - \epsilon \kappa_\mathrm{max}^\Gamma \norm{\nabla d }_{L^\infty\left( \Gamma ; \, \mathbb{R}^n\right)}  .
\end{align}
\end{subequations}
Now, w.l.o.g., we assume $\norm{\nabla d }_{L^\infty\left( \Gamma ;\, \mathbb{R}^n\right)} , \norm{\nabla d_1 - \nabla d_2 }_{L^\infty\left( \Gamma ;\, \mathbb{R}^n\right)} \not= 0$.
Moreover, we choose $\delta , \epsilon > 0$ such that $T_2 = 0$ and $T_3 = 0$, i.e.,
\begin{align}
\delta &= \frac{2 \kappa_\mathrm{min}^\Gamma}{ D \kappa_\mathrm{max}^\Gamma  \norm{\nabla d_1 - \nabla d_2 }_{L^\infty\left( \Gamma ;\, \mathbb{R}^n\right)}}, \quad\enspace 
 \epsilon = \frac{4}{2\xi - 1}   \frac{\kappa_\mathrm{min}^\Gamma}{D \kappa_\mathrm{max}^\Gamma  \norm{\nabla d }_{L^\infty\left( \Gamma ;\, \mathbb{R}^n\right)}}.
\end{align}
Then, the condition in \cref{eq:wellposedcond} guarantees $T_1 > 0$. 
Thus, by using Poincaré's inequality, we obtain the coercivity of the overall bilinear form~$\mathcal{A}$.
\end{proof}

In case of a classical solution, a corresponding strong formulation of the weak system~\cref{eq:weak} is, for $i\in\{ 1,2 \}$, given by%
\begin{subequations}
\begin{alignat}{2}
-\nabla\cdot \left( \matr{K}_i\nabla p_i  \right) &= q_i \quad\enspace &&\text{in } \Omega_i ,  \\
\label{eq:strongB} -\nabla \cdot \Big[ \matr{K}_\Gamma \Big( \nabla \left( dp_\Gamma \right) - \restr{p_1}{\Gamma_1} \!\nabla d_1 - \restr{p_2}{\Gamma_2} \!\nabla d_2 \Big)  \Big] &= q_\Gamma -\jump{\matr{K}\nabla p} \quad\enspace &&\text{in } \Gamma , \\
\avg{\matr{K}\nabla p}  &= \frac{ K_\Gamma^\perp}{d} \jump{p} \quad\enspace &&\text{on } \Gamma , \label{eq:coupling_A} \\
\jump{\matr{K}\nabla p} &= \beta_\Gamma \big( p_\Gamma - \avg{p}\big) \quad &&\text{on } \Gamma , \label{eq:coupling_B} \\
p_i &= 0 \quad\enspace &&\text{on } \rho_i , \label{eq:strongE} \\
p_\Gamma &= 0 \qquad &&\text{on } \partial\Gamma  . \label{eq:strongF}
\end{alignat}%
\label{eq:strong}%
\end{subequations}
Here, we observe that the quantity 
\begin{align}
\vct{u}_\Gamma := - \matr{K}_\Gamma \Big( \nabla \left( dp_\Gamma \right) -  \restr{p_1}{\Gamma_1} \!\nabla d_1 - \restr{p_2}{\Gamma_2} \!\nabla d_2  \Big)  
\end{align}
in \cref{eq:strongB} takes the role of the effective velocity inside the reduced fracture~$\Gamma$. 
Further, for a symmetric fracture with constant aperture, i.e., $d_1 = d_2 = d/2 \equiv \mathrm{const}.$, the model in \cref{eq:strong} coincides with the model proposed in~\cite{martin05}. 
Therefore, the new model~\cref{eq:strong} can be viewed as an extension of the model in~\cite{martin05} for general asymmetric fractures with spatially varying aperture. 

\subsection{Model Variants} \label{sec:sec41}
According to the derivation of the reduced model~\eqref{eq:strong} in \Cref{sec:sec3}, there should be a gap between the bulk domains~$\Omega_1$ and $\Omega_2$ on either side of the fracture as illustrated in \Cref{fig:Omega_Gamma}.
However, for numerical calculations in practice, the bulk domains~$\Omega_1$ and $\Omega_2$ are usually rectified such that the interface~$\Gamma$ is part of their boundary, i.e., $\partial \Omega_1^\mathrm{rct.}  \cap  \partial \Omega_2^\mathrm{rct.} = \overline{\Gamma}$. 
The corresponding reduced model with simplified bulk geometry is obtained from the model in~\cref{eq:strong} by replacing the bulk domains~$\Omega_1$ and~$\Omega_2$ with the domains
\begin{subequations}
\begin{align}
\Omega_1^\mathrm{rct.} = \big\{ \vct{\gamma} + \lambda \vct{n} \in \Omega \ \big\vert\ \vct{\gamma} \in \Gamma ,\, \lambda < 0 \big\} , \\
\Omega_2^\mathrm{rct.} = \big\{ \vct{\gamma} + \lambda \vct{n} \in \Omega \ \big\vert\ \vct{\gamma} \in \Gamma ,\, \lambda > 0 \big\} .
\end{align}%
\label{eq:bulkdeformation}%
\end{subequations}%
This kind of bulk rectification requires one to neglect the terms containing aperture gradients~$\nabla d_1$,~$\nabla d_2$ in the coupling conditions~\eqref{eq:coupling_A} and~\eqref{eq:coupling_B}.
The geometrical difference between the full-dimensional model in~\cref{eq:darcydecomposed}, the reduced model in~\cref{eq:strong}, and the corresponding reduced model with bulk rectification is illustrated in \Cref{fig:bulkdeformation}. 

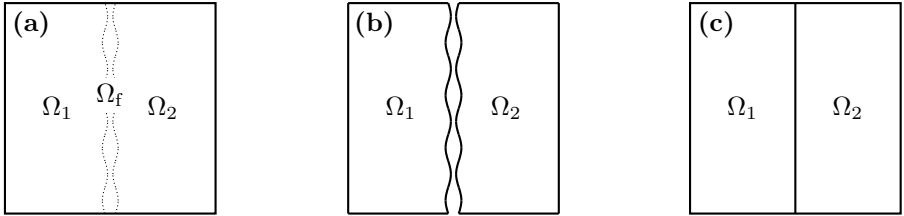
\begin{figure}[ht]
\centering
\tikzstyle{angle}=[draw=gray,angle eccentricity=.6,angle radius=0.6cm]
\begin{subfigure}{0.3\textwidth}
\centering
\begin{tikzpicture}[rotate=90, scale=0.7]
\draw[thick] (0,0) rectangle (4,4);
\draw[densely dotted] (0, 1.9)     sin (0.25, 1.85);
\draw[densely dotted] (0.25, 1.85) cos (0.5, 1.9);
\draw[densely dotted] (0.5, 1.9)   sin (0.75, 1.95);
\draw[densely dotted] (0.75, 1.95) cos (1, 1.9);
\draw[densely dotted] (1, 1.9)     sin (1.25, 1.85);
\draw[densely dotted] (1.25, 1.85) cos (1.5, 1.9);
\draw[densely dotted] (1.5, 1.9)   sin (1.75, 1.95);
\draw[densely dotted] (1.75, 1.95) cos (2, 1.9);
\draw[densely dotted] (2, 1.9)     sin (2.25, 1.85);
\draw[densely dotted] (2.25, 1.85) cos (2.5, 1.9);
\draw[densely dotted] (2.5, 1.9)   sin (2.75, 1.95);
\draw[densely dotted] (2.75, 1.95) cos (3, 1.9);
\draw[densely dotted] (3, 1.9)     sin (3.25, 1.85);
\draw[densely dotted] (3.25, 1.85) cos (3.5, 1.9);
\draw[densely dotted] (3.5, 1.9)   sin (3.75, 1.95);
\draw[densely dotted] (3.75, 1.95) cos (4, 1.9);
\draw[densely dotted] (0, 2.1)     sin (0.25, 2.15);
\draw[densely dotted] (0.25, 2.15) cos (0.5, 2.1);
\draw[densely dotted] (0.5, 2.1)   sin (0.75, 2.05);
\draw[densely dotted] (0.75, 2.05) cos (1, 2.1);
\draw[densely dotted] (1, 2.1)     sin (1.25, 2.15);
\draw[densely dotted] (1.25, 2.15) cos (1.5, 2.1);
\draw[densely dotted] (1.5, 2.1)   sin (1.75, 2.05);
\draw[densely dotted] (1.75, 2.05) cos (2, 2.1);
\draw[densely dotted] (2, 2.1)     sin (2.25, 2.15);
\draw[densely dotted] (2.25, 2.15) cos (2.5, 2.1);
\draw[densely dotted] (2.5, 2.1)   sin (2.75, 2.05);
\draw[densely dotted] (2.75, 2.05) cos (3, 2.1);
\draw[densely dotted] (3, 2.1)     sin (3.25, 2.15);
\draw[densely dotted] (3.25, 2.15) cos (3.5, 2.1);
\draw[densely dotted] (3.5, 2.1)   sin (3.75, 2.05);
\draw[densely dotted] (3.75, 2.05) cos (4, 2.1);
\node (omegaone) at (2,3) {$\Omega_1$};
\node (omegatwo) at (2,1) {$\Omega_2$};
\node (abc) at (3.6, 3.5) {\textbf{(a)}};
\node[circle,fill=white,inner sep=0pt] (omegaf) at (2.25,2) {$\Omega_\mathrm{f}$};
\end{tikzpicture}
\end{subfigure}
\hfill 
\begin{subfigure}{0.3\textwidth}
\centering
\begin{tikzpicture}[rotate=90, scale=0.7]
\draw[thick] (0, 0) -- (4, 0);
\draw[thick] (0, 0) -- (0, 1.9);
\draw[thick] (0, 2.1) -- (0, 4);
\draw[thick] (0, 4) -- (4, 4);
\draw[thick] (4, 0) -- (4, 1.9);
\draw[thick] (4, 2.1) -- (4, 4); 

\draw[thick] (0, 1.9)     sin (0.25, 1.85);
\draw[thick] (0.25, 1.85) cos (0.5, 1.9);
\draw[thick] (0.5, 1.9)   sin (0.75, 1.95);
\draw[thick] (0.75, 1.95) cos (1, 1.9);
\draw[thick] (1, 1.9)     sin (1.25, 1.85);
\draw[thick] (1.25, 1.85) cos (1.5, 1.9);
\draw[thick] (1.5, 1.9)   sin (1.75, 1.95);
\draw[thick] (1.75, 1.95) cos (2, 1.9);
\draw[thick] (2, 1.9)     sin (2.25, 1.85);
\draw[thick] (2.25, 1.85) cos (2.5, 1.9);
\draw[thick] (2.5, 1.9)   sin (2.75, 1.95);
\draw[thick] (2.75, 1.95) cos (3, 1.9);
\draw[thick] (3, 1.9)     sin (3.25, 1.85);
\draw[thick] (3.25, 1.85) cos (3.5, 1.9);
\draw[thick] (3.5, 1.9)   sin (3.75, 1.95);
\draw[thick] (3.75, 1.95) cos (4, 1.9);
\draw[thick] (0, 2.1)     sin (0.25, 2.15);
\draw[thick] (0.25, 2.15) cos (0.5, 2.1);
\draw[thick] (0.5, 2.1)   sin (0.75, 2.05);
\draw[thick] (0.75, 2.05) cos (1, 2.1);
\draw[thick] (1, 2.1)     sin (1.25, 2.15);
\draw[thick] (1.25, 2.15) cos (1.5, 2.1);
\draw[thick] (1.5, 2.1)   sin (1.75, 2.05);
\draw[thick] (1.75, 2.05) cos (2, 2.1);
\draw[thick] (2, 2.1)     sin (2.25, 2.15);
\draw[thick] (2.25, 2.15) cos (2.5, 2.1);
\draw[thick] (2.5, 2.1)   sin (2.75, 2.05);
\draw[thick] (2.75, 2.05) cos (3, 2.1);
\draw[thick] (3, 2.1)     sin (3.25, 2.15);
\draw[thick] (3.25, 2.15) cos (3.5, 2.1);
\draw[thick] (3.5, 2.1)   sin (3.75, 2.05);
\draw[thick] (3.75, 2.05) cos (4, 2.1);
\node (omegaone) at (2,3) {$\Omega_1$};
\node (omegatwo) at (2,1) {$\Omega_2$};
\node (abc) at (3.6, 3.5) {\textbf{(b)}};
\end{tikzpicture}
\end{subfigure}
\hfill
\begin{subfigure}{0.3\textwidth}
\centering
\begin{tikzpicture}[rotate=90, scale=0.7]
\draw[thick] (0,0) rectangle (4,4);
\draw[thick] (0, 2) -- (4, 2);
\node (omegaone) at (2,3) {$\Omega_1$};
\node (omegatwo) at (2,1) {$\Omega_2$};
\node (abc) at (3.6, 3.5) {\textbf{(c)}};
\end{tikzpicture}
\end{subfigure}
\caption{Bulk domains \textbf{(a)} in the full-dimensional model~\eqref{eq:darcydecomposed}, \textbf{(b)} in the reduced model~\eqref{eq:strong} without rectification, \textbf{(c)} in the reduced model~\eqref{eq:strong} with rectification.}
\label{fig:bulkdeformation}
\end{figure}

In contrast to the model in~\cite{martin05}, the new model~\cref{eq:strong} contains aperture gradients~$\nabla d_1$, $\nabla d_2$ in the effective flow equation~\eqref{eq:strongB} and in the coupling conditions~\eqref{eq:coupling_A} and~\eqref{eq:coupling_B}. 
In order to study the effect of the aperture gradients as well as the effect of a rectified bulk geometry as discussed above, we define simplified variants of the reduced model~\eqref{eq:strong}.
On the one hand, we can neglect the aperture gradients~$\nabla d_1$, $\nabla d_2$ in \cref{eq:strongB}, i.e., \cref{eq:strongB} is replaced by the equation
\begin{alignat}{2}
\label{eq:strongB_ii} -\nabla \cdot \big[ \matr{K}_\Gamma  \nabla \left( dp_\Gamma \right)  \big] &= q_\Gamma -\jump{\matr{K}\nabla p} \quad\enspace &&\text{in } \Gamma . 
\end{alignat}
On the other hand, the aperture gradients~$\nabla d_1$, $\nabla d_2$ could be neglected in the coupling conditions~\eqref{eq:coupling_A} and~\eqref{eq:coupling_B}, which corresponds to a rectification of the bulk domains~$\Omega_1$ and~$\Omega_2$.
This suggests to define the following model variants. 
\\[4pt]
\begin{tabularx}{\textwidth}{lX}
~~~\emph{Model I:} & The new model~\eqref{eq:strong} without change. \\
~~~\emph{Model I-R:} & The model~\eqref{eq:strong} with the bulk domains~$\Omega_1$ and~$\Omega_2$ replaced by the rectified domains~$\Omega_1^\mathrm{rct.}$ and~$\Omega_2^\mathrm{rct.}$ from~\cref{eq:bulkdeformation}. Terms containing $\nabla d_1$, $\nabla d_2$ are neglected in~\cref{eq:coupling_A,eq:coupling_B} but not in~\cref{eq:strongB}. \\
~~~\emph{Model II:} & The model~\eqref{eq:strong} with unchanged bulk domains. Terms containing $\nabla d_1$, $\nabla d_2$ are neglected in~\cref{eq:strongB} but not in~\cref{eq:coupling_A} and~\eqref{eq:coupling_B}. \\
~~~\emph{Model II-R:} & The model~\eqref{eq:strong} with the bulk domains~$\Omega_1$ and~$\Omega_2$ replaced by the rectified domains~$\Omega_1^\mathrm{rct.}$ and~$\Omega_2^\mathrm{rct.}$ from~\cref{eq:bulkdeformation}. Terms containing $\nabla d_1$, $\nabla d_2$ are neglected completely. \\
\end{tabularx} \\[4pt]
Model~II-R is basically the model proposed in~\cite{martin05} with the only difference that the aperture~$d$ in \cref{eq:strongB_ii} can still be a function that is not necessarily constant as assumed in~\cite{martin05}. 
In particular, in model~II-R, there is no information about the aperture gradients~$\nabla d_1$, $\nabla d_2$ and the bulk geometry at the fracture. 
In contrast, the new model~I includes all this information. 
The models~I-R and~II are intermediate models. 

\section{Discontinuous Galerkin Discretization} \label{sec:sec5}
In this section, following~\cite{antonietti19}, we formulate three discontinuous Galerkin~(DG) discretizations, one for the full-dimensional model~\eqref{eq:darcydecomposed}, one for the reduced models~II and~II-R, and one for the reduced models~I and~I-R from \Cref{sec:sec4}, where, in this order, each discretization extends the previous one. 
The choice of a DG scheme as discretization for the reduced fracture models comes naturally as it can easily deal with discontinuities across the fracture and suits the formulation of the coupling conditions~\eqref{eq:coupling_A} and \eqref{eq:coupling_B} in terms of jump and average operators. 
For simplicity, we assume that~$\Omega$ is a polytopial domain.
Besides, we consider inhomogeneous Dirichlet boundary conditions for all discretizations.

\subsection{Meshes and Notations}
Let $\mathcal{I}$ be the index family of bulk domains, i.e., $\mathcal{I} = \{1,2,\mathrm{f}\}$ in the full-dimensional case and $\mathcal{I} = \{ 1, 2\}$ for the reduced models. 
Further, for $i\in \mathcal{I}$, let $\mathcal{T}_{h,i}$ be a polytopial mesh of the bulk domain~$\Omega_i$ of closed elements~$T_{h,i} \in \mathcal{T}_{h,i}$ with disjoint interiors. 
Besides, we write $\mathcal{T}_h := \bigcup_{i\in\mathcal{I}} \mathcal{T}_{h,i}$ for the overall bulk mesh, which may be non-conforming. 
Moreover, we denote by~$\mathcal{F}_h$ the facet grid induced by~$\mathcal{T}_h$ which contains all one-codimensional intersections between grid elements~$T\in\mathcal{T}_h$ with neighboring grid elements or the domain boundary~$\partial \Omega $.  
For the reduced models, we denote by~$\mathcal{F}^\Gamma_h$ a one-co\-di\-men\-sio\-nal polytopial mesh of the interface~$\Gamma$ induced by the bulk grids~$\mathcal{T}_{h,1}$ and~$\mathcal{T}_{h,2}$.
Specifically, for the reduced models~\mbox{I-R} and~II-R, i.e., in case of a reduced model with rectified bulk domains as defined in \cref{eq:bulkdeformation}, the fracture mesh~$\mathcal{F}^\Gamma_h$ is part of the facet grid~$\mathcal{F}_h$ and given by 
\begin{subequations}
\begin{align}
\mathcal{F}^\Gamma_h := \left\{ F \in \mathcal{F}_h \ \middle\vert \ F \subset \overline{\Gamma } \right\}.
\end{align}
In the other case, for the reduced models~I and~II without bulk rectification, the fracture mesh~$\mathcal{F}_h^\Gamma$ can be defined by
\begin{align}
\mathcal{F}^\Gamma_h := \Big\{  \mathscr{P}_\Gamma ( F_1 )  \cap \mathscr{P}_\Gamma ( F_2 ) \ \Big\vert\  F_i = \partial T_i \cap \overline{\Gamma}_i \not= \emptyset , \ T_i \in \mathcal{T}_{h,i} \ \text{ for }  i \in \{ 1, 2 \} \Big\}. \label{eq:FGammah_trafo}
\end{align}%
\end{subequations}%
In \cref{eq:FGammah_trafo}, $\mathscr{P}_\Gamma$ denotes the orthogonal projection onto the hyperplane~$\Gamma$ given by
\begin{align}
\mathscr{P}_\Gamma \colon \Omega \rightarrow \Omega, \enspace \big( \eta ,  \vct{t}^T \big)^{\! T}_{\!\mathcal{N}} \mapsto \big( 0, \vct{t}^T \big)^{\! T}_{\!\mathcal{N}} .
\end{align}
Regarding the facet grid~$\mathcal{F}_h$, we distinguish between the set of facets~$\mathcal{F}_h^\partial$ on the domain boundary~$\partial \Omega$ and the set of facets~$\mathcal{F}_h^\circ$ in the interior of~$\Omega$ excluding the interface grid~$\mathcal{F}^\Gamma_h$, i.e., we can write~$\mathcal{F}_h$ as the disjoint union $\mathcal{F}_h = \mathcal{F}_h^\circ \;\dot{\cup}\; \mathcal{F}_h^\partial \;\dot{\cup}\; \left( \mathcal{F}^\Gamma_h \cap \mathcal{F}_h \right) $.
In addition, for the reduced models, we denote by~$\mathcal{E}_h^\Gamma$ the set of edges of the interface grid~$\mathcal{F}^\Gamma_h$, i.e., the set of two-codimensional intersections between elements~$F\in\mathcal{F}^\Gamma_h$ or the boundary~$\partial\Gamma$. 
More specifically, we distinguish between the set of edges~$\mathcal{E}^{\circ}_h$ in the interior of the interface~$\Gamma$ and the set of edges~$\mathcal{E}^{\partial}_h$ at the boundary~$\partial \Gamma$ such that $\mathcal{E}^\Gamma_h =  \mathcal{E}^{\circ}_h \; \dot{\cup} \; \mathcal{E}^{\partial}_h$.

For $A \subset \mathbb{R}^n$, let $\mathcal{P}_k ( A )$ denote the space of polynomials on~$A$ whose degrees do not exceed~$k\in\mathbb{N}_0$.
Then, we define the finite-dimensional function spaces
\begin{subequations}
\begin{align}
\Phi^\mathrm{b}_h &:= \big\{ \phi_h \in L^2 (\Omega ) \ \big\vert \ \restr{\phi_h}{T} \in \mathcal{P}_{k_T}(T)\ \text{for all}\ T\in\mathcal{T}_h \big\} , \\
\Phi^\Gamma_h &:= \big\{ \phi^\Gamma_h \in L^2 (\Gamma ) \ \big\vert \ \restr{\phi^\Gamma_h}{F} \in \mathcal{P}_{k_F}(F)\ \text{for all}\ F\in\mathcal{F}_h^\Gamma \big\} ,\\
\Phi_h &:= \Phi_h^\mathrm{b} \times \Phi_h^\Gamma
\end{align}
\end{subequations}
with individual polynomial degrees~$k_T\in\mathbb{N}$ and $k_F \in \mathbb{N}$ for each bulk element~$T\in\mathcal{T}_h$ and interface element~$F\in\mathcal{F}_h^\Gamma$.
Further, we introduce jump and average operators for DG discretizations. 
Although we use the same notation, we note that the following definition is different from \Cref{def:jumpavg}. 
\begin{definition}[Jump and average operators for DG schemes]
Let $\mathcal{M}_h = \mathcal{T}_h$ and $\mathcal{S}_h^\circ = \mathcal{F}_h^\circ$, or $\mathcal{M}_h = \mathcal{F}_h^\Gamma$ and $\mathcal{S}_h^\circ = \mathcal{E}_h^\circ$. Further, we define the function spaces
\begin{align}
\Sigma (\mathcal{M}_h ) := \prod_{M\in\mathcal{M}_h} \! L^2 (\partial M ) , \qquad L^2(\mathcal{S}_h^\circ ) := \prod_{S\in \mathcal{S}_h^\circ} L^2 (S ) .
\end{align}
In general, functions in~$\Sigma (\mathcal{M}_h )$ will be double-valued on facets~$S\in\mathcal{S}_h^\circ$. For a function~$\phi_h\in \Sigma (\mathcal{M}_h )$, we denote the component of~$\phi_h$ associated with the mesh element~$M\in\mathcal{M}_h$ by~$\phi_h^M$. Besides, for a mesh element~$M\in\mathcal{M}_h$, we write $\vct{n}_M$ for the outer unit normal on~$\partial M$. 
We can now define the jump and average operators
\begin{subequations}
\begin{alignat}{3}
\jump{\,\cdot\, } &\colon \Sigma (\mathcal{M}_h ) \rightarrow \left[ L^2 (\mathcal{S}_h^\circ ) \right]^n,  &&\qquad \ \; \jump{\,\cdot\, } &&\colon \big[ \Sigma (\mathcal{M}_h )\big]^n \rightarrow L^2 (\mathcal{S}_h^\circ ) , \\
 \avg{\,\cdot\, } &\colon \Sigma (\mathcal{M}_h ) \rightarrow L^2 (\mathcal{S}_h^\circ ), &&\qquad \avg{\,\cdot\, } &&\colon \big[ \Sigma (\mathcal{M}_h )\big]^n \rightarrow \left[ L^2  (\mathcal{S}_h^\circ ) \right]^n   .
\end{alignat}
\end{subequations}
Let $\phi_h \in \Sigma (\mathcal{M}_h )$ and $\vct{\zeta}_h \in \big[ \Sigma (\mathcal{M}_h ) \big]^n$. For an internal facet $S\in\mathcal{S}_h^\circ$ with adjacent mesh elements~$M_1 \not= M_2 \in \mathcal{M}_h$, we define
\begin{subequations}
\begin{alignat}{3}
\restr{\jump{\phi_h}}{S} &:=  \phi_h^{M_1}\vct{n}_{M_1}  + \phi_h^{M_2} \vct{n}_{M_2}, \qquad &&\enspace\restr{\jump{\vct{\zeta}_h}}{S} &&:= \vct{\zeta}_h^{M_1} \cdot \vct{n}_{M_1} + \vct{\zeta}_h^{M_2} \cdot \vct{n}_{M_2}, \\
\restr{\avg{\phi_h}}{S} &:= \frac{1}{2} \left( \phi_h^{M_1} + \phi_h^{M_2} \right), \qquad &&\restr{\avg{\vct{\zeta}_h}}{S} &&:= \frac{1}{2} \left( \vct{\zeta}_h^{M_1} + \vct{\zeta}_h^{M_2} \right) .
\end{alignat}%
\end{subequations}%
\end{definition}

\subsection{Discrete Model with Full-Dimensional Fracture}
Let $g \in H^{1/2} (\partial \Omega )$ denote the given pressure on the boundary~$\partial \Omega $ for the Dirichlet condition in \cref{eq:dirichletdd}. 
Then, in order to obtain a DG discretization of the full-dimensional model~\eqref{eq:darcydecomposed}, we define the bilinear form~$\mathcal{A}_h^\mathrm{b}\colon \Phi_h^\mathrm{b} \times \Phi_h^\mathrm{b} \rightarrow \mathbb{R}$ associated with bulk flow and the corresponding linear form~$\mathcal{R}_h^\mathrm{b} \colon \Phi^\mathrm{b}_h \rightarrow \mathbb{R}$ by
\begin{subequations}
\begin{align}
\begin{split}
\mathcal{A}_h^\mathrm{b} (p_h^\mathrm{b} , \phi_h^\mathrm{b} ) &= \sum_{T_\in\mathcal{T}_h}\int_T \matr{K} \nabla p_h^\mathrm{b} \cdot \nabla \phi_h^\mathrm{b} \,\mathrm{d}V 
+ \sum_{F\in\mathcal{F}_h^\circ} \int_F \mu_F \jump{p_h^\mathrm{b}} \cdot \jump{\phi_h^\mathrm{b}} \,\mathrm{d} \sigma \\
&\quad - \sum_{F\in\mathcal{F}_h^\circ} \int_F \Big[ \jump{\phi_h^\mathrm{b}} \cdot  \avg{\matr{K}\nabla p_h^\mathrm{b}} + \jump{p_h^\mathrm{b}} \cdot  \avg{\matr{K}\nabla\phi_h^\mathrm{b}} \Big] \,\mathrm{d}\sigma \\
&\quad + \sum_{F\in\mathcal{F}_h^\partial} \int_F \mu_F^\mathrm{b} p_h^\mathrm{b} \phi_h^\mathrm{b} \,\mathrm{d}\sigma 
- \!\sum_{F\in\mathcal{F}_h^\partial} \int_F \big[ p_h^\mathrm{b} \matr{K}\nabla \phi_h^\mathrm{b} + \phi_h^\mathrm{b} \matr{K}\nabla p_h^\mathrm{b} \big] \!\cdot \mathrm{d}\vct{\sigma} , 
\end{split} \\
\begin{split}
\mathcal{R}^\mathrm{b}_h (\phi_h^\mathrm{b}) &=
\sum_{T\in\mathcal{T}_h} \int_T q \phi_h^\mathrm{b} \,\mathrm{d}V 
+ \sum_{F\in\mathcal{F}_h^\partial} \bigg[ \int_F \mu_F^\mathrm{b} g \phi_h^\mathrm{b} \,\mathrm{d} \sigma -  \int_F g\matr{K}\nabla \phi_h^\mathrm{b} \cdot \mathrm{d}\vct{\sigma } \bigg] .
\end{split}
\end{align}%
\label{eq:bulkforms}
\end{subequations}
In \cref{eq:bulkforms}, $\mu_F^\mathrm{b}$ is a penalty parameter which we define facet-wise, for $F\in\mathcal{F}_h \setminus \mathcal{F}_h^\Gamma$, by
\begin{align}
\mu_F^\mathrm{b} := \begin{cases}
\mu_0^\mathrm{b} \frac{(k_T + 1) (k_T + n)}{h_T}, &\hspace{-0.1cm}\text{if } F  \in \mathcal{F}_h^\partial ,\, F \subset \partial T,\, T \in\mathcal{T}_h , \\
\mu_0^\mathrm{b} \!\max\limits_{T = T_1 , T_2 } \!\left\{ \! \frac{(k_T + 1) (k_T + n)}{h_T} \!\right\} , &\hspace{-0.1cm}\text{if } F  \in \mathcal{F}_h^\circ ,\, F \subset \partial T_1 \cap \partial T_2 ,\, T_1 \not= T_2 \in \mathcal{T}_h .
\end{cases}%
\label{eq:bulkpenalty}%
\end{align}
In \cref{eq:bulkpenalty}, $\mu_0^\mathrm{b} > 0$ is a sufficiently large constant and $h_T$ denotes the maximum edge length of a grid element~$T\in\mathcal{T}_h$.

A DG discretization of the full-dimensional system~\eqref{eq:darcydecomposed} is now given by the following problem.
Find $p_h^\mathrm{b} \in \Phi_h^\mathrm{b}$ such that
\begin{align}
\mathcal{A}_h^\mathrm{b} (p_h^\mathrm{b} , \phi_h^\mathrm{b} ) = \mathcal{R}_h^\mathrm{b} (\phi_h^\mathrm{b} ) \qquad \text{for all } \phi_h^\mathrm{b} \in \Phi_h^\mathrm{b} . \label{eq:dgbulk}
\end{align}

\subsection{Discrete Model with Interfacial Fracture}
Let $g_\Gamma \in H^{1/2} (\Gamma )$ and $g = (g_1 , g_2 ) \in H^{1/2} (\rho_1 ) \times H^{1/2} (\rho_2 )$ denote the given pressure functions on the external boundaries for the Dirichlet conditions~\eqref{eq:strongE} and~\eqref{eq:strongF}. 
We continue to extend the DG discretization in \cref{eq:dgbulk} to a discretization of the reduced interface models~I, I-R, II, and~II-R from \Cref{sec:sec4}. 
Here, the models~I and I-R and the models~II and II-R can be treated together, respectively, since they only differ in their bulk geometry with otherwise identical weak formulation. 

We define the bilinear forms~$\mathcal{A}_{h}^{\Gamma_1} \colon \Phi_h^\Gamma \times \Phi_h^\Gamma \rightarrow \mathbb{R}$, $\mathcal{A}_{h}^{\Gamma_2} \colon \Phi_h^\mathrm{b} \times \Phi_h^\Gamma$, and $\mathcal{I}_h \colon \Phi_h \times \Phi_h \rightarrow \mathbb{R}$, as well as the linear form~$\mathcal{R}_h^\Gamma \colon \Phi_h^\Gamma \rightarrow \mathbb{R}$, by 
\begin{subequations}
\begin{align}
\begin{split}
\label{eq:BGamma1} \mathcal{A}^{\Gamma_1}_h (p_h^\Gamma , \phi_h^\Gamma ) &= \sum_{F\in F_h^\Gamma} \int_F \matr{K}_\Gamma \nabla ( d p_h^\Gamma ) \cdot \nabla \phi_h^\Gamma \,\mathrm{d} \sigma 
+ \sum_{E \in \mathcal{E}^\circ_h } \int_E \mu^\Gamma_E \jump{p_h^\Gamma } \cdot  \jump{\phi_h^\Gamma } \,\mathrm{d}r \\
&\quad - \sum_{E\in\mathcal{E}^\circ_h} \int_E \Big[ \jump{\phi_h^\Gamma} \cdot  \avg{\matr{K}_\Gamma \nabla (d p_h^\Gamma )} + \jump{dp_h^\Gamma} \cdot  \avg{\matr{K}_\Gamma \nabla \phi_h^\Gamma}\Big] \, \mathrm{d} r \\
&\mkern-16mu + \sum_{E\in \mathcal{E}_h^\partial} \bigg[ \int_E \mu^\Gamma_E p_h^\Gamma \phi_h^\Gamma \,\mathrm{d}r 
- \int_E \Big[ \phi_h^\Gamma \matr{K}_\Gamma \nabla (dp_h^\Gamma ) - dp_h^\Gamma \matr{K}_\Gamma \nabla \phi_h^\Gamma \Big] \cdot \mathrm{d} \vct{r} \bigg] ,
\end{split} \\
\begin{split}
\label{eq:BGamma2} \mathcal{A}^{\Gamma_2}_h (p_h^\mathrm{b} , \phi_h^\Gamma ) &= 
-\!\sum_{F\in F_h^\Gamma} \int_F \Big[ p_h^{(1)}\nabla d_1 + p_h^{(2)}\nabla d_2\Big] \!\cdot \matr{K}_\Gamma \nabla \phi_h^\Gamma \,\mathrm{d}\sigma \\
&\mkern-48mu + \!\sum_{E\in\mathcal{E}_h^\circ} \int_E \avg{p_h^\mathrm{b}} \jump{\phi_h^\Gamma} \cdot \matr{K}_\Gamma  \nabla d \, \mathrm{d} r 
 + \!\sum_{E\in\mathcal{E}^\partial_h} \int_E \phi_h^\Gamma \big[ p_h^{(1)} \nabla d_1 + p_h^{(2)} \nabla d_2 \big] \! \cdot \mathrm{d} \vct{r} ,
\end{split} \\
\begin{split}
\label{eq:Idg} \mathcal{I}_h (p_h , \phi_h ) &= \sum_{F\in\mathcal{F}_h^\Gamma} \int_F \frac{K_\Gamma^\perp}{d} \jump{p_h^\mathrm{b}} \cdot \jump{\phi_h^\mathrm{b}} \,\mathrm{d}\sigma  \\
&\qquad + \sum_{F\in\mathcal{F}_h^\Gamma} \int_F \beta_\Gamma  \big( p^\Gamma_h - \avg{p^\mathrm{b}_h} \big)   \big( \phi^\Gamma_h - \avg{\phi^\mathrm{b}_h} \big) \,\mathrm{d} \sigma ,
\end{split} \\
\mathcal{R}^{\Gamma}_h (\phi_h) = &\!\sum_{F\in F_h^\Gamma} \!\int_F q_\Gamma \phi^\Gamma_h \,\mathrm{d}\sigma + \!\sum_{E\in \mathcal{E}_h^\partial} \!\bigg[  \!\int_E \mu_\Gamma g_\Gamma \phi_h^\Gamma \,\mathrm{d}r - \!\int_E d g_\Gamma \matr{K}_\Gamma \nabla \phi_h^\Gamma \cdot \mathrm{d}\vct{r} \bigg] ,
\end{align}
\end{subequations}
where $p_h = (p_h^\mathrm{b} , p_h^\Gamma) , \phi_h = (\phi_h^\mathrm{b} , \phi_h^\Gamma ) \in \Phi_h$ with $p_h^\mathrm{b} = \smash{\big( p_h^{(1)}, p_h^{(2)}\big)} \in \Phi_h^\mathrm{b}$.
For the reduced models~I and~II without bulk rectification, the evaluation of bulk functions in~$\smash{\Phi_h^\mathrm{b}}$ on the interface~$\Gamma$ in the~\cref{eq:BGamma2,eq:Idg} is to be understood in the sense of restrictions to the interfaces~$\Gamma_1$ and~$\Gamma_2$ as defined in~\cref{eq:interfacerestr}.
Besides, in \cref{eq:BGamma1}, $\smash{\mu^\Gamma_E}$ is a penalty parameter on the interface~$\Gamma$ that is defined in analogy to~\cref{eq:bulkpenalty}.

A DG discretization of the the reduced interface models~II and~II-R is now given by the following problem.
Find $p_h = (p_h^\mathrm{b} , p_h^\Gamma ) \in \Phi_h$ such that
\begin{align}
\mathcal{A}^\mathrm{b}_h (p_h^\mathrm{b} , \phi_h^\mathrm{b} ) + \mathcal{A}^{\Gamma_1}_h (p_h^\Gamma , \phi_h^\Gamma ) + \mathcal{I}_h (p_h , \phi_h )  = \mathcal{R}^\mathrm{b}_h (\phi_h^\mathrm{b} ) + \mathcal{R}^\Gamma_h (\phi_h^\Gamma )  
\label{eq:dgmartin}
\end{align}
holds for all $\phi_h = (\phi_h^\mathrm{b} , \phi_h^\Gamma ) \in \Phi_h$. 

Finally, a DG discretization of the reduced models~I and I-R extending the discretization in~\cref{eq:dgmartin} can be formulated as follows. 
Find $p_h = (p_h^\mathrm{b} , p_h^\Gamma ) \in \Phi_h$ so that
\begin{align}
\mathcal{A}^\mathrm{b}_h (p_h^\mathrm{b} , \phi_h^\mathrm{b} ) + \mathcal{A}^{\Gamma_1}_h (p_h^\Gamma , \phi_h^\Gamma ) + 
\mathcal{A}^{\Gamma_2}_h (p_h^\mathrm{b} , \phi_h^\Gamma ) + \mathcal{I}_h (p_h , \phi_h )  = \mathcal{R}^\mathrm{b}_h (\phi_h^\mathrm{b} ) + \mathcal{R}^\Gamma_h (\phi_h^\Gamma ) 
\label{eq:dgnew}%
\end{align}%
holds for all $\phi_h = (\phi_h^\mathrm{b} , \phi_h^\Gamma ) \in \Phi_h$.

\section{Numerical Results} \label{sec:sec6}
We present numerical results to validate the new reduced interface model~\eqref{eq:strong} and explore its capabilities.
In particular, we investigate how the use of a simplified bulk geometry and the negligence of aperture gradients~$\nabla d_1$,~$\nabla d_2$ affects the accuracy of the reduced model~\cref{eq:strong}. 
For this, numerical solutions of the reduced models I, I-R, II, and~II-R from \Cref{sec:sec41} are compared with a numerical reference solution of the full-dimensional model~\eqref{eq:darcydecomposed}. 
Specifically, for the full-dimensional reference solution, the average pressure~$p_\Gamma^\mathrm{ref}$ across the fracture is computed according to~\cref{eq:pGamma}. 
Then, the different reduced models are assessed in terms of their solution for the effective pressure~$p_\Gamma$ inside the fracture and its deviation from the averaged reference solution~$p_\Gamma^\mathrm{ref}$, particularly, by calculating the discrete $L^2$-error over the interface~$\Gamma$.

All subsequent test problems are performed on the computational domain~$\Omega = (0,1)^n  \subset \mathbb{R}^n$ with $n = 2$ or $n=3$ and feature a single fracture with sinusoidal aperture that is represented by the interface~$\Gamma = \{ \vct{x} \in \Omega \ \vert\ x_1 = \tfrac{1}{2} \}$ in the reduced model~\eqref{eq:strong} and its variants.
For the reduced models, the coupling parameter~$\xi$ is chosen as~$\xi = \frac{2}{3}$ as suggested by the derivation in \Cref{sec:sec3}.
Further, all test problems feature a vanishing source term~$q\equiv 0$ so that the flow is determined only by the choice of boundary conditions.
In addition, the bulk permeability is defined as $\matr{K}_1 = \matr{K}_2 \equiv \matr{I}$, where $\matr{I} \in \mathbb{R}^{n\times n}$ denotes the identity matrix. 
The fracture permeability~$\matr{K}_\mathrm{f}$ differs depending on the test case. 

The results in this section were obtained from an implementation of the DG schemes~\eqref{eq:dgbulk}, \eqref{eq:dgmartin}, and \eqref{eq:dgnew} in \texttt{DUNE}~\cite{bastian21}. 
The program code is openly available~(see \cite{hoerl22b} and corresponding repository\footnote{\url{https://github.com/maximilianhoerl/mmdgpy/tree/paper}}).
Specifically, the implementation relies on~\texttt{DUNE-MMesh}~\cite{burbulla22}, a grid module tailored for applications with interfaces. 
In particular, \texttt{DUNE-MMesh} is a useful tool for mixed-dimensional models, such as the DG schemes~\eqref{eq:dgmartin} and~\eqref{eq:dgnew}, as it allows to export a predefined set of facets from the bulk grid as separate interface grid and provides coupled solution strategies to simultaneously solve bulk and interface schemes. 
Further, the implementation depends on \texttt{DUNE-FEM}~\cite{dedner10}, a discretization module providing the capabilities to implement efficient solvers for a wide range of partial differential equations, which we access through its Python interface, where, using the Unified Form Language (\texttt{UFL})~\cite{alnaes14}, the description of models is close to their variational formulation. 

\subsection{Flow Perpendicular to a Fracture with Constant Total Aperture}  \label{sec:perp_asym}
\subsubsection{Two-Dimensional Test Problem} \label{sec:perp_asym_2d}
For the first test problem, we consider a fracture with a serpentine geometry in the two-dimensional domain~$\Omega = (0,1)^2$. 
Nonetheless, the fracture is chosen such that it exhibits a constant total aperture~$d$. 
Specifically, we define the aperture functions~$d_1$ and~$d_2$ by
\begin{align}
d_1 ( x_2 ) &= d_0 + \tfrac{1}{2}  d_0 \sin (8  \pi  x_2 ) , \qquad d_2 ( x_2 ) = d_0 - \tfrac{1}{2}  d_0  \sin (8  \pi  x_2 ) ,
\end{align}
where $d_0 >0$ is a free parameter. 
Then, the total aperture is constant and given by~$d(x_2) = 2d_0$. 
Further, on the whole boundary~$\partial \Omega$, we impose Dirichlet conditions and require the pressure~$p$ to be equal to $g (\vct{x} ) = 1 - x_1$.
Thus, the flow direction will be from left to right, perpendicular to the fracture. 
Besides, for the full-dimensional model~\eqref{eq:darcydecomposed}, the permeability inside the fracture is defined by $\matr{K}_\mathrm{f} = \tfrac{1}{2}\matr{I}$.
As a consequence, the effective fracture permeabilities in the reduced model~\eqref{eq:strong} and its variants are given by $\matr{K}_\Gamma = \tfrac{1}{2}\matr{I}$ and $K_\Gamma^\perp = \frac{1}{2}$.
In particular, the fracture is less permeable than the bulk domains~$\Omega_1$ and~$\Omega_2$.
The fracture geometry and the resulting full-dimensional solution are illustrated in \Cref{fig:perp_asym} for the case of~$d_0 = 10^{-1}$.
\begin{figure}[tbh]
\centering
\begin{subfigure}{0.495\textwidth}
\centering
\includegraphics[height=0.75\textwidth]{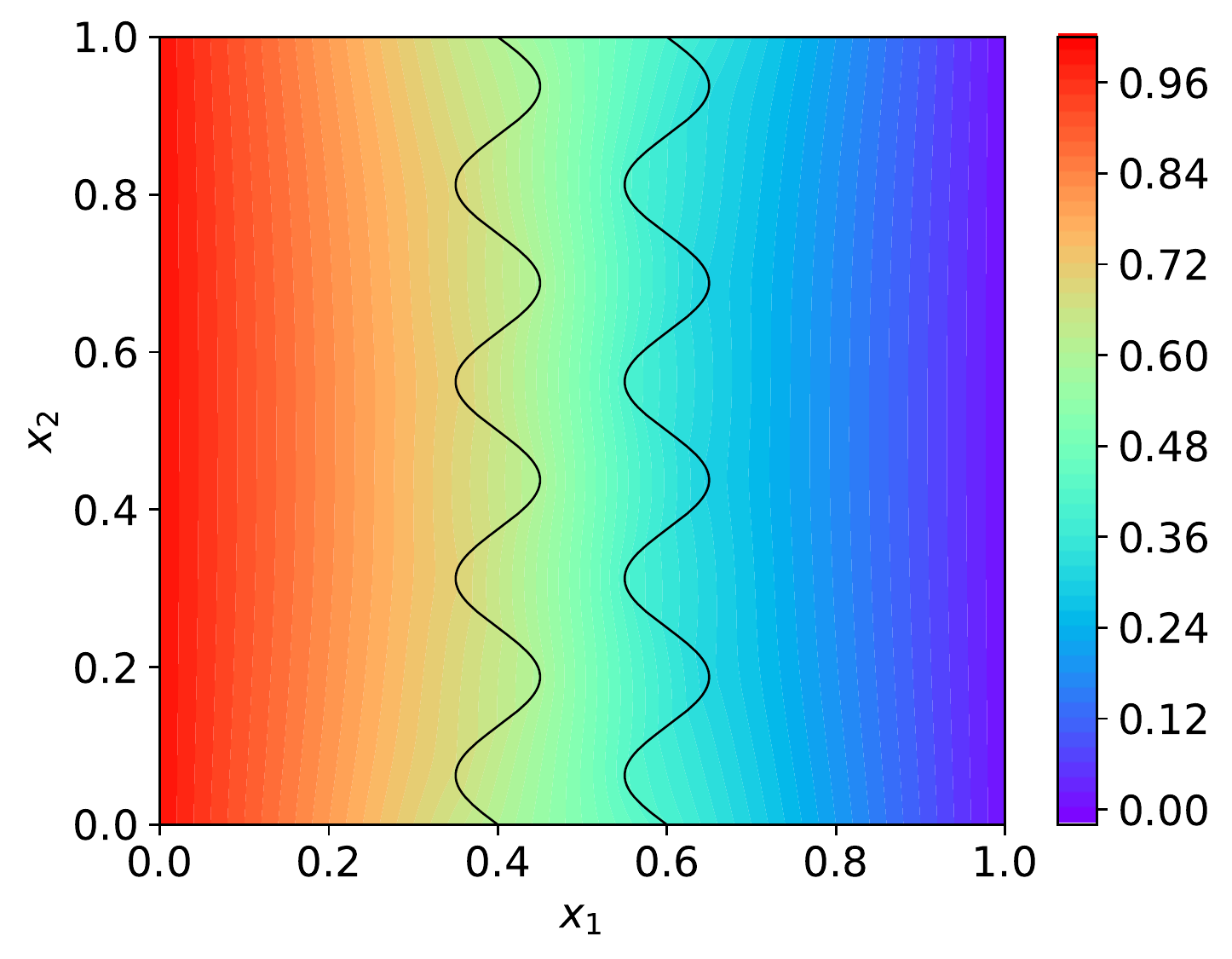}
\end{subfigure}
\hfill
\begin{subfigure}{0.495\textwidth}
\centering
\includegraphics[height=0.75\textwidth]{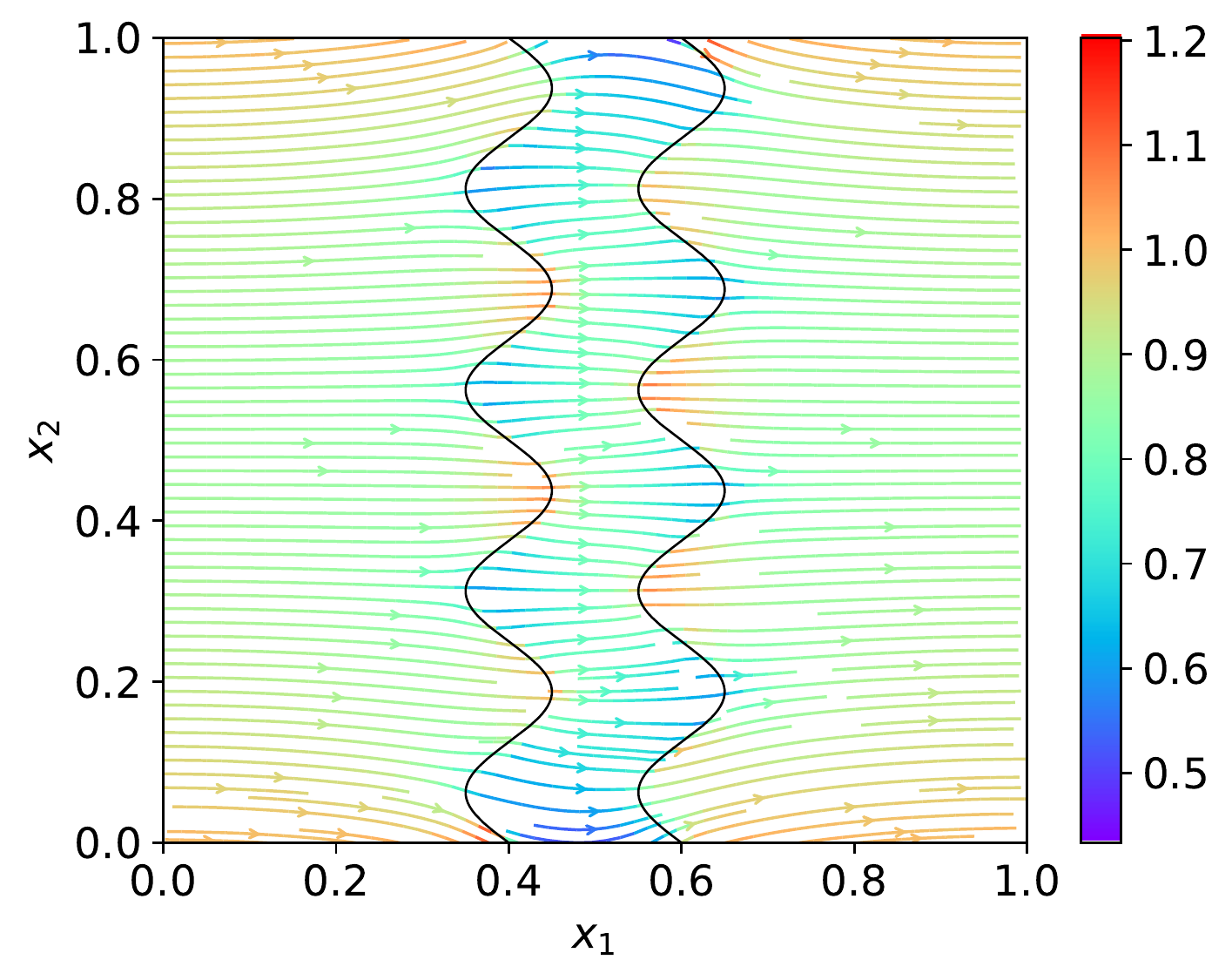}
\end{subfigure}
\caption{Full-dimensional numerical reference solution (\Cref{sec:perp_asym_2d}) for the pressure~$p$ (left) and the velocity~$-\matr{K}\nabla p$ (right) for the case of~$d_0 = 10^{-1}$.}
\label{fig:perp_asym}
\end{figure}

\Cref{fig:perp_asym_error} shows the DG solutions for the effective pressure~$p_\Gamma$ in the reduced models~I, I-R, II, and~II-R in comparison to the full-dimensional reference solution~$p_\Gamma^\mathrm{ref}$. 
On the one hand, the effective pressure~$p_\Gamma = p_\Gamma (x_2 )$ is plotted for a fixed valued of~$d_0 = 10^{-1}$, where one can see a clear difference between the solutions of the various reduced models. 
As expected, model~I performs best, while model~II-R performs worst. 
On the other hand, \Cref{fig:perp_asym_error} also displays the $L^2$-error~$\smash{\norm{p_\Gamma - p_\Gamma^\mathrm{ref}}_{L^2(\Gamma )}}$ as function of the aperture parameter~$d_0$, where, again, the solution of model~I sticks out as the most accurate. 
In \Cref{fig:perp_asym_error}, the $L^2$-errors of the reduced models~I-R and~II-R show a similar behavior. 
In particular, their convergence towards the reference solution~$p_\Gamma^\mathrm{ref}$ for a decreasing aperture is considerably slower than the convergence for the models~I and~II. 
Thus, in this test problem, it is primarily the rectification of the bulk domains~$\Omega_1$ and~$\Omega_2$ that negatively affects the model error and rate of convergence with respect to a decreasing aperture.
However, comparing the solutions of model~I and model~II, there is also an undeniable effect on the accuracy of the solution in connection with the inclusion of aperture gradients~$\nabla d_1$, $\nabla d_2$ in~\cref{eq:strongB}.
For small apertures, the error of model~I seems to stagnate.
This is attributed to numerical errors in the computation of the full-dimensional reference solution~$p_\Gamma^\mathrm{ref}$ and discussed in greater detail in \Cref{sec:perp_sym}.
\begin{figure}[tbh]
\centering
\includegraphics[width=0.8\linewidth]{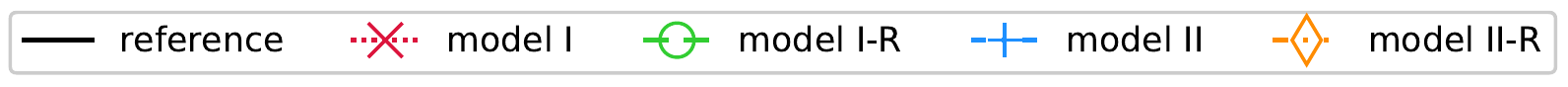}
\begin{subfigure}{.49\textwidth}
  \centering
  \includegraphics[width=\linewidth]{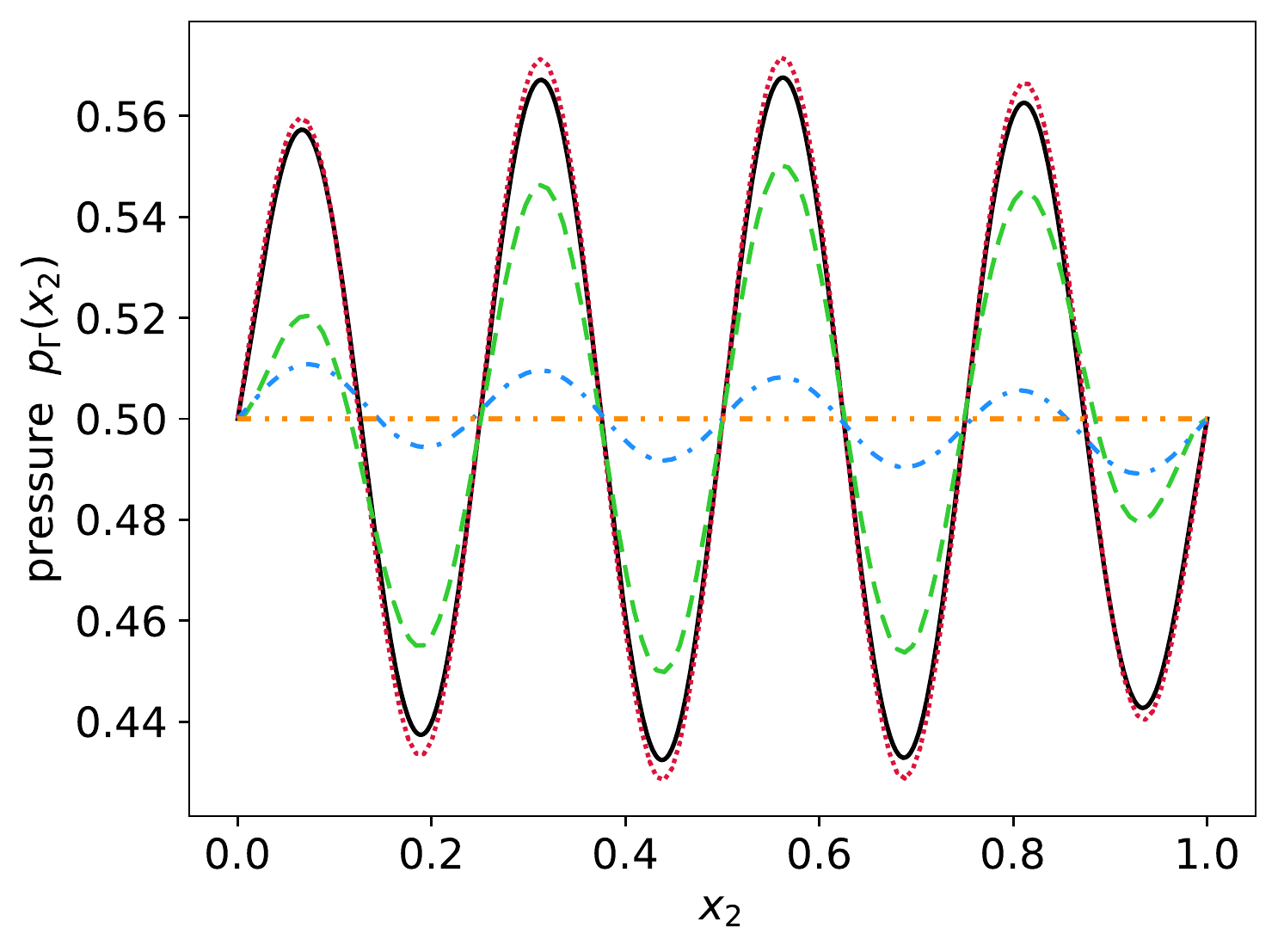}
\end{subfigure}\hfill%
\begin{subfigure}{.49\textwidth}
  \centering
  \includegraphics[width=\linewidth]{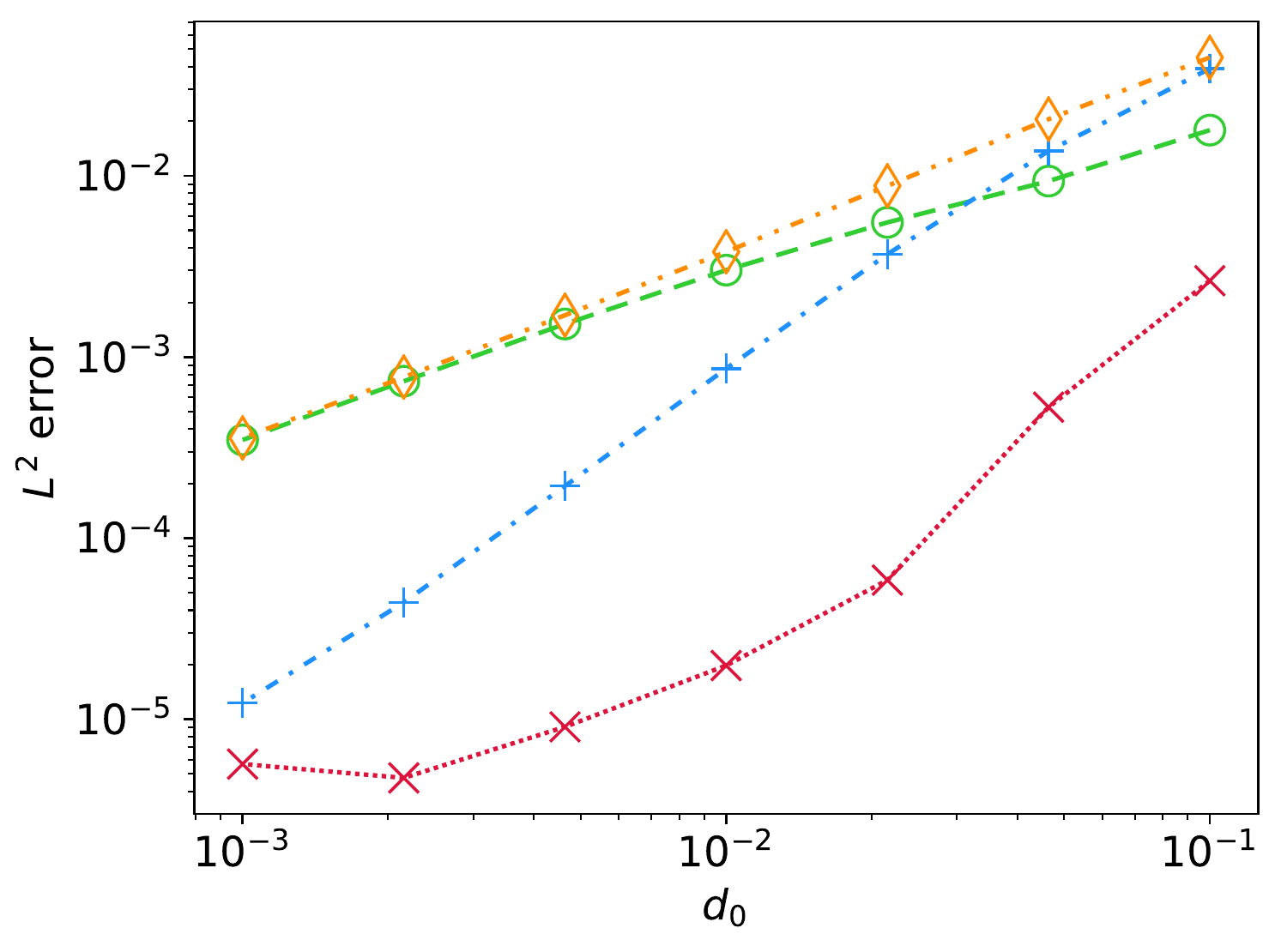}
\end{subfigure}
  \caption{Numerical solutions (\Cref{sec:perp_asym_2d}) for the effective pressure~$p_\Gamma$ for the reduced model~\eqref{eq:strong} and its variants  in comparison to the numerical reference solution~$\smash{p_\Gamma^\mathrm{ref}}$ for $d_0 = 10^{-1}$~(left) and $L^2$-error~$\Vert p_\Gamma - p_\Gamma^\mathrm{ref} \Vert_{L^2(\Gamma )}$ as function of~$d_0$~(right).}
 \label{fig:perp_asym_error}
\end{figure}

\subsubsection{Three-Dimensional Test Problem} \label{sec:perp_asym_3d}
Next, we extend the test problem from \Cref{sec:perp_asym_2d} to the three-dimensional case.
For this, we define the aperture functions~$d_1$ and~$d_2$ by
\begin{subequations}
\begin{align}
d_1 ( x_2, x_3 ) &= d_0 + \tfrac{1}{2}  d_0 \big( \sin (8  \pi  x_2 ) + \sin ( 8 \pi x_3 ) \big), \\
d_2 ( x_2, x_3 ) &= d_0 - \tfrac{1}{2}  d_0  \big( \sin (8  \pi  x_2 ) + \sin ( 8 \pi x_3 ) \big) ,
\end{align}
\end{subequations}
with a parameter~$d_0 > 0$ so that the total aperture~$d (x_2, x_3 ) = 2d_0$ is constant.
The resulting geometry is illustrated in \Cref{fig:perp_asym_3d_bulk}.
The permeability and boundary conditions are defined as in \Cref{sec:perp_asym_2d}.
\begin{figure}[tbh]
\begin{minipage}[b]{0.47\textwidth}
\centering
 \includegraphics[height=0.7\textwidth]{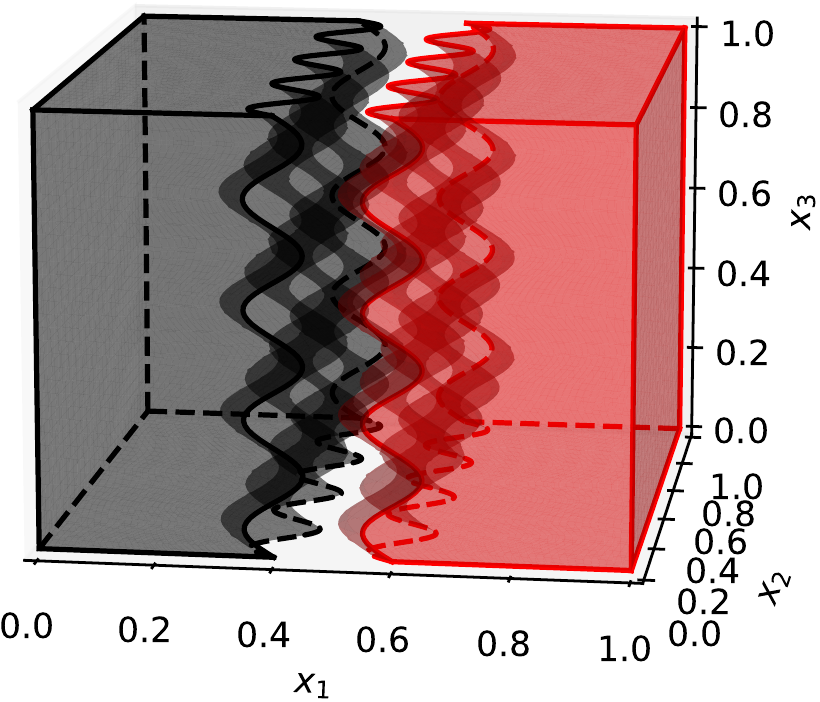}
\caption{Bulk domains $\Omega_1$ (black) and $\Omega_2$~(red) (\Cref{sec:perp_asym_3d}) for $d_0 = 10^{-1}$.}
\label{fig:perp_asym_3d_bulk}
\end{minipage}
\hfill
\begin{minipage}[b]{0.495\textwidth}
\centering
 \includegraphics[width=0.85\textwidth]{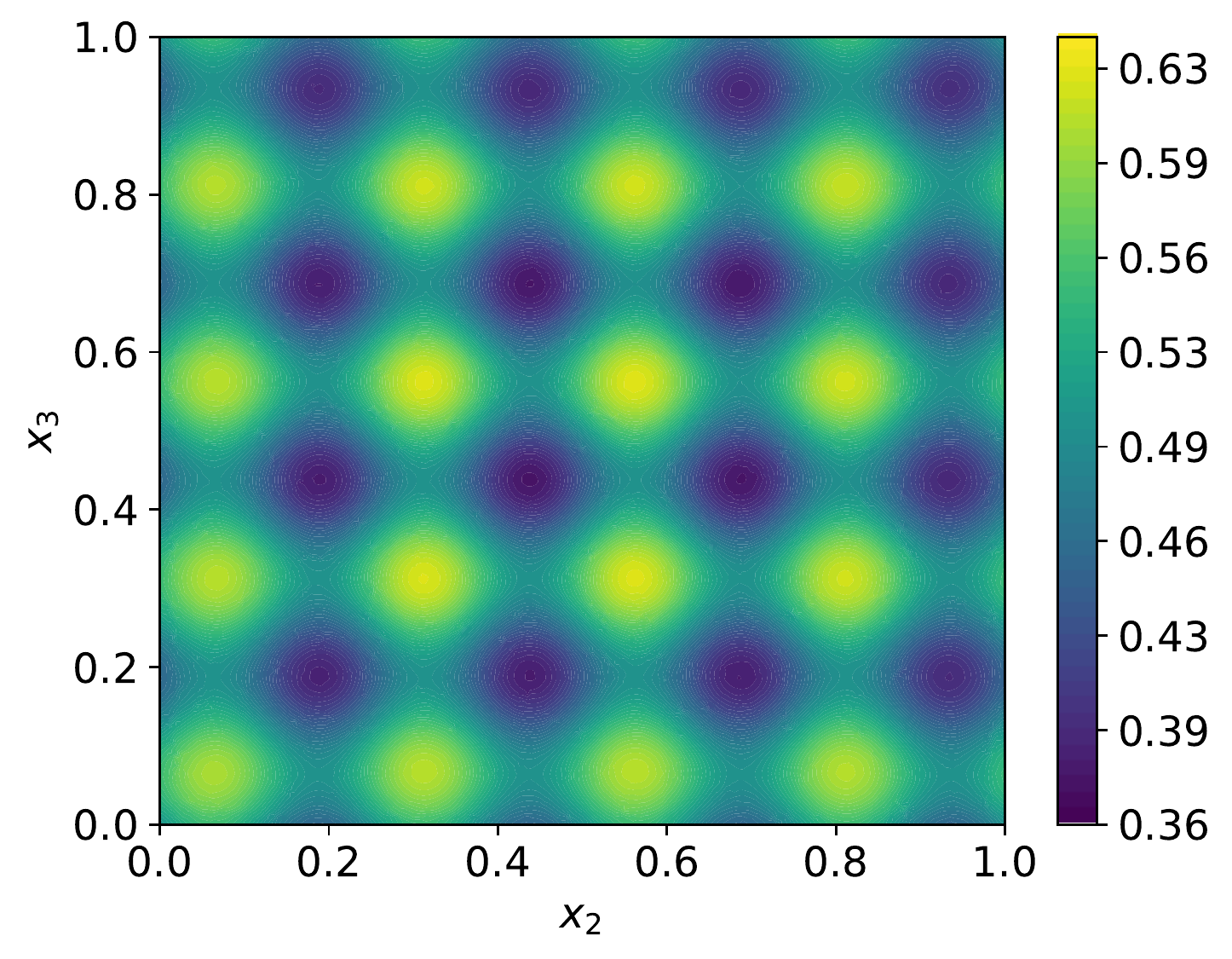}
\caption{Numerical reference solution~$p_\Gamma^\mathrm{ref}$ (\Cref{sec:perp_asym_3d}) for $d_0 = 10^{-1}$.}
\label{fig:perp_asym_3d_ref}
\end{minipage}
\end{figure}
\begin{figure}[tbh]
\begin{subfigure}{.495\textwidth}
  \centering
  \includegraphics[width=0.85\linewidth]{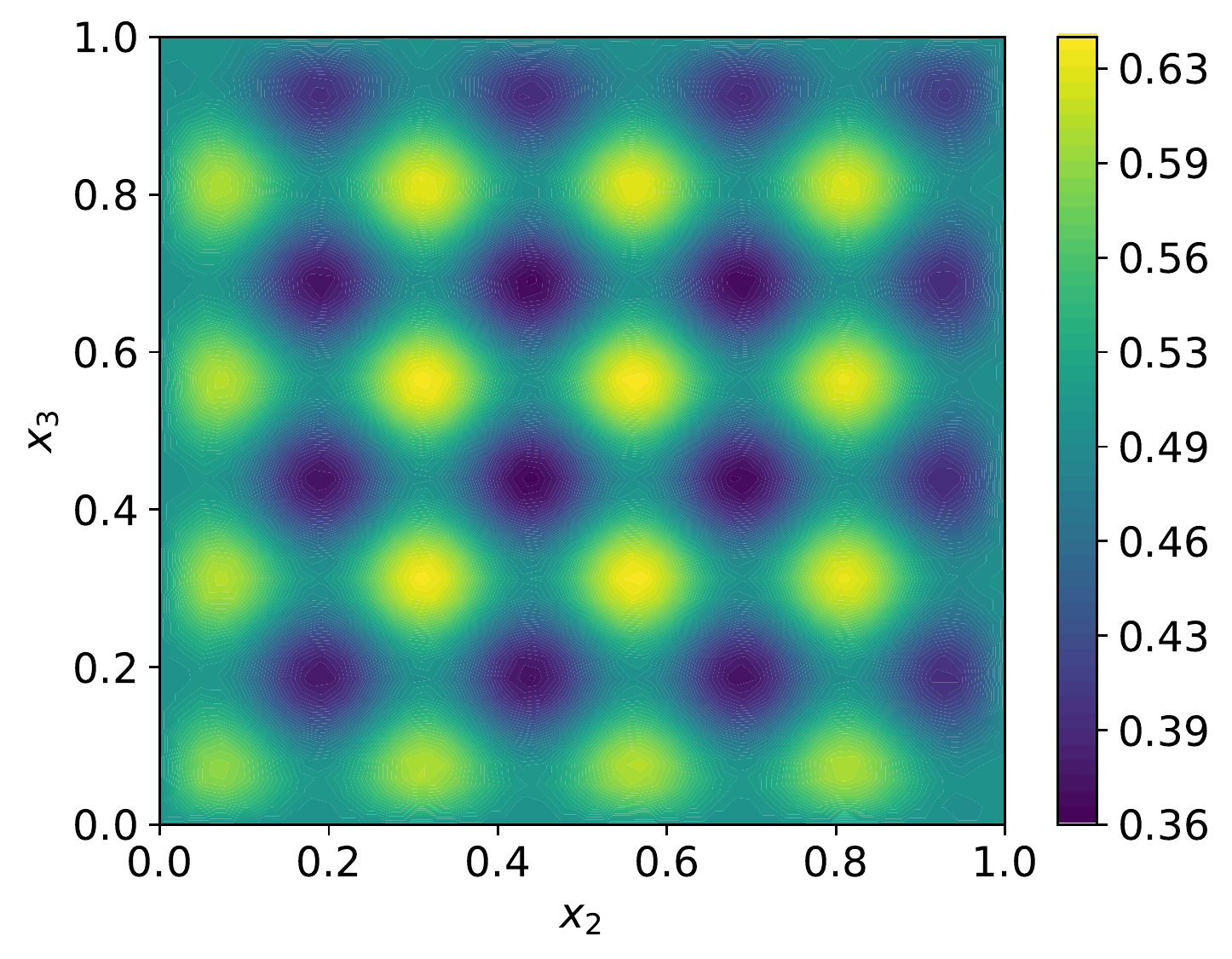}
\end{subfigure}\hfill%
\begin{subfigure}{.495\textwidth}
  \centering
  \includegraphics[width=0.85\linewidth]{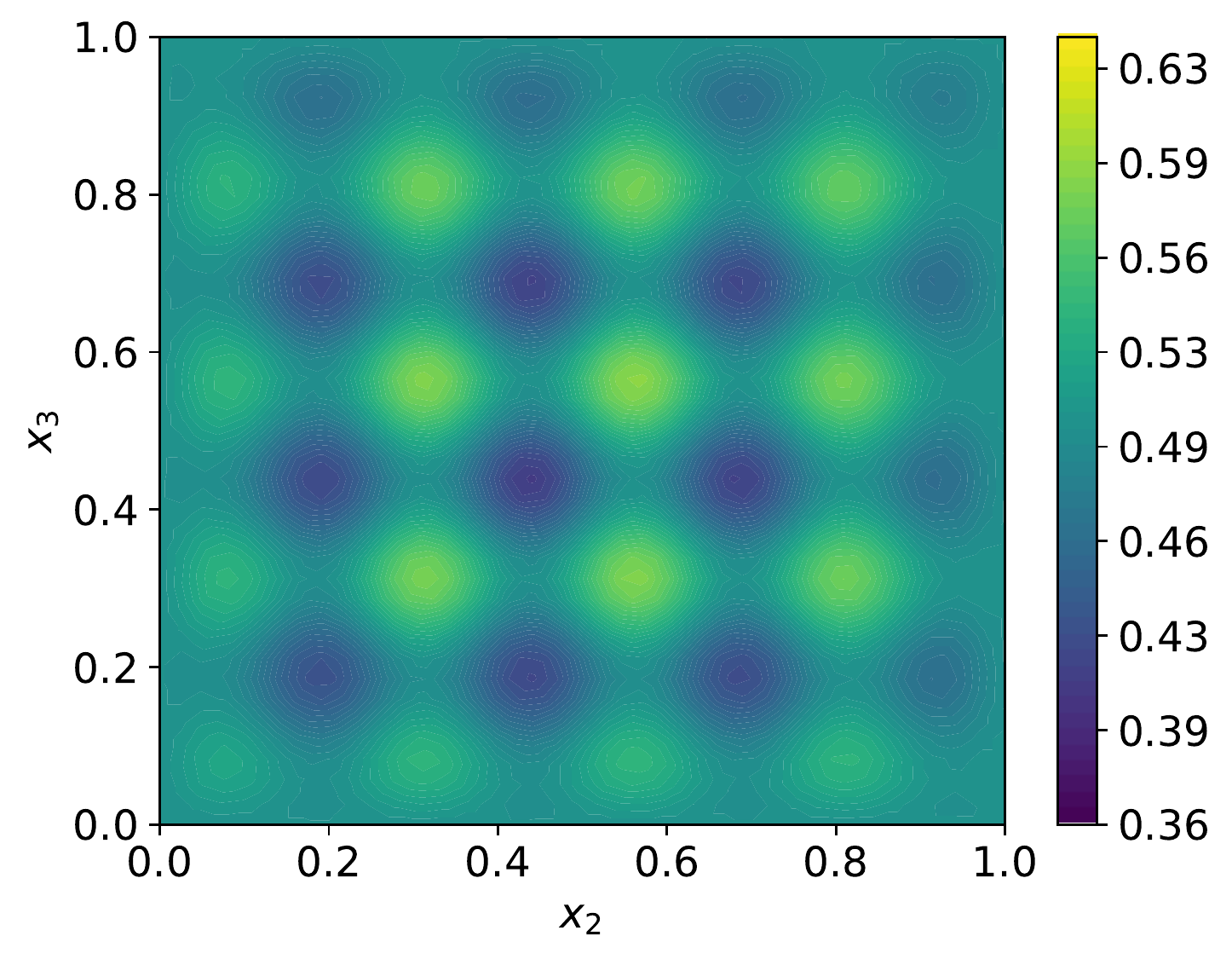}
\end{subfigure}
\begin{subfigure}{.495\textwidth}
  \centering
  \includegraphics[width=0.85\linewidth]{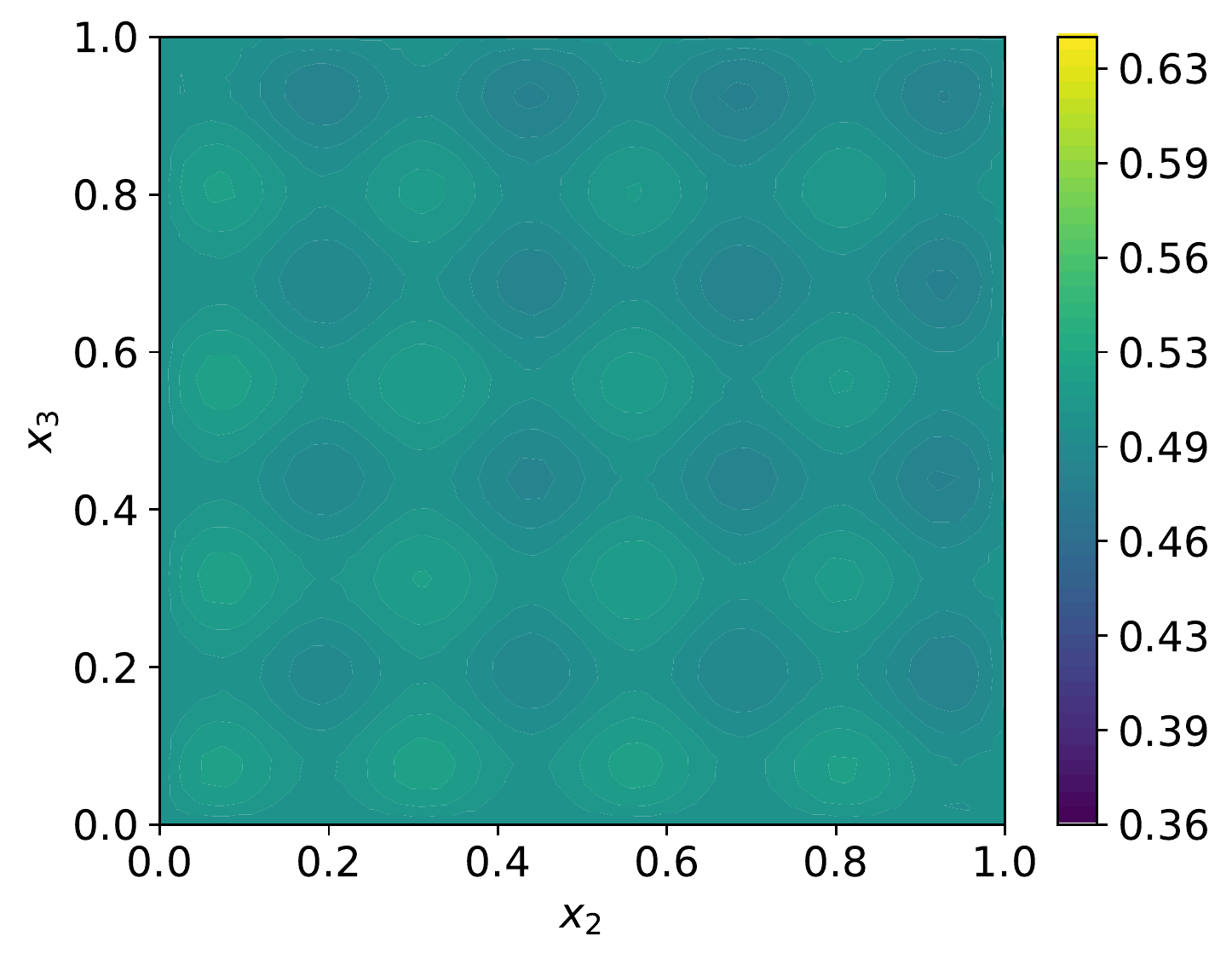}
\end{subfigure}\hfill%
\begin{subfigure}{.495\textwidth}
  \centering
  \includegraphics[width=0.85\linewidth]{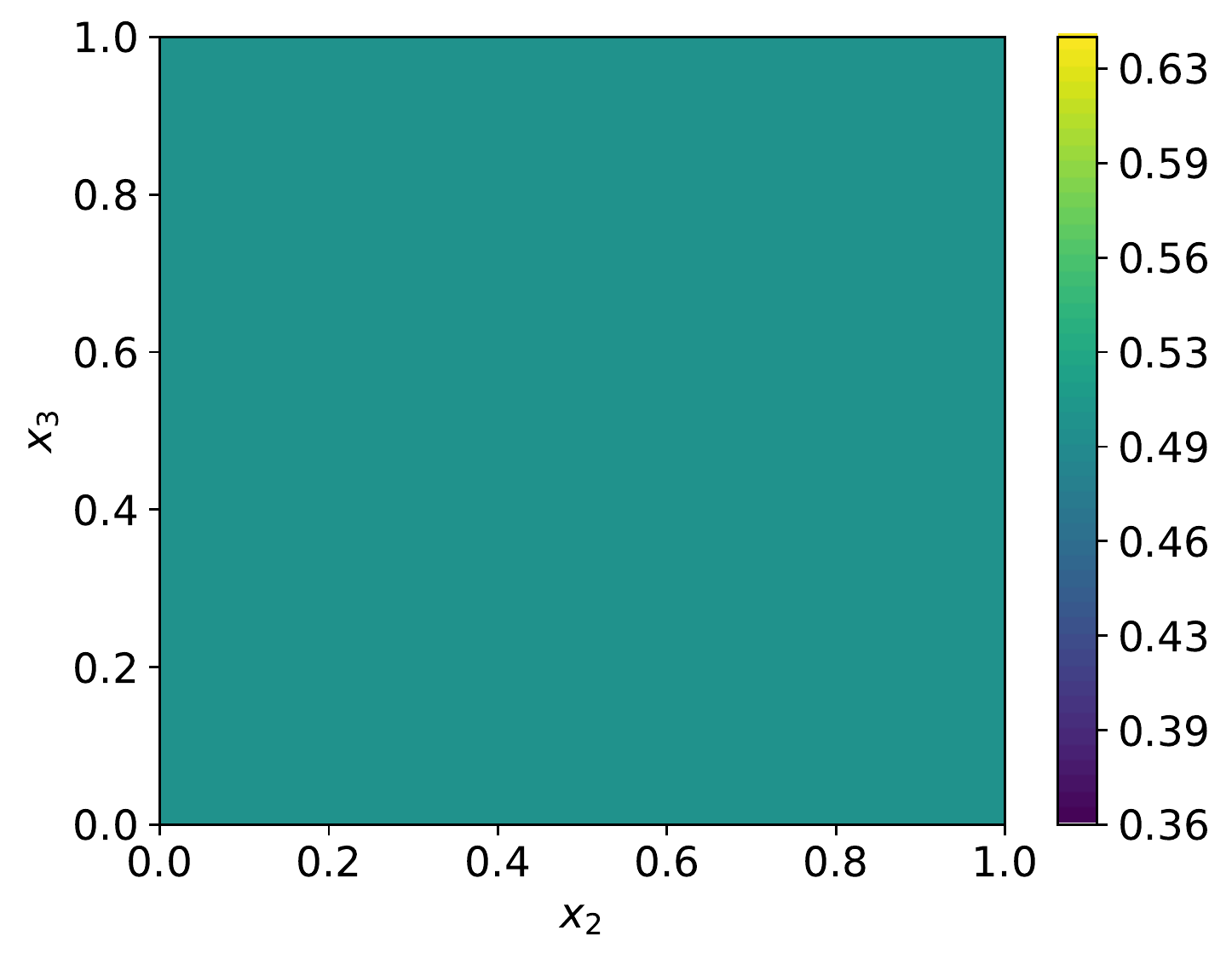}
\end{subfigure}
\caption{Numerical solutions (\Cref{sec:perp_asym_3d}) for the effective pressure~$p_\Gamma$ inside the fracture for the reduced model~\eqref{eq:strong} and its variants for $d_0 = 10^{-1}$: model~I~(top left), model~I-R~(top right), model~II~(bottom left), and model~II-R~(bottom right).}
 \label{fig:perp_asym_3d_reduced}
\end{figure}
\begin{figure}[htb]
\centering
\begin{minipage}{0.49\textwidth}
\centering
\includegraphics[width=\linewidth]{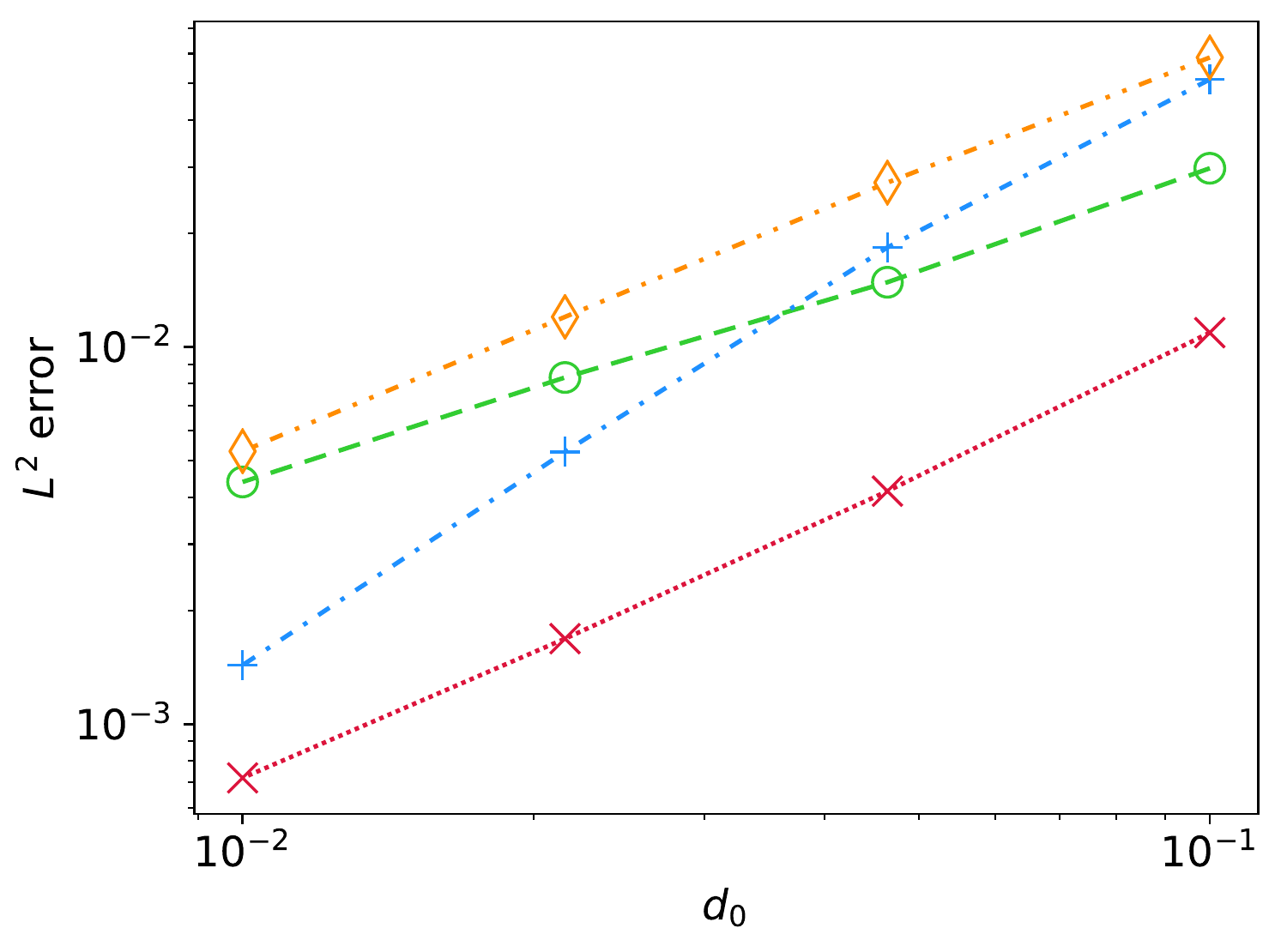}
\end{minipage} 
\quad
\begin{minipage}{0.3\textwidth}
\includegraphics[width=0.6\linewidth]{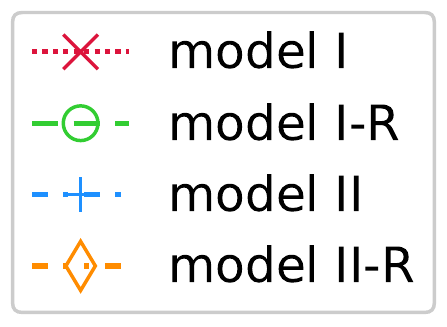}
\end{minipage}
\caption{$L^2$-error~$\smash{\Vert p_\Gamma - p_\Gamma^\mathrm{ref} \Vert_{L^2(\Gamma )}}$ as function of~$d_0$ (\Cref{sec:perp_asym_3d}).}
\label{fig:perp_asym_3d_error}
\end{figure}

\Cref{fig:perp_asym_3d_ref} displays the numerical reference solution for the averaged pressure~$p_\Gamma^\mathrm{ref}$ for~$d_0 = 10^{-1}$.
Besides, the DG solutions for the effective pressure~$p_\Gamma$ in the reduced model~\eqref{eq:strong} and its variants are shown in \Cref{fig:perp_asym_3d_reduced}. 
Here, in comparison with the reference solution in \Cref{fig:perp_asym_3d_ref}, a behavior analogous to the two-dimensional case in \Cref{fig:perp_asym_error} becomes apparent. 
The solution of model~I matches well with the reference solution, whereas the solutions of the models~I-R and~II reproduce the sine-like pattern of the reference solution at a too low amplitude and the solution of model~II-R is virtually constant.
This is also reflected by the $L^2$-error~$\smash{\norm{p_\Gamma - p_\Gamma^\mathrm{ref}}_{L^2(\Gamma )}}$, which is displayed in \Cref{fig:perp_asym_3d_error} as function of the aperture parameter~$d_0$. 
The $L^2$-error shows the same trends for a declining aperture as observed in the two-dimensional case. 

\subsection{Flow Perpendicular to an Axisymmetric Fracture} \label{sec:perp_sym}
For this next test problem, we again consider a sinusoidal fracture in the two-dimensional domain~$\Omega = (0,1)^2$,  however, this time with a non-constant total aperture~$d$. 
More particularly, the aperture functions~$d_1$ and~$d_2$ are defined by 
\begin{align}
d_1 ( x_2 ) &= d_0 + \tfrac{1}{2}  d_0 \sin (8  \pi  x_2 ) , \qquad  d_2 ( x_2 ) = d_0 + \tfrac{1}{2}  d_0  \sin (8  \pi  x_2 ) \label{eq:d12_sym}
\end{align}
with a free parameter~$d_0 > 0$.
Thus, the interface~$\Gamma$ is the center line of an axisymmetric fracture and the total aperture~$d$ ranges between~$d_0$ and~$3d_0$. 
In addition, the permeability~$\matr{K}_\mathrm{f}$ inside the fracture and the given pressure~$g$ at the external boundary~$\partial \Omega$ are defined as in \Cref{sec:perp_asym_2d}.
The fracture geometry and full-dimensional solution from the DG scheme~\eqref{eq:dgbulk} are shown in  \Cref{fig:perp_sym} for the case of~$d_0 = 10^{-1}$.
\begin{figure}[tbh]
\centering
\begin{subfigure}{0.495\textwidth}
\centering
\includegraphics[height=0.75\textwidth]{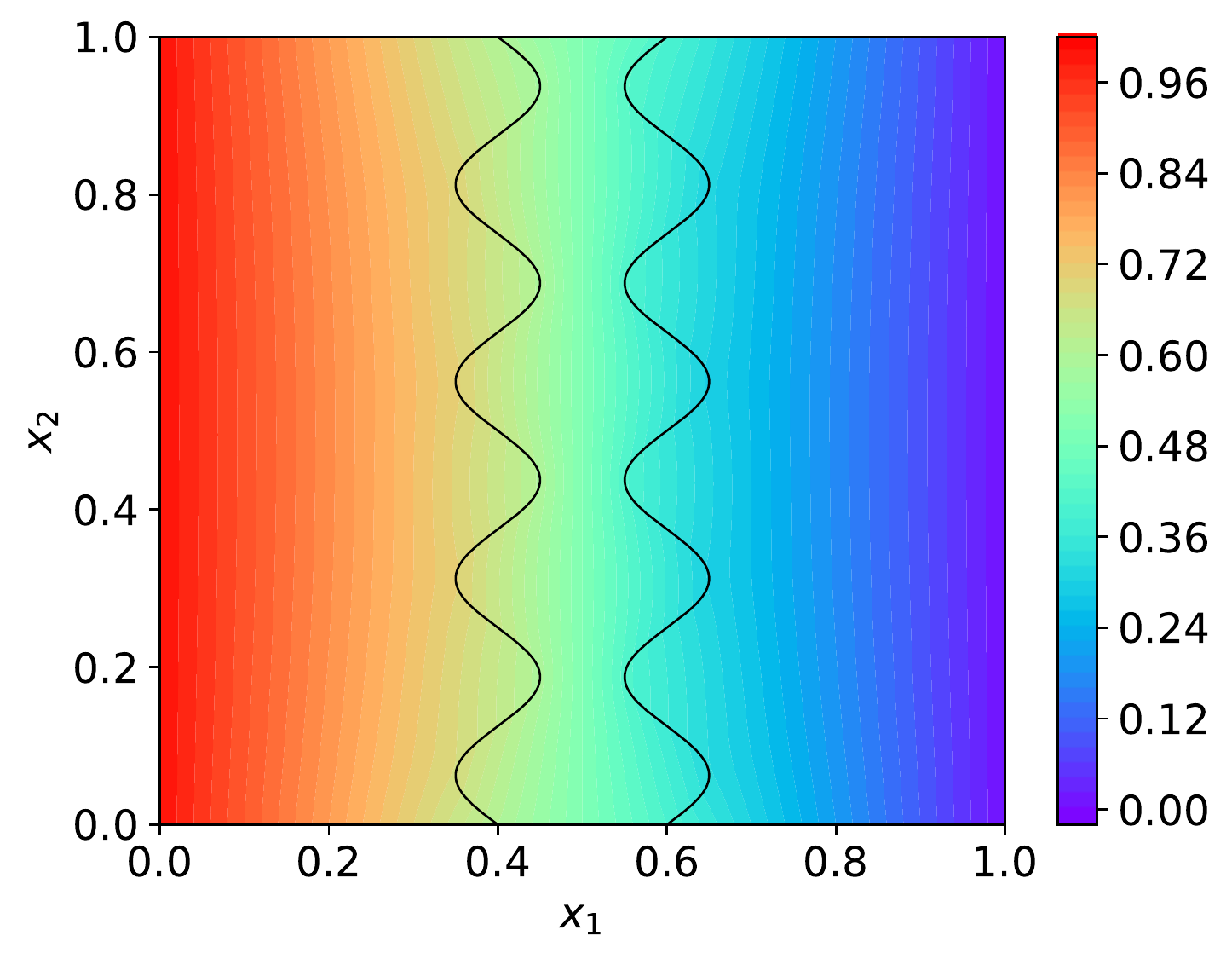}
\end{subfigure}
\hfill
\begin{subfigure}{0.495\textwidth}
\centering
\includegraphics[height=0.75\textwidth]{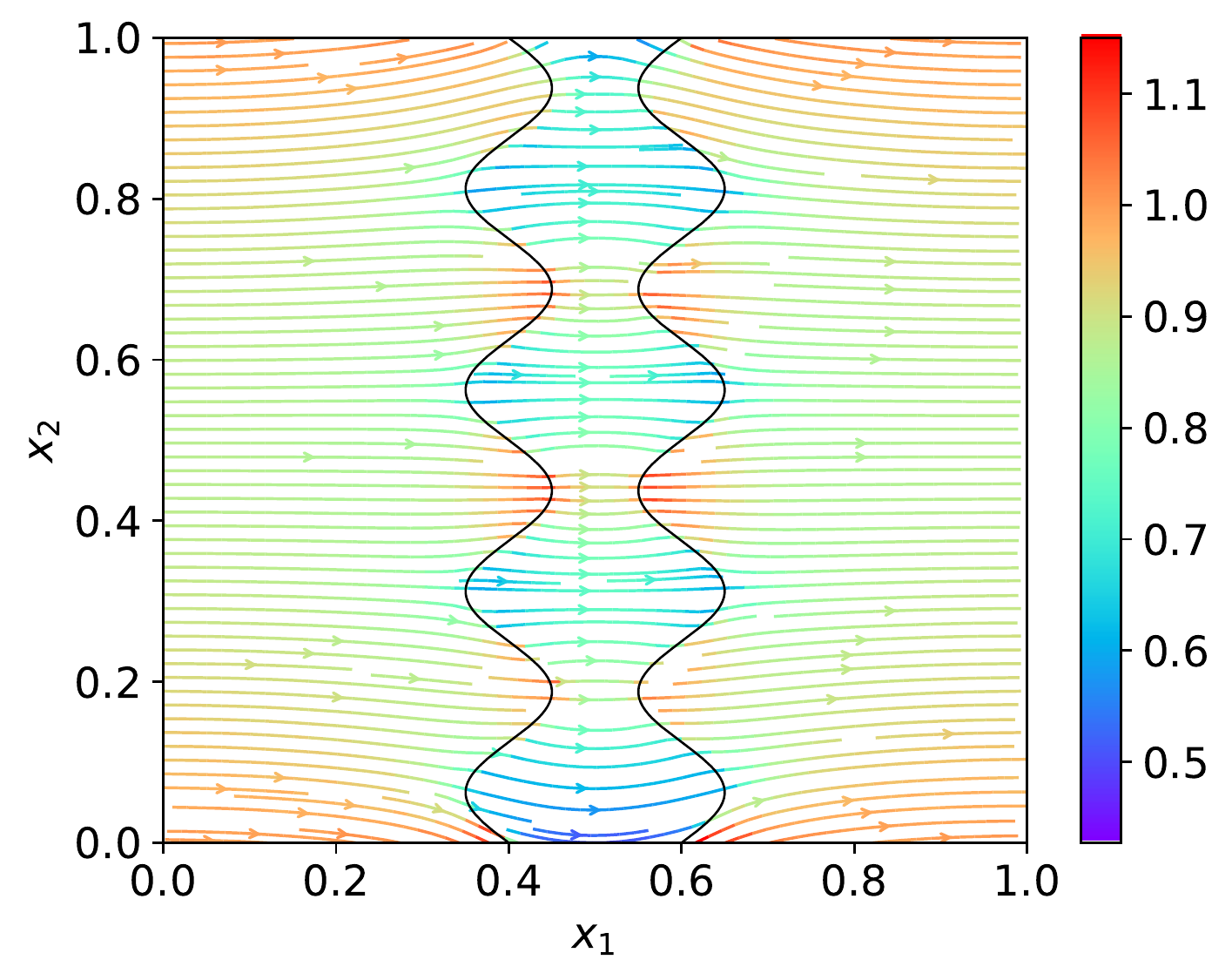}
\end{subfigure}
\caption{Full-dimensional numerical reference solution (\Cref{sec:perp_sym}) for the pressure~$p$ (left) and the velocity~$-\matr{K}\nabla p$ (right) for the case of~$d_0 = 10^{-1}$.}
\label{fig:perp_sym}
\end{figure}
\begin{figure}[tbh]
\centering
\includegraphics[width=0.8\linewidth]{media/legend.pdf}
\begin{subfigure}{.49\textwidth}
  \centering
  \includegraphics[width=\linewidth]{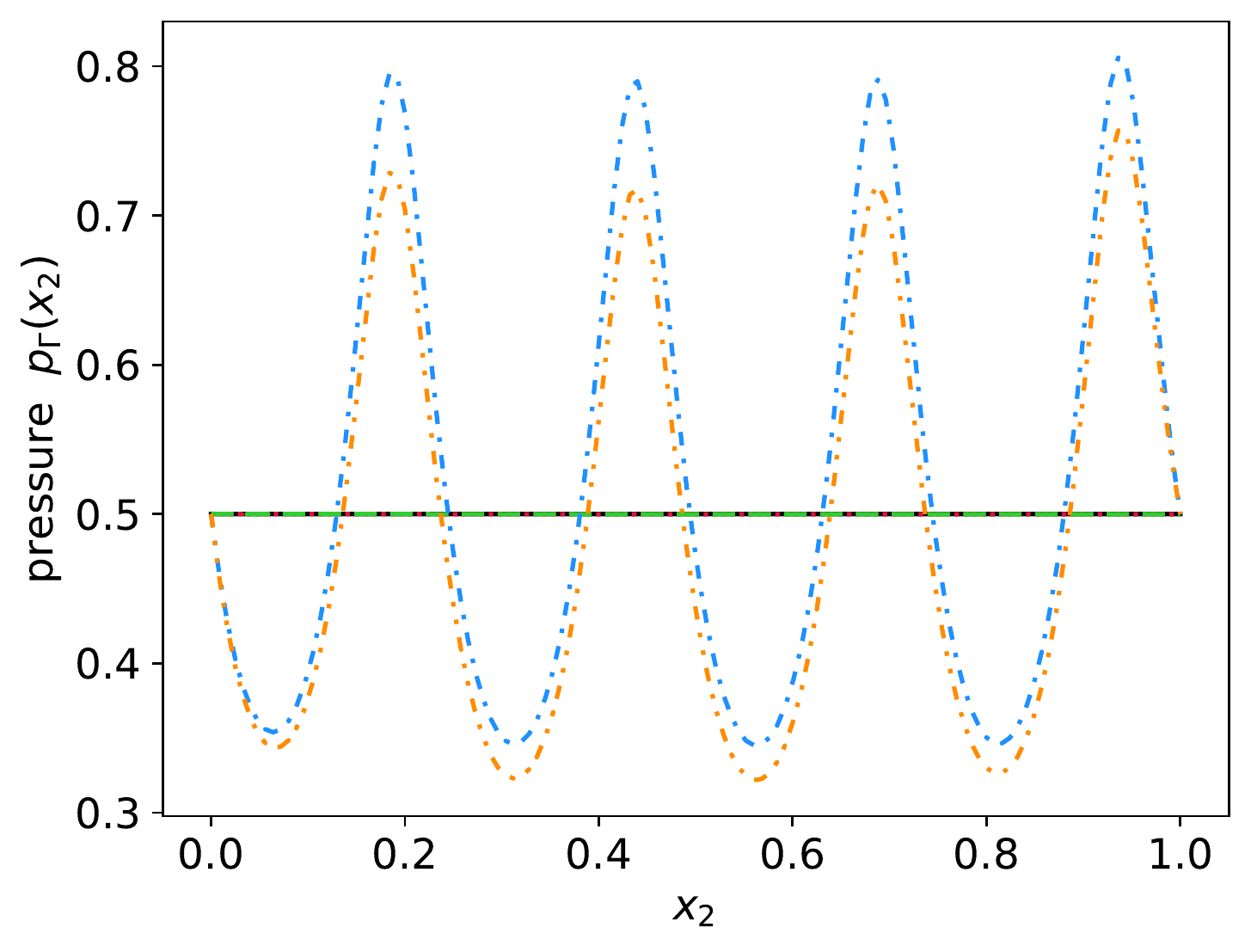}
\end{subfigure}\hfill%
\begin{subfigure}{.49\textwidth}
  \centering
  \includegraphics[width=\linewidth]{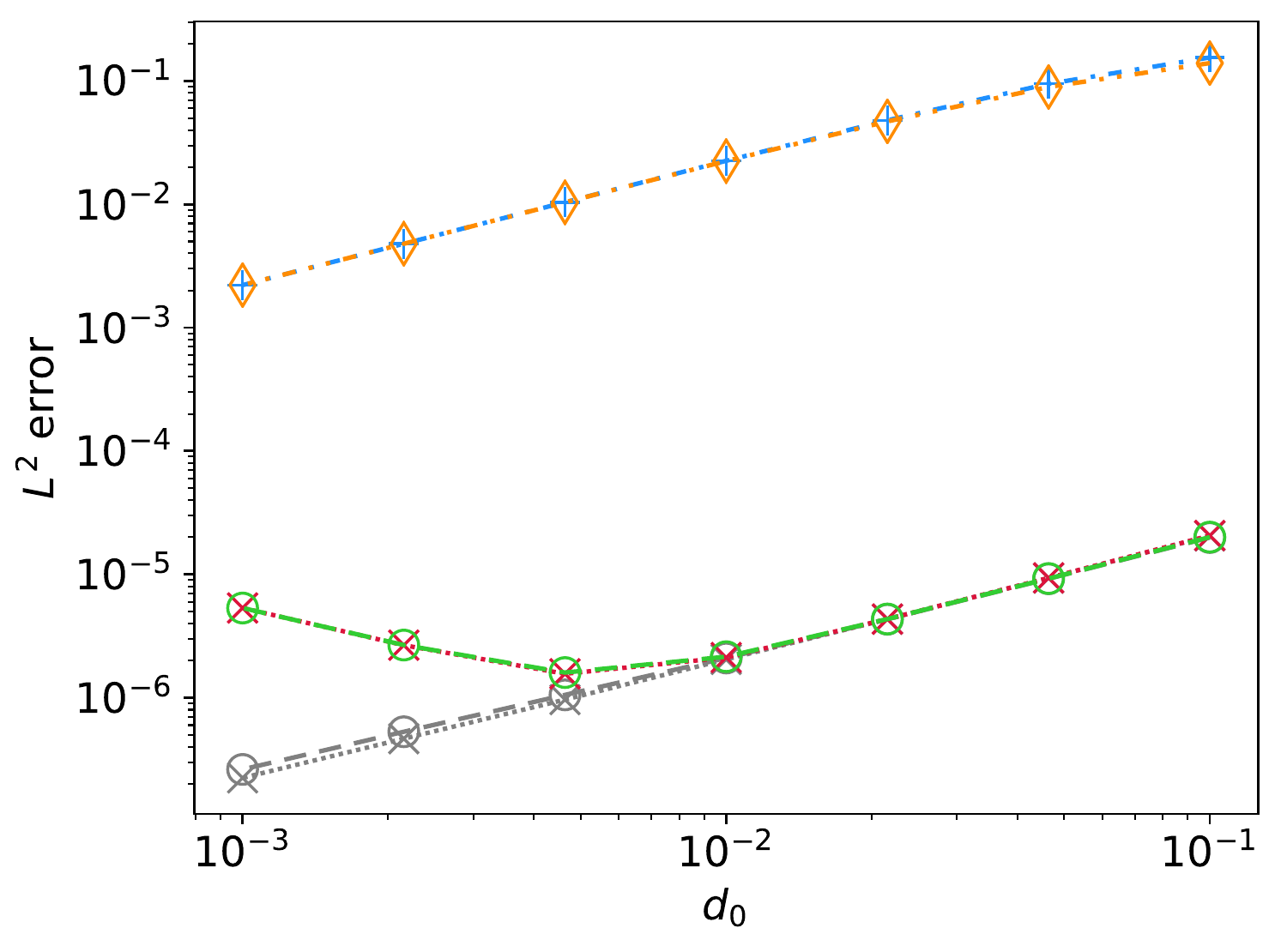}
\end{subfigure}
  \caption{Numerical solutions (\Cref{sec:perp_sym}) for the effective pressure~$p_\Gamma$ for the reduced model~\eqref{eq:strong} and its variants in comparison to the numerical reference solution~$\smash{p_\Gamma^\mathrm{ref}}$ for $d_0 = 10^{-1}$~(left) and $L^2$-error~$\Vert p_\Gamma - p_\Gamma^\mathrm{ref} \Vert_{L^2(\Gamma )}$ as function of~$d_0$~(right) with respect to the numerical reference solution~$\smash{p_\Gamma^\mathrm{ref}}$~(in color) and the presumably exact reference solution~$p_\Gamma^\mathrm{ref} \equiv \smash{\frac{1}{2}}$~(gray).}
  \label{fig:perp_sym_error}
\end{figure}

A comparison between the numerical solution~$p_\Gamma$ of the different reduced models and the averaged full-dimensional reference solution~$p_\Gamma^\mathrm{ref}$ inside the fracture can be found in~\Cref{fig:perp_sym_error}. 
Here, it occurs that the solutions of the models~I and I-R and the solutions of the models~II and II-R respectively show a very similar behavior. 
While the solutions of model~II and~II-R slowly display convergence towards the reference solution with declining aperture parameter~$d_0$, there is already a remarkable agreement between the solutions of model~I and~I-R and the reference solution.
As compared to the solutions of model~II and~II-R, the solutions of model~I and~I-R are more accurate by several orders of magnitude.
Thus, in contrast to the test problem in \Cref{sec:perp_asym_2d}, the artificial rectification of the bulk domains~$\Omega_1$ and~$\Omega_2$ is virtually without effect, while the inclusion of aperture gradients~$\nabla d_1$, $\nabla d_2$ in~\cref{eq:strongB} seems all the more significant in order to obtain accurate solutions. 
This is probably due to the symmetry of the problem and cannot be expected in general.

Further, it is noticeable in \Cref{fig:perp_sym_error} that, for the models~I and~I-R, the $L^2$-error with respect to the numerical reference solution~$p_\Gamma^\mathrm{ref}$ suffers a stop of convergence at small values of the aperture parameter~$d_0$.
Inspecting the numerical reference solution~$p_\Gamma^\mathrm{ref}$ for small apertures, one observes an unphysical oscillatory behavior, which might be associated with the integration of the full-dimensional reference solution according to \cref{eq:pGamma}.
In particular, these spurious oscillations display amplitudes in the range of~$10^{-5}$ to~$10^{-6}$ and hence can fully explain the total $L^2$-error and stop of convergence in \Cref{fig:perp_sym_error}.
Furthermore, the symmetry of the test problem in this section suggests that the effective pressure inside the fracture exactly equals~$p_\Gamma^\mathrm{ref} \equiv \frac{1}{2}$.
Thus, we can consider the $L^2$-error with respect to this presumably exact solution, which is also shown in \Cref{fig:perp_sym_error}.
Remarkably, in this case, one observes unimpeded convergence with the decline of the aperture parameter~$d_0$.
This confirms that we are dealing with a numerical error in the computation of the reference solution and not with a systematic model error.

\subsection{Tangential Flow through an Axisymmetric Fracture} \label{sec:tang_sym}
In this test problem, we consider an axisymmetric sinusoidal fracture as in~\Cref{sec:perp_sym} with the aperture functions~$d_1$ and~$d_2$ defined by~\cref{eq:d12_sym}.
Besides, we define the permeability inside the fracture by~$\matr{K}_\mathrm{f} = 2\matr{I}$ for the full-dimensional model~\eqref{eq:darcydecomposed}, which results in the effective permeabilities~$\matr{K}_\Gamma = 2 \matr{I}$ and $K_\Gamma^\perp = 2$ for the reduced model~\eqref{eq:strong} and its variants. 
In particular, the fracture permeability is larger than the bulk permeability. 
The pressure~$p$ at the boundary~$\partial \Omega$ is given by the function~$g(\vct{x}) = 4  x_1 ( 1 - x_1 )  ( 1 - x_2 )$.
This results in an inflow at bottom of the domain with the fracture as the preferential flow path. 
\Cref{fig:tang_sym} illustrates the fracture geometry and the resulting solution from the DG scheme~\eqref{eq:dgbulk} for~$d_0 = 10^{-1}$.
\begin{figure}[tbh]
\centering
\begin{subfigure}{0.495\textwidth}
\centering
\includegraphics[height=0.75\textwidth]{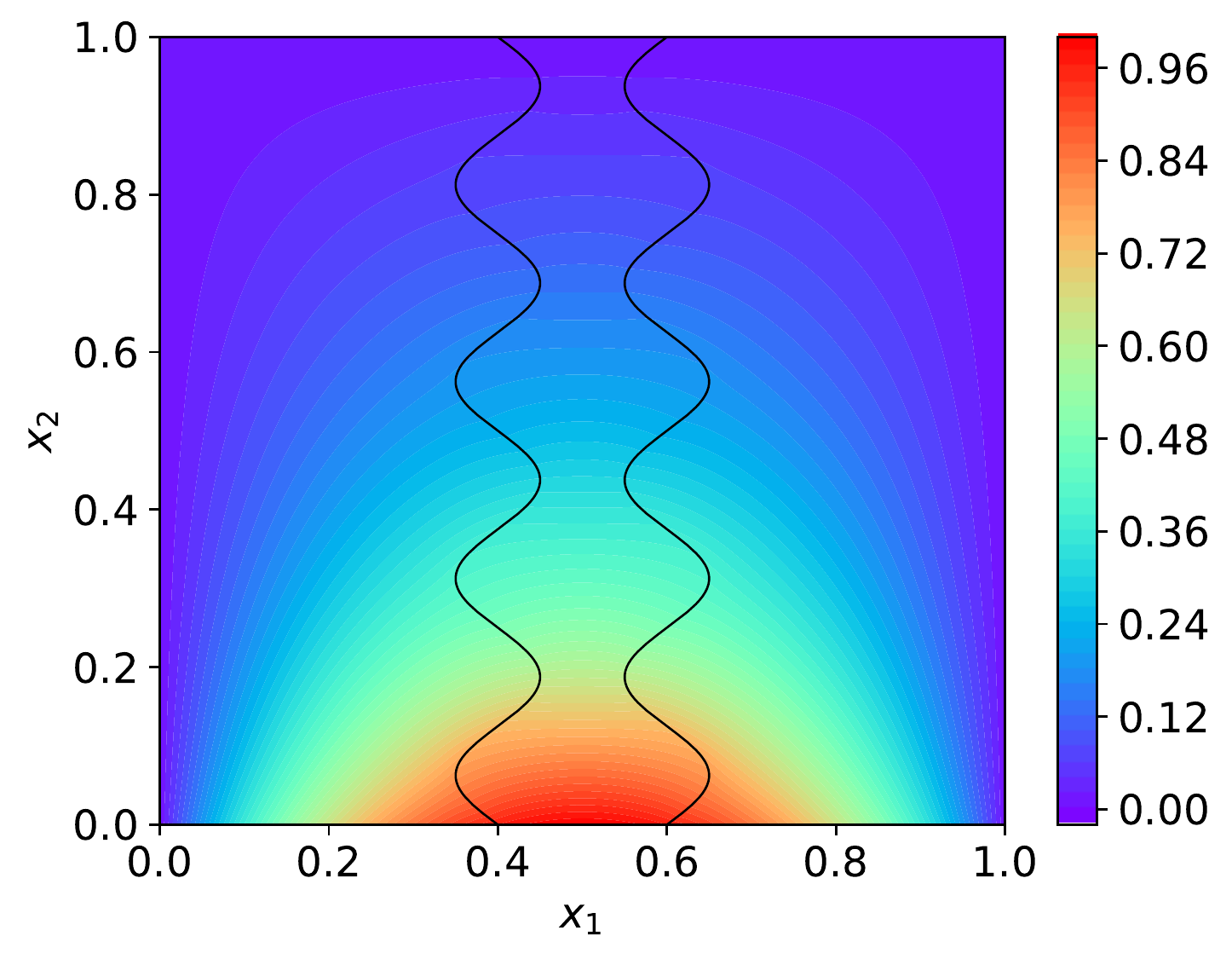}
\end{subfigure}
\hfill
\begin{subfigure}{0.495\textwidth}
\centering
\includegraphics[height=0.75\textwidth]{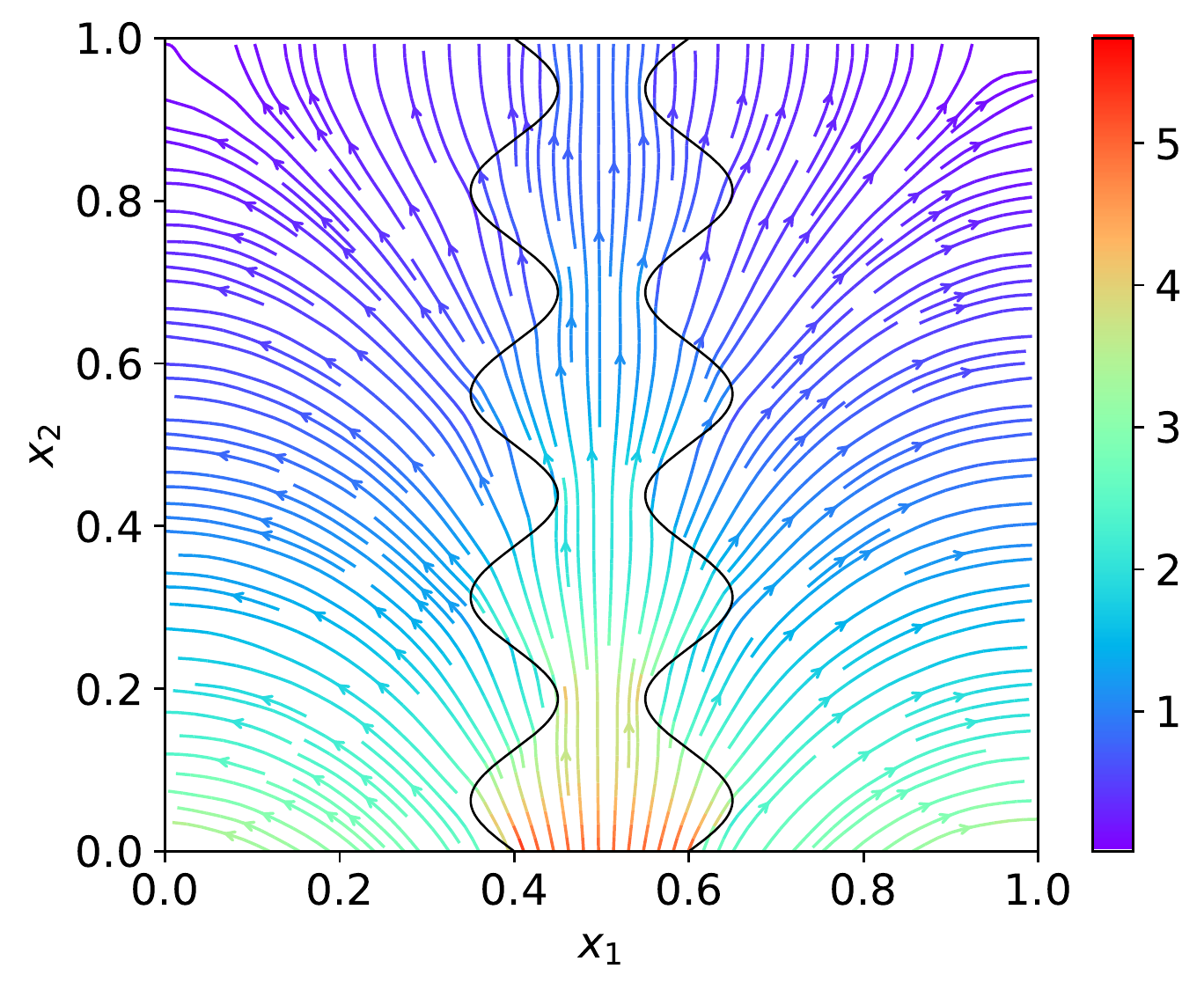}
\end{subfigure}
\caption{Full-dimensional numerical reference solution (\Cref{sec:tang_sym}) for the pressure~$p$ (left) and the velocity~$-\matr{K}\nabla p$ (right) for the case of~$d_0 = 10^{-1}$.}
\label{fig:tang_sym}
\end{figure}
\begin{figure}[tbh]
\centering
\includegraphics[width=0.8\linewidth]{media/legend.pdf}
\begin{subfigure}{.49\textwidth}
  \centering
  \includegraphics[width=\linewidth]{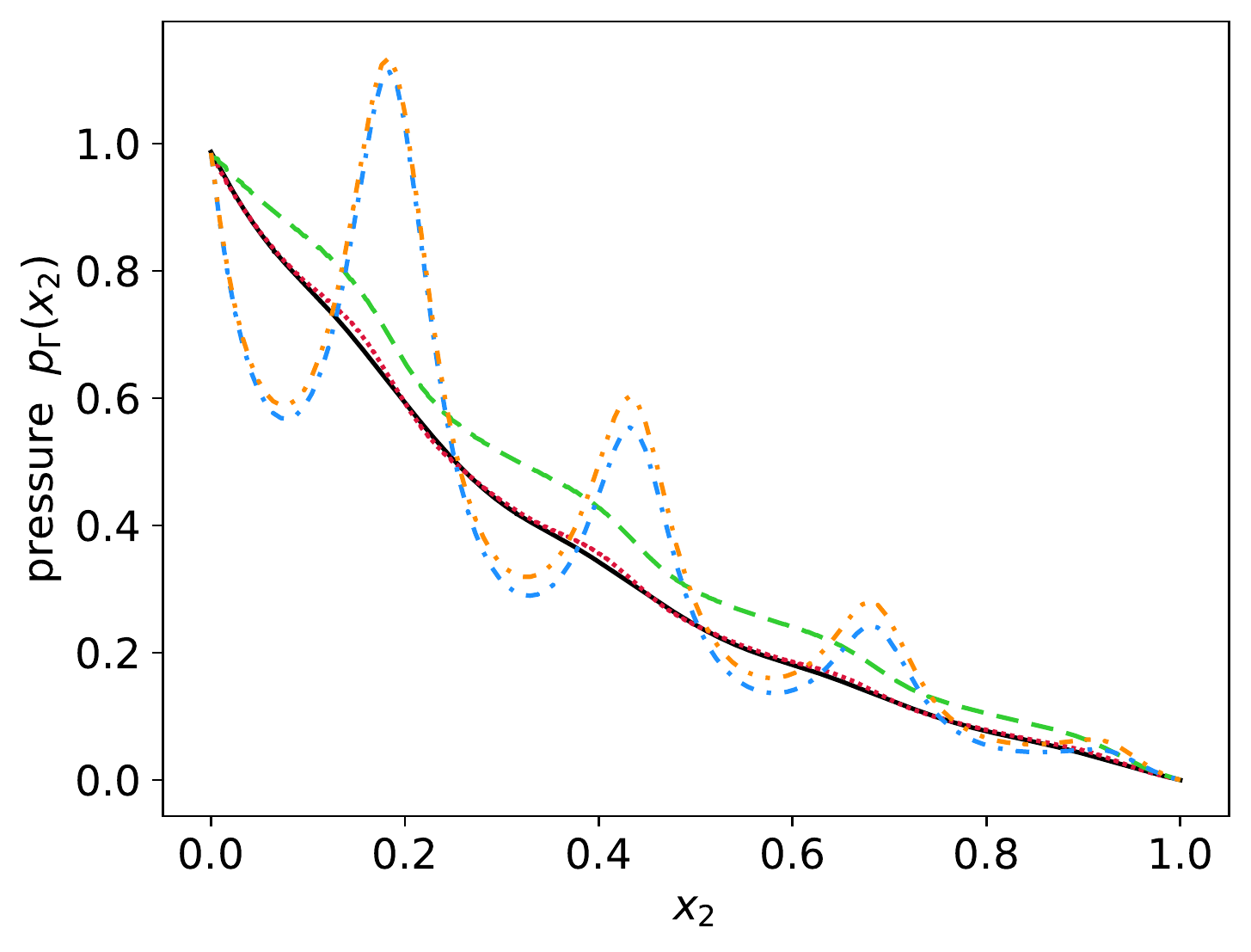}
\end{subfigure}\hfill%
\begin{subfigure}{.49\textwidth}
  \centering
  \includegraphics[width=\linewidth]{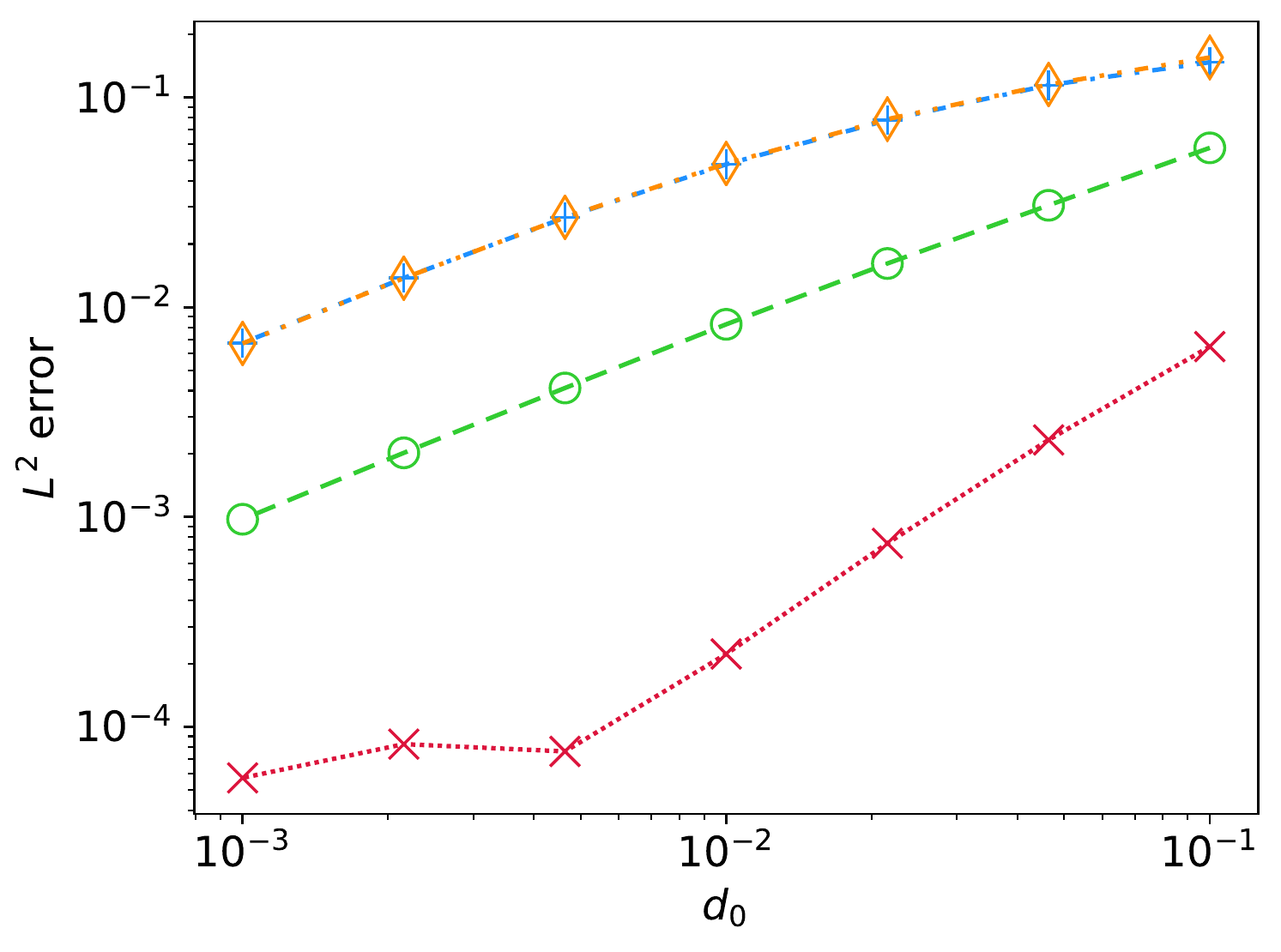}
\end{subfigure}
  \caption{Numerical solutions (\Cref{sec:tang_sym}) for the effective pressure~$p_\Gamma$ for the reduced model~\eqref{eq:strong} and its variants in comparison to the numerical reference solution~$\smash{p_\Gamma^\mathrm{ref}}$ for $d_0 = 10^{-1}$~(left) and $L^2$-error~$\Vert p_\Gamma - p_\Gamma^\mathrm{ref} \Vert_{L^2(\Gamma )}$ as function of~$d_0$~(right).}
  \label{fig:tang_sym_error}
\end{figure}

\Cref{fig:tang_sym_error} shows the DG solutions~$p_\Gamma$ of the different reduced models in comparison with the numerical reference solution~$p_\Gamma^\mathrm{ref}$. 
In particular, it can be seen that the models~II and~II-R display a similar behavior and are the least accurate, while model~I shows the best match with the reference solution. 
Further, in \Cref{fig:tang_sym_error}, the $L^2$-error~$\Vert p_\Gamma - p_\Gamma^\mathrm{ref} \Vert_{L^2(\Gamma )}$ displays a convergence with the decline of the aperture parameter~$d_0$ for all variants of the reduced model~\eqref{eq:strong}, where the solution of model~I converges faster than the solutions of the other models.
However, for model~I, the convergence stagnates at small apertures, which is associated with numerical errors in the computation of the reference solution~$p_\Gamma^\mathrm{ref}$ as discussed in \Cref{sec:perp_sym}.
Notably, for the test problem in this section, one finds by comparing the solution of model~I with the solutions of model~I-R and model~II in \Cref{fig:tang_sym_error} that both the artificial rectification of the bulk domains~$\Omega_1$ and~$\Omega_2$ and the negligence of aperture gradients in \cref{eq:strongB} significantly impair the accuracy of the solution.

\section{Conclusion}
In this work, we have derived a new model for single-phase flow in fractured porous media, where fractures are represented as lower-dimensional interfaces.
The model accounts for asymmetric fractures with spatially varying aperture and can be viewed as a generalization of a previous model by Martin~et~al.~\cite{martin05}.
The new model allows to study rough-surfaced, possibly curvilinear and winding real-world fracture geometries, while avoiding thin equi-dimensional fracture domains that require highly resolved grids in numerical methods.
In various numerical experiments, we have found a remarkable agreement between the solution of the new interface model and the reference profile.
Moreover, it has been observed that neglecting any of the terms in the model associated with a varying fracture aperture can substantially impair the accuracy of the solution.

As a future perspective, it is planned to apply the new model to real-world fracture geometries. 
Future extensions to related interface systems include two-phase flow systems~(cf.~\cite{burbulla22b}) and heterogeneous flow systems with different flow models for bulk and fracture domains, e.g., a free-flow regime inside the fracture.


\bibliographystyle{siamplain}
\bibliography{references}
\end{document}